\RequirePackage{fix-cm}
\documentclass[smallcondensed, final]{svjour3}          % twocolumn
\smartqed  % flush right qed marks, e.g. at end of proof
\usepackage{graphicx}
\usepackage{epstopdf}
\usepackage{setspace}
\usepackage{graphicx} 
\usepackage{amsfonts, amsmath, amssymb, latexsym} 
\usepackage{floatflt}
\usepackage{makeidx}

\journalname{arXiv}

\makeindex

\begin{document}

    \title{Proof of the Voronoi conjecture for 3-irreducible parallelotopes}
    \author{Andrei Ordine}
    %\dept{Department of Mathematics and Statistics}
    %\degree{Doctor of Philosophy}
    %\authorrunning{Andrei Ordine} % if too long for running head

    \institute{Andrei Ordine \at
              Canada Pension Plan Investment Board \\
              Tel.: +1 416 721 6270\\
              One Queen Street East, Toronto, ON, M5C 2W5
              \email{aordine@cppib.com}           %  \\
%             \emph{Present address:} of F. Author  %  if needed
}

    \date{2017 January 5}

    \maketitle

\begin{abstract}
    This article proves the Voronoi conjecture 
    on parallelotopes in the special case of 3-irreducible tilings. 

    Parallelotopes are convex polytopes
    which tile the Euclidean space by their translated copies,
    like in the honeycomb arrangement of hexagons in the plane.
    An important example of parallelotope is
    the Dirichlet-Voronoi domain for a translation lattice. 
    For each point $\lambda$ in a translation lattice, we define its 
    Dirichlet-Voronoi (DV) domain to be the set of points 
    in the space which are at least as close to $\lambda$ as to 
    any other lattice point.

    The Voronoi conjecture, formulated by the great Ukrainian 
    mathematician George Voronoi in 1908, states that any
    parallelotope is affinely equivalent
    to the DV-domain for some lattice. 
    
    Our work proves the Voronoi conjecture for $3$-irreducible
    parallelotope tilings of arbitrary dimension: we define the $3$-irreducibility 
    as the property of having only irreducible dual $3$-cells.
    This result generalizes a theorem of Zhitomirskii (1927), and
    suggests a way to solve the conjecture in the general case. 

\end{abstract}

 %   \section{Acknowledgments}
%      
%    I thank my supervisor, Prof. R.M.Erdahl, for all the
%    discussions and support, and Nikolai Dolbilin, my teacher
%    and supervisor at Moscow State University, who introduced
%    me to discrete geometry and gave me the taste for mathematics.    

%    The paper would have been impossible without the generous
%    help of David Gregory, who listened to my presentations
%   and made a lot of suggestions. Kostya Rybnikov provided 
%    great help by discussing the problem during his visit to Queen's in 2003.
%    Many thanks to Claude Tardif, Resa Naserasr and other 
%    participants of the Discrete Math Seminar at Queen's University
%    for the discussions and support.
    
%    The research was financially supported by Graduate
%    Awards (2001-2004) and Bauman Fellowship (2001-2002) 
%    of Queen's University, and by Ontario Graduate Scholarship    
%    (2004-2005). I am grateful to the Department
%    of Mathematics and Statistics for providing employment
%    during my graduate studies. 
% \end{section}
 
    \bibliographystyle{alpha}
    \pagestyle{plain}   

\newtheorem{calculation}{Calculation}
%\newtheorem{example}{Example}

% For marginal comments
\newcommand{\NB}[1]{\marginpar{\raggedright \tiny #1}}

%Common math expressions
\def\ker{\mathop{\rm ker}}
\def\lin{\mathop{\rm lin}}
\def\aff{\mathop{\rm aff}}
\def\conv{\mathop{\rm conv}}
\def\rank{\mathop{\rm rank}}
\def\combdim{\mathop{\rm combdim}}
\def\int{\mathop{\rm int}}
\def\relint{\mathop{\rm relint}}
\def\Vert{\mathop{\rm Vert}}
\def\St{\mathop{\rm St}}
\def\Cl{\mathop{\rm Cl}}
\def\Sk{\mathop{\rm Sk}}
\def\Sim{\mathop{\rm Sim}}
\def\sign{\mathop{\rm sgn}}
\def\sgn{\mathop{\rm sgn}}
\def\Im{\mathop{\rm Im}}
\def\cone{\mathop{\rm cone}}

%Note: \dim, \max, \min,  \sin  are part of LaTeX and needn't be defined.

%Multiplication sign with less space surrounding it
\def\tx{\! \times \! }

%Commands for \begin{proof} and \end{proof}
\def\proof{\par{\noindent \it Proof}. \ignorespaces}
\def\endproof{\hfill\vbox{\hrule height0.6pt\hbox{%
    \vrule height1.3ex width0.6pt\hskip0.8ex
    \vrule width0.6pt}\hrule height0.6pt}\medskip}

% Use \pmod for modulo

%\begin{abstract} We prove the Voronoi conjecture on parallelohedra for the case of
%tilings where the dual 3-cells are simplices, octahedra  or pyramids.
%\end{abstract}

\begin{section}{Introduction}
\label{introduction-section}

\begin{subsection}{Voronoi's conjecture on parallelotopes}

This paper investigates parallelotopes, 
convex polytopes which tile Euclidean 
space by parallel copies. For instance, 
a parallelotope in the plane is either
a parallelogram or a centrally symmetric hexagon.
\index{parallelotope}
\index{parallelohedron}
\index{Dirichlet-Voronoi domain}
\index{DV-domain}

An important example of a parallelotope is given 
by a Dirichlet-Voronoi domain. For each point $\lambda$
in a translation lattice, we define its 
Dirichlet-Voronoi (DV) domain to be the set of points 
in the space which are at least as close to $\lambda$ as to 
any other lattice point. A DV-domain is a centrally symmetric 
polytope. Since DV-domains obviously tile the space, 
and are all equal up to translation, they are parallelotopes.  
Figure \ref{dv-domain-figure} shows an example 
of a Dirichlet-Voronoi domain in the plane.
\index{Dirichlet-Voronoi domain}
\index{DV-domain}

\begin{figure}
\begin{center}
\resizebox{120pt}{!}{\includegraphics[clip=true,keepaspectratio=true]{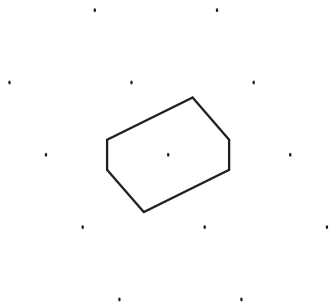}}
\end{center}
\label{dv-domain-figure}
\caption{A Dirichlet-Voronoi domain}
\end{figure}
George Voronoi, a Ukrainian mathematician, conjectured in 1908 (\cite{bib_voronoi_rech_par})
that any parallelotope is a DV-domain for some lattice, 
or is affinely equivalent to one.
This is a strong and fascinating conjecture since it gives 
an analytical description to a combinatorial object. 
Voronoi proved it for primitive tilings of $d$-dimensional 
space, that is tilings where exactly $d+1$ parallelotopes 
meet at each vertex.
\index{primitive tiling}

Since then, the conjecture has been established 
in a number of special cases.
Zhitomirskii 
generalized Voronoi's argument to tilings 
of $d$-dimensional space where $3$ parallelotopes 
meet at each $(d-2)$-face (1927, \cite{bib_zhitomirskii}),
Delaunay proved the conjecture for tilings
of $4$-dimensional space (1929, \cite{bib_delaunay_4d}).
Robert Erdahl proved it for space  filling zonotopes 
(1999, \cite{bib_erdahl_zonotopes}). Delaunay 
challenged Soviet mathematicians to resolve 
the conjecture in his afterword to the Russian translation of Voronoi's paper 
\cite{bib_voronoi_rech_par}. We tried to answer the 
challenge and obtained a proof of the conjecture
in a new special case which naturally generalizes 
Voronoi's and Zhitomirskii's work. The next section states 
our results. The results were first published online in our PhD
thesis in 2005, and this paper is the first presentation of these results
in an article format. Paper \cite{bib-ordine} 
by A.Ordine and A.Magazinov includes additional results
from the PhD thesis. Since then, A.Gavriluk generalized 
the Voronoi results to parallelohedra with simply connected
$\delta$-surface (\cite{bib-garber}), and A.Magazinov
proved that if the Voronoi conjecture holds for a parallelohedron $P$,
then it holds for the Minkowski sum $P+I$, where $I$ is a segment,
so long as $P+I$ is a parallelotope (\cite{bib-magazinov}).

Theoretical and practical applications of parallelotopes are numerous. 
Parallelotopes in dimension 3 are of great importance in crystallography. 
They were first classified by Evgraf Fedorov, the famous geologist
and crystallographer, in 1885 (\cite{bib_fedorov}).
Parallelotopes are also used in vector quantization (digitalization 
of vector data). For example, the  
parallelotope with 14 facets shown in figure 
\ref{3d-parallelotopes-figure} (rightmost) is used 
to encode the three-component signal vector 
in video monitors. In geometry of numbers, parallelotopes 
are applied to the study of arithmetic properties
of positive definite quadratic forms.

\begin{figure}
\begin{center}
\resizebox{240pt}{!}{\includegraphics[clip=true,keepaspectratio=true]{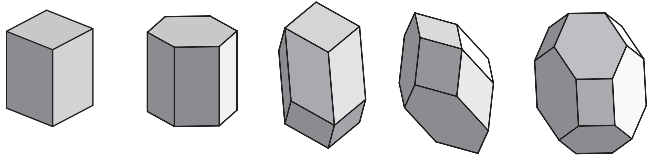}}
\end{center}
\caption{$3$-dimensional parallelotopes}
\label{3d-parallelotopes-figure}
\end{figure}
\end{subsection}

\begin{subsection}{Our results}
The idea behind our work is to investigate reducibility 
properties of parallelotopes.

A parallelotope is called {\em reducible} if it can be 
represented as a Minkowski sum of two parallelotopes
of smaller dimensions which belong to complementary affine 
spaces (a direct Minkowski sum). In dimension $3$, 
the reducible parallelotopes are the hexagonal
prism and the parallelepiped (see figure 
\ref{3d-parallelotopes-figure}). Note that the 
set ${\cal N}$ of normal vectors to facets of
a reducible parallelotope can be broken down
into nonempty subsets ${\cal N}_1$ and ${\cal N}_2$ 
so that the linear spaces $\lin({\cal N}_1)$
and $\lin({\cal N}_2)$ are complementary:
\begin{equation}
{\mathbb R}^d = \lin({\cal N}_1) \oplus \lin({\cal N}_2)
\end{equation}
\index{reducible parallelotope}

Consequently, if the Voronoi conjecture is true for irreducible 
parallelotopes, then it is true in all cases. But 
reducibility is a very restrictive property. 
Can we relax it?

We define the {\em $k$-reducibility} of a normal 
\footnote{In a {\em normal}, or {\em face-to-face} 
\index{normal tiling}
\index{face-to-face tiling}
tiling two parallelotopes intersect over a common
face or do not intersect at all. This definition
excludes tilings such as the brick wall pattern, where
rectangles may share only a part of an edge.}
parallelotope tiling of the Euclidean 
$d$-dimensional space as follows.
\index{k-reducible tiling}
\index{local reducibility}

Let $F^{d-k}$ be a face of the tiling, $k>1$. Consider
the set ${\cal N}_{F^{d-k}}$ of normal vectors
to facets of the tiling which contain $F^{d-k}$
in their boundaries. If 
${\cal N}_{F^{d-k}}$ can be broken down into
nonempty subsets ${\cal N}_{F^{d-k}}^1$ and 
${\cal N}_{F^{d-k}}^2$ so that 
\begin{equation}
\lin({\cal N}_{F^{d-k}})
 = \lin({\cal N}_{F^{d-k}}^1) \oplus 
\lin({\cal N}_{F^{d-k}}^2),
\end{equation}
then the tiling is called {\em locally reducible}
at face $F^{d-k}$. If the tiling is locally reducible 
at one or more faces $F^{d-k}$, then it is called 
{\em $k$-reducible}.

A normal tiling which 
is not $k$-reducible is called $k$-irreducible.
If a tiling is $k$-reducible, 
then it is $m$-reducible for all $1 < m < k$.

The Voronoi conjecture for $2$-irreducible tilings was 
proved by Zhitomirskii in 1927 (\cite{bib_zhitomirskii}).
Our main result is the following theorem:

\begin{theorem} \label{main_theorem}
The Voronoi conjecture is true for $3$-irreducible 
parallelotope tilings.
\end{theorem}

This is a natural generalization of Zhitomirskii's 
result. Before this work started, Sergei Ryshkov 
and Konstantin Rybnikov considered the Voronoi 
conjecture for $3$-irreducible tilings. They tried
to apply the methods of algebraic topology
\cite{bib-rybnikov-private}.

We can consider $k$-irreducible tilings,
$k=4,5,\dots,d$ and try to prove the Voronoi
conjecture for them. We speculate it can be a way
to attack the Voronoi conjecture in the general case. 
The conclusion of the paper has more on this topic.
\end{subsection}

\begin{subsection}{Plan of the paper}
In the next section, we present 
the existing results in the parallelotope
theory. 
The rest of the paper is devoted to proving 
the result of this work, theorem \ref{main_theorem}.
Knowledge of polytope theory is expected 
on the part of the reader; a good reference is McMullen 
and Shephard's book \cite{bib_mcmullen_convex_polytopes}. 
We also use some results and terminology 
of the theory of complexes; paper 
\cite{bib-alexander-proof-and-extension}
by Alexander gives enough information.

Section \ref{candesc-section}
introduces the canonical scaling and gives 
a tool for building it. The existence 
of canonical scaling is equivalent to 
the Voronoi conjecture. section 
\ref{dual-cell-combinatorics-section} defines
dual cells, objects which we use to describe
stars of faces in the tiling. In section 
\ref{topological-results-section} the results
of the two previous sections are applied
to identify and study 
incoherent parallelogram dual cells,
objects which shouldn't exist for the Voronoi
conjecture to hold.
section \ref{affine-results-section} investigates
polytope-theoretic properties of dual cells. It 
is independent from section \ref{topological-results-section}.
Finally, the main result (theorem \ref{main_theorem}) 
is proved in section \ref{proof-section}, where
all results obtained are applied to showing
that no incoherent parallelogram cells
can exist in a $3$-irreducible tiling.

We finish with a discussion which summarizes our ideas and suggests ways to solve
the Voronoi conjecture in the general case. 
\end{subsection}
\end{section}

\begin{section}{Overview of classical results}
\label{classical-results-section}

\begin{subsection}{The Minkowski and Venkov theorems}
We now present the classical results in parallelotope theory. In 1897,
Minkowski published his famous theorem on polytopes (\cite{bib_minkowski_polytopes}).

\begin{theorem} 
\label{minkowski_polytope_theorem}
Let $d\ge 2$. Suppose that ${\bf n}_1, {\bf n}_2, \dots, {\bf n}_k 
\in {\mathbb R}^d$ are unit vectors that span ${\mathbb R}^d$, 
and suppose that  $\alpha_1, \dots, \alpha_k > 0$. 
Then there exists a $d$-dimensional polytope $P$ 
having external facet normals ${\bf n}_1,{\bf n}_2,\dots,{\bf n}_k$ 
and corresponding facet $d-1$-volumes $\alpha_1, \dots, \alpha_k$ 
if and only if
\begin{equation}
{\bf n}_1 \alpha_1 + \dots + {\bf n}_k \alpha_k = 0.
\end{equation}
Moreover, such a polytope is unique up to a translation.
\end{theorem}
\index{Minkowski theorem on polytopes}

The theorem has an important implication for parallelotopes.

\begin{corollary}(Minkowski, \cite{bib_minkowski_1905}, 1905)
\begin{enumerate}
\item A parallelotope is centrally symmetric
\footnote{A set $A$ is called {\em centrally symmetric}
if there is a point $a$ such that the mapping
$*: x\to 2a - x$ maps $A$ onto itself. Point $a$
is called the {\em center of symmetry}.}
\index{central symmetry}
\item Facets of a parallelotope are centrally symmetric
\item The projection of a parallelotope along a face of $(d-2)$ 
dimensions onto a complementary 2-space is either a parallelogram, 
or a centrally symmetric hexagon.
\footnote{Properties
(1) and (2) are not independent: if $d\ge 3$, then the central 
symmetry of a polytope follows from the central symmetry of its 
facets (ie. $(d-1)$-dimensional faces), see 
\cite{bib_venkov_conditions_dependency}.} 
\end{enumerate}
\end{corollary}
\medskip
\begin{figure}
\begin{center}
\resizebox{100pt}{!}{\includegraphics[clip=true,keepaspectratio=true]
{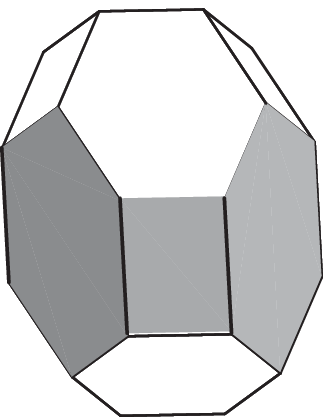}}
\resizebox{200pt}{!}{\includegraphics[clip=true,keepaspectratio=true]
{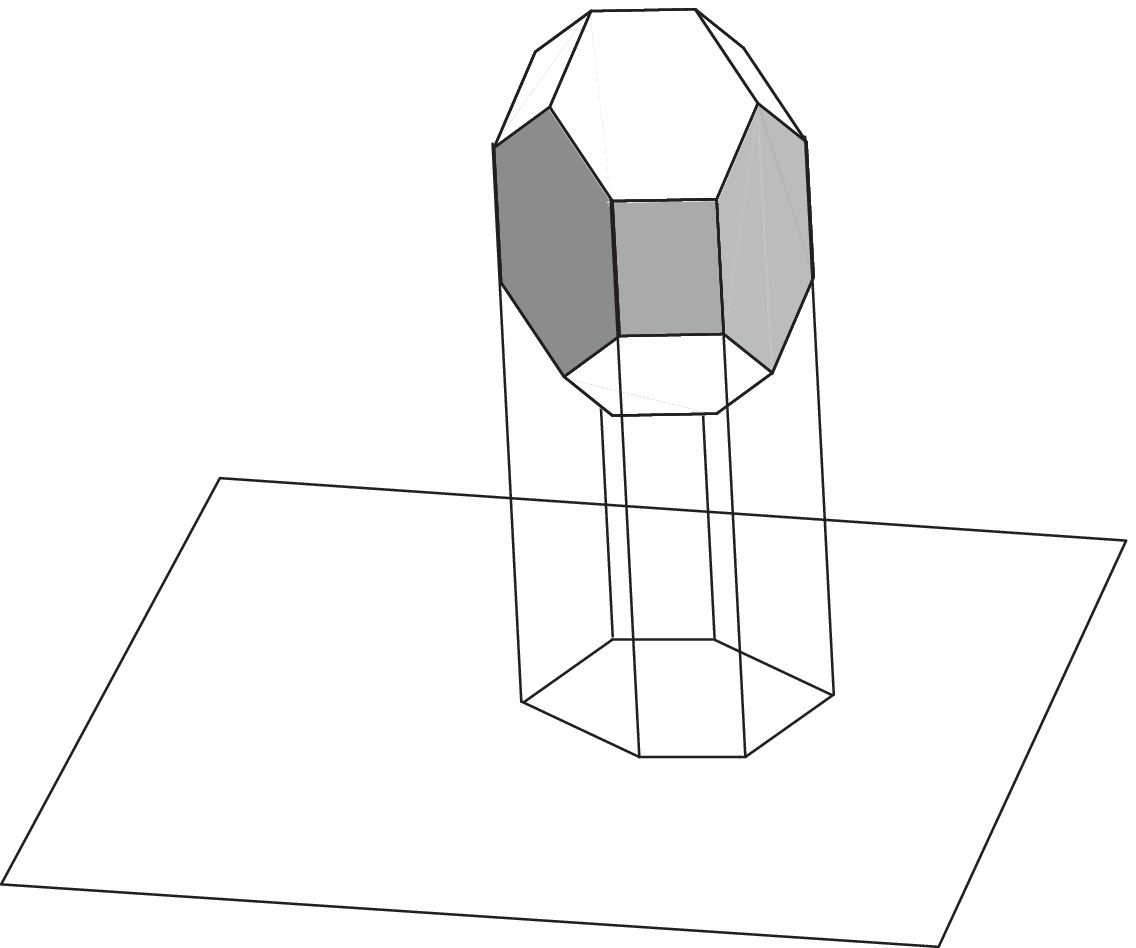}}
\end{center}
\caption{A belt and the projection along the corresponding $(d-2)$-face, 
$d = 3$}
\label{belt-figure}
\end{figure}

\begin{theorem}(Venkov, \cite{bib_venkov_criterion} 1954)
\label{venkov_criterion_theorem}
Conditions 1-3 are sufficient for a polytope to be 
a parallelotope admitting a normal (face-to-face) tiling.
\end{theorem}

A result by Nikolai Dolbilin (\cite{bib_dolbilin_extension_theorem}) gives 
an easy way to prove this theorem. 

\begin{corollary}An arbitrary parallelotope admits a normal tiling of 
the space.
\end{corollary}

Indeed, a parallelotope satisfies conditions 1-3 
by Minkowski's corollary, so by Venkov's theorem 
it admits a normal tiling of the space by parallel copies.
In the rest of the paper, we will only consider normal tilings.

\label{belt-definition-page}
\index{Venkov conditions}
\index{belt}
Conditions 1-3 are known as {\em Venkov conditions}. 
The second Venkov condition implies that facets 
of the parallelotopes are organized into {\em belts}: 
to define a belt, take a facet and a $(d-2)$ face 
on its boundary. Since the facet is centrally symmetric,
there is a second $(d-2)$-face parallel to the first one, 
its symmetric copy. Now,
the second $(d-2)$-face belongs to another facet 
of the parallelotope which is centrally symmetric as well. 
Its central symmetry produces a third $(d-2)$-face.
Proceeding this way, we will come back to the facet 
we have started from. The facets that we have visited 
form a {\em belt} (see figure \ref{belt-figure}). 
When we project the parallelotope along
the $(d-2)$-face, we get a parallelogram 
or a hexagon, by Venkov condition 3. 
The edges of this figure are the images of the facets 
in the belt; therefore, the belt contains $4$ or $6$ facets. 

The belt and the $(d-2)$-face are called {\em quadruple} if the projection
is a parallelogram, or {\em hexagonal} if the projection is a hexagon.
\end{subsection}

\begin{subsection}{Stars of $(d-2)$-faces}
\label{d-2-stars-section}
We now discuss one important corollary of the 
Venkov conditions. The {\em star} $\St(F)$
of a face $F$ of the tiling is the collection
of all faces of the tiling which contain $F$.

\begin{lemma} (\cite{bib_delaunay_4d})
The star of a hexagonal $(d-2)$-face 
contains $3$ facets, where no $2$ facets are parallel.
The star of a quadruple $(d-2)$-face contains $4$ 
facets. There are $2$ pairs of parallel facets,
but facets from different pairs are not parallel. 
\end{lemma}

The proof can be found in Delaunay's paper
\cite{bib_delaunay_4d}. We will present
the idea briefly. 

Fix a parallelotope $P_0$ in the tiling and 
consider the collection ${\cal L}$ of parallelotopes 
which can be reached from $P_0$ by crossing facets
parallel to $F^{d-2}$. Delaunay 
called ${\cal L}$ a ``couche'', or layer.

Let $h$ be the projection along $F^{d-2}$ 
onto a complementary $2$-space $L^2$. The images
$h(P)$ of parallelotopes $P\in {\cal L}$ form
a hexagonal or parallelogram tiling of $L^2$,
depending on whether  $F^{d-2}$ is hexagonal 
or quadruple. The projection gives a 1-1
correspondence between parallelotopes in ${\cal L}$
and the tiles in $L^2$. This allows for
the description of the star of $F^{d-2}$,
on the basis of the star of a vertex $h(F^{d-2})$
in the tiling of $L^2$.
\end{subsection}

\begin{subsection}{Voronoi's work}
\label{voronoi_work_section}
Voronoi in paper \cite{bib_voronoi_rech_par} 
proved that a primitive parallelotope 
is affinely equivalent to a DV-domain 
for some lattice (a primitive $d$-dimensional 
parallelotope produces a tiling where each 
vertex belongs to exactly $d+1$ parallelotopes). 
Since our work is based on Voronoi's ideas, 
we explain them now for a simple yet illustrative 
example, a centrally symmetric hexagon.
\begin{figure}
\begin{center}
\resizebox{120pt}{!}{\includegraphics[clip=true,keepaspectratio=true]{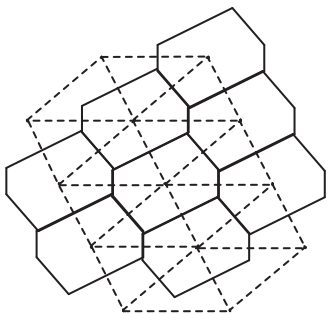}}
\end{center}
\caption{A planar DV-tiling with the reciprocal}
\label{dv-tiling-figure}
\end{figure}

(A) Suppose that we have a DV-tiling
of the plane by equal hexagons. The tiling 
has the characteristic 
property that the line segments connecting 
centers of adjacent parallelotopes are orthogonal
to their common edge, like in figure \ref{dv-tiling-figure}.
The dashed lines form a {\em reciprocal}, a rectilinear
dual graph to the tiling.
\index{reciprocal}

With each edge $F$ of the tiling, we can associate 
a positive number $s(F)$ equal to the length 
of the corresponding reciprocal edge. The numbers
have the property that for each triplet of edges $F$, $G$ and $H$
incidental to the same vertex, we have
\begin{equation}
\label{hexagonal-canonical-condition}
s(F) {\bf n}_F + s(G) {\bf n}_G + s(H) {\bf n}_H = 0
\end{equation}
where ${\bf n}_F$, ${\bf n}_G$ and ${\bf n}_H$ are the 
unit normal vectors to the facets, with directions
chosen appropriately (clockwise or counterclockwise). 
This equation simply 
represents the fact that the sum of edge vectors
of the reciprocal triangle is equal to $0$.

\medskip

(B) Now, consider a tiling by an arbitrary centrally
symmetric hexagon. We would like to find an 
affine transformation which maps the tiling
onto a tiling by DV-domains, as in figure
\ref{affine-transform-figure}.
\begin{figure}
\begin{center}
\resizebox{200pt}{!}{\includegraphics[clip=true,keepaspectratio=true]{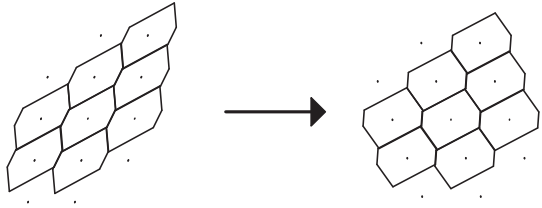}}
\end{center}
\caption{Mapping a hexagonal tiling onto a DV-tiling}
\label{affine-transform-figure}
\end{figure}
\label{hexagonal-canonical-scaling}
We can still assign a positive number $s(F)$ to each edge $F$
of the tiling, so that condition \ref{hexagonal-canonical-condition} 
holds for each triplet of edges incident to the same
vertex. 

Take an arbitrary vertex of the tiling
and let $e_F$, $e_G$ and $e_H$ be the edge
vectors of the three edges (pointing away from the vertex).
We let ${\bf n}_F = R_{\pi/2} \frac{e_F}{|e_F|}$,
${\bf n}_G = R_{\pi/2} \frac{e_G}{|e_G|}$,
${\bf n}_H = R_{\pi/2} \frac{e_H}{|e_H|}$
where $R_{\pi/2}$ is the rotation by angle $\pi/2$.
Vectors ${\bf n}_F$, ${\bf n}_G$ and ${\bf n}_H$
are normal to facets $F$, $G$ and $H$. It it easy to see that 
there are unique (up to a common multiplier) 
positive numbers $s(F)$, $s(G)$ and $s(H)$
so that equation \ref{hexagonal-canonical-condition} holds.
Then assign $s(F)$ to all edges parallel to $F$, 
$s(G)$ to all edges parallel to $G$, $s(H)$ to all edges 
parallel to $H$. This construction will be
called a {\em canonical scaling} of the tiling.
The numbers $s(F)$ are called {\em scale factors}.
The choice of terminology is due to the fact
that we are scaling normal vectors to edges
of the tiling.
\index{canonical scaling}
\index{scale factor}

A canonical scaling allows us to find the desired
affine transformation, using a wonderful method
invented by Voronoi.  His method involves
the construction of a convex surface in $3$-space, 
made of planar hexagons, which projects to 
a hexagonal tiling of the $(x_1,x_2)$-plane. 
This surface can be considered as a piecewise-planar 
function $z = G(x_1,x_2)$ (see figure
\ref{generatrissa-illustration}).

We will define $G$ in terms of its gradient vector, 
which will be constant on each of the hexagons. 
The rule of gradient assignment is as follows. 
We set the gradient to $0$ on some fixed hexagon
and let a point $x$ travel on the plane, avoiding
vertices of the tiling.
{\em When $x$ crosses the edge $F$ between 
two adjacent hexagons, the gradient of the function 
$G$ changes by $s(F) {\bf n}_F$, where ${\bf n}_F$ 
is the unit normal vector to the edge directed from the first
hexagon to the second one, and $s(F)$ is the scale factor.}

This definition of the gradient is consistent, that is,
different paths to the same hexagon result 
in the same gradient. To prove the consistency,
we need to check that the 
sum of increments $s(F) {\bf n}_F$ along a closed
circuit of hexagons is zero. This follows
from equation (\ref{hexagonal-canonical-condition}).

\index{generatrissa}
The function $G$ was called the {\em generatrissa} 
by Voronoi. A drawing of its graph 
is shown in figure \ref{generatrissa-illustration}.
\begin{figure}
\begin{center}
\resizebox{180pt}{!}{\includegraphics[clip=true,keepaspectratio=true]
{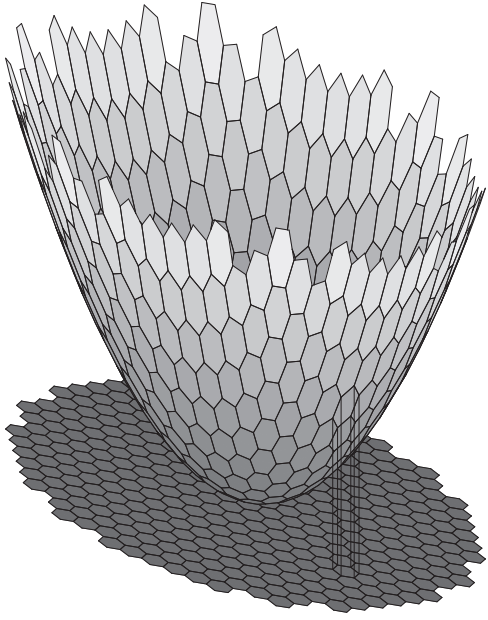}}
\caption{The generatrissa}
\label{generatrissa-illustration}
\end{center}
\end{figure} 
It looks like a paraboloid, a graph of a positive 
definite quadratic form. This is not a coincidence. 
We can actually calculate the quadratic form 
whose graph is inscribed into the graph of $G$. 

\index{facet vector}
Let $\lambda_1$, $\lambda_2$ be two non-collinear 
{\em facet vectors}, translation vectors 
which shift the hexagon onto its neighbors. 
Each center of a hexagon in the tiling can be 
represented as a sum of vectors $\lambda_1$,
$\lambda_2$ with integer coefficients. What is 
the value $G(x)$ of the generatrissa
at the point $x = k_1 \lambda_1 + k_2 \lambda_2$? 
To calculate $G(x)$, we traverse
the path between $0$ and $x$ in the following way:
\medskip
\begin{center}
\resizebox{180pt}{!}{\includegraphics[clip=true,keepaspectratio=true]
{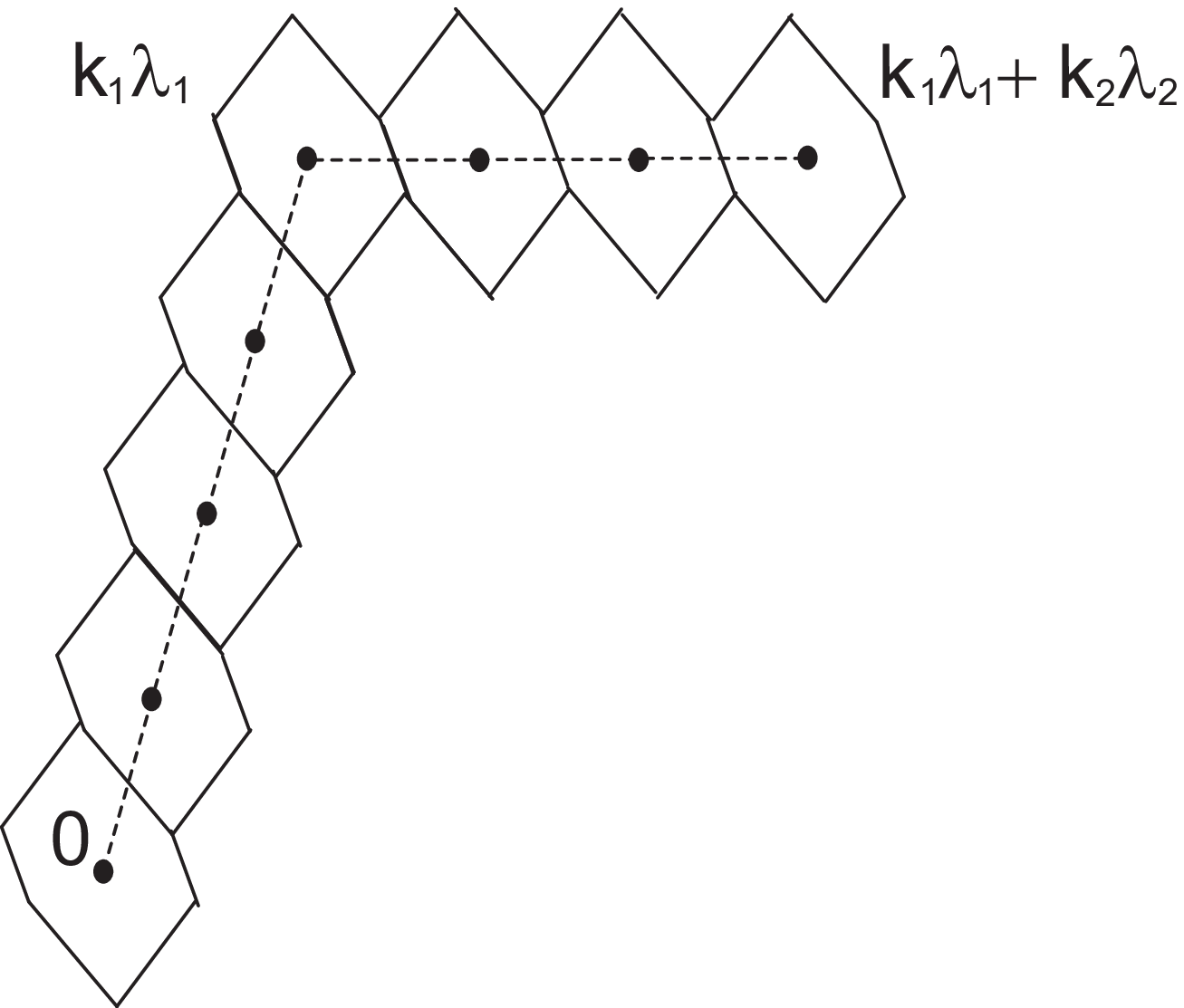}}
\end{center}

We have $G(0) = 0$. To calculate the difference 
$G(k_1 \lambda_1) - G(0)$, we integrate the gradient 
of the generatrissa $G$ along the line segment 
$[0,k_1 \lambda_1]$. Whenever we switch to the 
adjacent hexagon, the gradient increases 
by ${\bf n}_1$. The result is
\begin{equation}
G(k_1 \lambda_1) - G(0) = 
(\lambda_1 \cdot {\bf n}_1)(1 + 2 + \dots + (k_1 - 1) + \frac{k_1}{2}) = 
(\lambda_1 \cdot {\bf n}_1)\frac{k_1^2}{2}.
\end{equation}
The gradient of the generatrissa at the point 
$k_1 \lambda_1$ is equal to $k_1 {\bf n}_1$. 
Next, we calculate the difference between 
$G(k_1 \lambda_1 + k_2 \lambda_2)$ and $G(k_1 \lambda_1)$. 
Again, we integrate the gradient of $G$ along 
the line segment. We start with gradient
$k_1 {\bf n}_1$ at point $k_1 \lambda_1$ and note 
that the gradient increases by ${\bf n}_2$ 
each time we cross the boundary between 
hexagons. Therefore
\begin{equation}
\begin{split}
G(k_1 \lambda_1 + k_2 \lambda_2) - G(k_1 \lambda_1) &= \\
(\lambda_2 \cdot {\bf n}_2)(1 + 2 + \dots + (k_2 - 1) + \frac{k_2}{2}) + 
(k_2 \lambda_2) \cdot (k_1 {\bf n}_1) &= \\
(\lambda_2 \cdot {\bf n}_2)\frac{k_2^2}{2} + (\lambda_2 \cdot {\bf n}_1) k_1 k_2.
\end{split}
\end{equation}

Summing the two differences, we have
\begin{equation}
G(k_1 \lambda_1 + k_2 \lambda_2) =
G(x) = (\lambda_1 \cdot {\bf n}_1)\frac{k_1^2}{2} + 
(\lambda_2 \cdot {\bf n}_2)\frac{k_2^2}{2} + (\lambda_2 \cdot {\bf n}_1) k_1 k_2.
\end{equation}

We use the symbol $Q(x)$ for the quadratic form 
$(\lambda_1 \cdot {\bf n}_1)\frac{y_1^2}{2} + 
(\lambda_2 \cdot {\bf n}_2)\frac{y_2^2}{2} + 
(\lambda_2 \cdot {\bf n}_1) y_1 y_2$, where
$x = y_1 \lambda_1 + y_2 \lambda_2$. 

The two functions $G(x)$ and $Q(x)$ coincide at the centers of hexagons. 
Moreover, one can check that, at these points, the gradients of the functions
are equal, which shows that each lifted hexagon is tangent to the graph 
of the function $Q(x)$. This explains why the generatrissa looks like 
the graph of a quadratic form. 

Next, we prove that the generatrissa is a convex function. Take any two 
points in the plane of the tiling and connect them by a line segment. 
By a small shift of the points, we can assure that the line segment 
does not pass through the vertices of the tiling. The function $G$ limited 
to the line segment is piecewise linear. Its slope is equal
to the scalar product of the direction of the line segment to the gradient 
of the generatrissa. Each time the line segment crosses the edge, the 
gradient of the generatrissa changes by a normal vector to the edge being 
crossed, which points into the next hexagon. Therefore
the slope of $G$ increases along the line segment. This proves that 
$G$ is convex.
\begin{figure}
\begin{center}
\resizebox{120pt}{!}{\includegraphics[clip=true,keepaspectratio=true]
{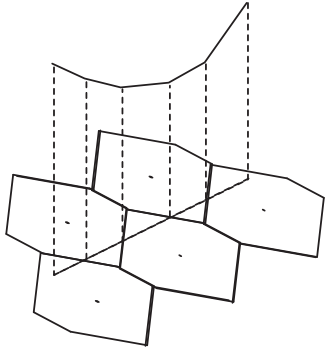}}
\end{center}
\caption{To the proof of the convexity of generatrissa}
\end{figure}
The quadratic form $Q$ is positive definite since it is bounded by the 
function $G$ from below and $G$ is positive everywhere outsize the 
zero hexagon (since the six hexagons surrounding the zero hexagon have 
been lifted upward).

Making an affine transformation if necessary, we can assume that 
$Q(x) = x_1^2 + x_2^2$. The generatrissa is then formed by tangent planes 
at points $(\lambda,Q(\lambda))$ where $\lambda$ is the center of some
hexagon of the tiling.

Let $\lambda_1$, $\lambda_2$ be the centers of two adjacent hexagons. 
It is a fact from elementary geometry that the intersection of the 
planes tangent to the graph of the positive definite quadratic form $Q$
at the points $(\lambda_1,Q(\lambda_1))$ and $(\lambda_2,Q(\lambda_2))$ 
projects to the bisector line of $\lambda_1$ and $\lambda_2$. Therefore 
each hexagon in the tiling is bounded by bisectors to the line segments
which connect its center to the centers of adjacent parallelotopes.
This means that the hexagon is the DV-domain for the lattice of centers.

Of course, we could find a much simpler way to map an arbitrary
hexagonal tiling onto a DV-tiling. However, the method just
described  works almost verbatim 
for a parallelotope tiling of arbitrary 
dimension, so long as the tiling can be equipped 
with a canonical scaling. In the next section, 
we describe a tool for obtaining it.
\end{subsection}

\begin{subsection}{Space filling zonotopes and zone contraction}
Reciprocal construction is one direction of research in the Voronoi conjecture.
There is another direction which considers {\em zonotopes} 
and properties of parallelotopes described by {\em zones}.
A zonotope (zonohedron) is a Minkowski sum of a finite collection 
of line segments. Below we review some results on zonotopes that fill 
the space by parallel copies.
\index{zonotope}
\index{zone}

All parallelotopes of dimensions $2$ and $3$ are zonotopes, but this 
is not true for dimension $4$ and higher. However, zonotopes continue to play 
a prominent role in the $4$-dimensional parallelotope family; Delaunay in 
1929 proved that up to affine equivalence, a parallelotope 
of dimension $4$ is either the $24$-cell, a Minkowski sum of the $24$-cell 
and a zonotope, or a space filling zonotope (\cite{bib_delaunay_4d}).
The $24$-cell is the DV-domain for the lattice $\lbrace z \in {\mathbb Z}^4 :
z_1+\dots+z_4\equiv 0 \pmod 2 \rbrace$ in $4$-dimensional space.
Erdahl in 1999 proved the Voronoi conjecture for space filling 
zonotopes (\cite{bib_erdahl_zonotopes}).

Given a parallelotope, a {\em zone} is the set of edges parallel 
to a given vector. A zone is called {\em closed} if each $2$-dimensional 
face of the parallelotope either contains two edges from the zone, or none. 
If a parallelotope has a closed zone, it can
be transformed so that the lengths of edges in the zone decrease by the same
amount. This operation is called {\em zone contraction}; the reverse operation
is called {\em zone extension}. 
\index{zone contraction}
\index{zone extension}

Engel in \cite{bib_eng_00} discovered that there are at least $103769$ 
combinatorial types of
$5$-dimensional parallelotopes. This is an enormous jump from just $52$
combinatorial types\footnote{Two polytopes are of the same combinatorial type if and only
if their face lattices are isomorphic. See book 
\cite{bib_mcmullen_convex_polytopes} for details.} of parallelohedra 
of dimension $4$. He uses ``contraction 
types'' of parallelotopes to refine the combinatorial type classification. 
Contraction type is a complex attribute of a parallelotope, 
defined using zone extension and zone contraction operations. 

Finally we note that a direct, ``brute force'' computer enumeration of
parallelotopes in a given dimension is hard since the combinatorial complexity
of parallelotopes grows fast: the maximum number of facets of a parallelotope in 
dimension $d$ is $2(2^d - 1)$. 
\end{subsection}
\end{section}

\begin{section}{Canonical scaling and the quality translation theorem}
\label{dual-cells-and-complexes-section}
\label{reciprocals-and-complexes-section}
\label{candesc-section}
We begin to prove our main result, the Voronoi 
conjecture for the case of $3$-irreducible tilings.
In this section we define the canonical scaling
of a tiling, whose existence is equivalent to the 
statement of the Voronoi conjecture, and present the
quality translation theorem, a tool for building the canonical scaling.
As an illustration, we prove the Voronoi 
conjecture for the case of primitive
tilings.

\begin{subsection}{Canonical scaling}
We have encountered the canonical scaling 
of a hexagonal tiling of the plane in section 
\ref{voronoi_work_section}. It allowed us
to construct an affine transformation which mapped 
the hexagon onto a DV-domain, thus proving the Voronoi 
conjecture. It turns out that the same idea 
works for a general normal parallelotope tiling. 
The definition below depends on the discussion 
of stars of $(d-2)$-faces on page
\pageref{d-2-stars-section}.

\begin{definition}
\label{candesc-definition}
Let $S$ be an arbitrary collection of facets in the tiling.
A {\em canonical scaling} of $S$ is an assignment
of a positive number $s(F)$ (called a {\em scale
factor}) to each facet $F\in S$ such that 
\begin{enumerate}
\item If $F_1, F_2, F_3\in S$ are the facets in the star of
a hexagonal $(d-2)$-face, then for some choice of
unit normals ${\bf n}_{F_i}$ to the facets
\begin{equation}
\label{hexagonal-candesc-condition}
s(F_1) {\bf n}_{F_1} + s(F_2) {\bf n}_{F_2} 
+ s(F_3) {\bf n}_{F_3}  = 0
\end{equation}
\item If $F_1,F_2,F_3,F_4\in S$ are the facets in the star
of a quadruple $(d-2)$-face, then for some choice of
unit normals ${\bf n}_{F_i}$ to the facets
\begin{equation}
\label{quadruple-candesc-condition}
s(F_1) {\bf n}_{F_1} + s(F_2) {\bf n}_{F_2} 
+ s(F_3) {\bf n}_{F_3} + s(F_4) {\bf n}_{F_4} = 0
\end{equation}
\end{enumerate}
\index{canonical scaling}
If we have a collection $S$ of faces of the tiling of arbitrary
dimensions, for example, the whole tiling or the star of a face, 
by a canonical scaling of the collection $S$ we mean a canonical 
scaling of the subset of $S$ consisting of facets,
ie. $(d-1)$-faces. 
\end{definition}

This definition was originally developed
by Voronoi in \cite{bib_voronoi_rech_par}.
He used the term ``canonically defined parallelotope''
for what we call a ``canonical scaling of the tiling''.
Our choice to use the word ``scaling'' is due
to the fact that we are scaling (assigning 
lengths to) normal vectors to facets in the tiling.

We can interpret the scale factors 
in the canonical scaling of a tiling
as the edge lengths of a {\em reciprocal graph},
a rectilinear graph in space
whose vertices correspond to parallelotopes,
and whose edges are orthogonal to their
corresponding facets. See figure 
\ref{dv-tiling-figure} on page 
\pageref{dv-tiling-figure} for 
an example of a reciprocal graph.
\index{reciprocal graph}

Two canonical scalings of face collections 
$S$ and $S'$ {\em agree} (on $S\cap S'$) if their restrictions to 
$S\cap S'$ are the same up to a common multiplier.

Consider a special case when $S = \St(F^{d-2})$.

(a) If $F^{d-2}$ is hexagonal, then each two
normal vectors to facets in its star are noncollinear,
therefore a canonical scaling
of $S$ is unique up to a common multiplier.

This implies that if $\St(F^{d-2}) \subset S'$
and $s'$ is a canonical scaling of $S'$,
then 
\begin{equation}
\label{hexagonal-scale-equation}
\frac{s'(F_1)}{s'(F_2)} = \frac{s(F_1)}{s(F_2)}
\end{equation}

(b) If $F^{d-2}$ is quadruple, then there
are two pairs of parallel facets in the star
of $F^{d-2}$, say, $F_1$ and $F_3$, $F_2$
and $F_4$. Facets $F_1$ and $F_2$ are not parallel.
For equation \ref{quadruple-candesc-condition}
to hold, we must ${\bf n}_{F_1} = - {\bf n}_{F_3}$,
${\bf n}_{F_2} = -{\bf n}_{F_4}$, and 
\begin{equation}
\label{quadruple-scale-equation}
\begin{split}
s(F_1) = s(F_3) = a \\
s(F_2) = s(F_4) = b 
\end{split}
\end{equation}
However, $a$ and $b$ can be any positive numbers. 
Therefore a canonical scaling of 
$\St(F^{d-2})$ is not unique up to a common
multiplier.

Canonical scalings are important for our 
research because of the following theorem:

\begin{theorem}
\label{equivalence-proposition}
\label{equivalence-theorem}
Consider a normal parallelotope tiling of ${\mathbb R}^d$.
The following statements are equivalent:
\begin{itemize}
\item The tiling is affinely equivalent to the 
DV-tiling for some lattice.
\item The tiling has a canonical scaling.
\end{itemize}
\end{theorem}

Therefore, to prove the Voronoi conjecture
for a tiling, it is sufficient to obtain
a canonical scaling for the tiling.

We gave the proof of this result for the hexagonal 
tiling in section \ref{voronoi_work_section}.
Voronoi proved it for the case of primitive tilings.
See paper \cite{bib_deza_equiv_statements} by Deza 
for a proof in the general case.
\end{subsection}

\begin{subsection}{The quality translation theorem}
\label{quality-translation-section}

To obtain canonical scalings, we use powerful 
topological ideas. Originally developed by Voronoi, 
they were generalized by Ryshkov and Rybnikov 
(\cite{bib_ryshkov_rybnikov}) into the Quality 
Translation Theorem.

The theorem is applicable to a wide class of 
polyhedral complexes. We need to review the 
complex-theoretic terminology to be able 
to present the theorem.

\begin{definition}(\cite{lefschetz-algebraic-topology})
A polyhedral complex in ${\mathbb R}^n$ is 
a countable complex $K$ with the following 
properties:
\begin{enumerate}
\item a $p$-dimensional element, or $p$-cell, $c^p$ 
is a (bounded, convex) relatively open polytope 
in a $p$-dimensional affine space in ${\mathbb R}^n$;
\item the cells are disjoint;
\item the union of the cells $c' \prec c^p$ is 
the topological closure of $c^p$ in ${\mathbb R}^d$;
\item If $\phi(c)$ is the union of the cells 
$c'\notin \St_{K}(c)$, then the topological closure 
of $\phi(c)$ does not intersect $c$.
\end{enumerate}
We denote by $|K|$ the {\em support} of $K$, 
the union of all cells $c\in K$.
\end{definition}
\index{polyhedral complex}
\index{$\prec$}

We assume that the set $|K|$ has the topology 
induced from the Euclidean space. The {\em closure} $\Cl(c^p)$
of a cell $c^p$ is the collection of cells $c'$ with $c'\prec c^p$.
The star $\St_{K}(c^p)$ of 
$c^p$ is the collection of cells whose closures 
contain $c^p$. We write $\St(c^p)$ if it is clear
which complex we are considering. The boundary 
of cell $c^p$ is the closure of $c^p$ without $c^p$.
\index{closure}
\index{star}

The dimension $n$ of a complex, $\dim(K)$, is the 
maximal dimension of an element. An upper index is used 
to indicate the dimension of the complex and its cells.
A polyhedral complex $K^n$ is called {\em dimensionally 
homogeneous} if each cell  of dimension less than $n$ 
is on the boundary of a cell of dimension $n$. 
A {\em strongly connected complex} is one 
where each two distinct $n$-cells can be connected 
by a sequence of $n$-cells where two consecutive 
cells intersect over a $(n-1)$-cell. The $m$-dimensional
{\em skeleton} $\Sk_m(K^n)$ of the complex $K^n$ 
is the complex consisting of cells of $K^n$ 
of dimension $m$ or less.
\index{dimensionally homogeneous}
\index{strongly connected complex}
\index{skeleton}

Examples of polyhedral complexes are polytopes, 
normal tilings of the Euclidean space, and their 
skeletons. In all of these examples, cells 
are the relative interiors of faces.

\begin{definition} (\cite{bib_ryshkov_rybnikov})
An $n$-dimensional simply and strongly connected, 
dimensionally homogeneous polyhedral complex 
$K^n\subset{\mathbb R}^N$ is a QRR-complex\footnote{Authors of 
paper \cite{bib_ryshkov_rybnikov} use a notion 
of QRR-complex which is slightly more general.} 
if all of the stars $\St(c^p)$, $p < n-1$ satisfy 
one of the following two local conditions.

\noindent {\bf Local condition 1.} For $0\le p \le n-3$, 
$\St(c^p)$ is combinatorially equivalent to 
the product of $c^p$ and a relatively open cone 
with a simply connected and strongly connected 
$(n-p-1)$-dimensional finite polyhedron as its base.

\noindent {\bf Local condition 2.} For $p=n-2$, 
$\St(c^p)$ is combinatorially equivalent to the 
product of $c^p$ and a relatively open cone with a connected 
$1$-dimensional finite polyhedron as its base.
\footnote{The term {\em cone} in the local conditions 
means the union of line segments joining a point 
outside the $(n-p-1)$-space spanned by the base 
polyhedron with each point in the base polyhedron, 
minus the base polyhedron itself.}
\end{definition}
\index{QRR-complex}

\begin{proposition}(\cite{bib_ryshkov_rybnikov})
The following complexes are QRR:
\begin{enumerate}
\item Polytopes of arbitrary dimension.
\item Normal polytopal tilings of the Euclidean space ${\mathbb R}^d$.
\item Skeletons of dimension $2$ or more of QRR-complexes.
\end{enumerate}
\end{proposition}

A {\em combinatorial path} on a QRR-complex $K^n$ is a sequence of 
$n$-cells $[c_1^n,\dots,c_k^n]$ where two consecutive cells either share 
a common $(n-1)$-cell, called a {\em joint}, or coincide. A {\em combinatorial circuit}
is a path with at least two different $n$-cells, where the first and last 
cells coincide. A circuit is {\em $k$-primitive} if all its cells belong to 
the star of the same $k$-dimensional cell.

Suppose that we want to assign a positive number $s(c^n)$ 
to each $n$-cell of $K^n$. For example, we will need 
to calculate the scale factors in canonical scalings. 
Suppose that the numbers assigned to adjacent $n$-cells 
$c_1^n,c_2^n$ must be related as follows:
\begin{equation}
\frac{s(c_2^n)}{s(c_1^n)} = T[c_1^n,c_2^n],
\end{equation}
where values of $T[c_1^n,c_2^n]$ are given a priori.
Compare this with equation \ref{hexagonal-scale-equation}
on page \pageref{hexagonal-scale-equation}.

Assume $T[c_1^n,c_2^n]T[c_2^n,c_1^n] = 1$.  A positive number 
can be assigned to combinatorial paths $[c_1^n,\dots,c_k^n]$ by
\begin{equation}
T[c_1^n,\dots,c_k^n] = T[c_1^n,c_2^n]T[c_2^n,c_3^n]\dots T[c_{k-1}^n,c_k^n].
\end{equation}
The number $T[c_1^n,\dots,c_k^n]$ is called the {\em gain}
along the path $[c_1^n,\dots,c_k^n]$,
and the function $T$ is called the {\em gain function}. 
\index{gain}
\index{gain function}

We can  assign a number $s(c_0^n)=s_0$  
to a fixed cell $c_0^n$ and then try to use the gain  
function to assign a number to any other cell $c^n$. 
We connect $c^n$  with $c_0^n$ by a combinatorial path 
$[c_0^n,c_1^n,\dots,c_k^n]$ where $c_k^n = c^n$ 
and then let
\begin{equation}
s(c^n) = s(c_0^n) T[c_0^n,c_1^n,\dots,c_k^n].
\end{equation}
The question of consistency is whether {\em different} paths 
will lead to the same  value of $s(c^n)$. Clearly that will 
be the case if and only if the gain along every combinatorial 
circuit is $1$.
\begin{theorem}(Ryshkov, Rybnikov, \cite{bib_ryshkov_rybnikov})
\label{quality_translation_theorem}
\label{quality-translation-theorem}
(The quality translation theorem)
The gain $T$ along all combinatorial circuits is $1$ 
if and only if the gain along all $(n-2)$-primitive 
combinatorial circuits is $1$. 
\end{theorem}
\index{quality translation theorem}
In other words, if the process of assigning 
numbers to cells works consistently on the star 
of each $(n-2)$-face, then it works consistently 
on the whole complex.

The quality translation theorem 
can be used verbatim for assigning
elements of a linear space ${\mathbb R}^d$
to cells in a complex. The gain function
in this case is ${\mathbb R}^d$-valued.
Authors of paper  \cite{bib_ryshkov_rybnikov}
formulated the theorem in an even 
more general setting which allows,
for example, to assign colors to polytopes
in a tiling.
\end{subsection}

\begin{subsection}{Voronoi's result}
To illustrate the use of theorem \ref{quality_translation_theorem}, 
we will prove the following theorem of Voronoi 
\cite{bib_voronoi_rech_par}:
\begin{theorem}
\label{voronoi-theorem}
A primitive parallelotope tiling of $d$-dimensional
space is affinely equivalent to a DV-tiling.
\end{theorem}

\begin{proof} 
For $d=1$, the result is trivial. For $d=2$,
the only primitive parallelotope tiling is 
the hexagonal one. We've worked it out
in the introduction to the paper. 
We assume below that $d\ge 3$.

By theorem \ref{equivalence-proposition}
on page \pageref{equivalence-proposition}, it is sufficient 
to obtain a canonical scaling of the tiling. This 
will be the strategy of the proof.

First we describe the structure of the star
of a vertex $v$. By the assumption
of the theorem, there are exactly $d+1$ parallelotopes
$P_1,\dots,P_{d+1}$ in the star of $v$. Each edge
incident with $v$ has at least $d$ parallelotopes in its star,
however it cannot have all $d+1$ parallelotopes, 
because their intersection is $v$. Therefore each 
edge has exactly $d$ parallelotopes in its star. 

Since the collection of parallelotopes in the star of 
an edge uniquely defines the edge,
there are at most $d+1$ edges incident with $v$
(which is the number of ways to choose
$d$ parallelotopes out of $d+1$). On the other hand,
there cannot be $d$ or less edges because in this case
all edges must be in the same parallelotope.

We have proved that there are exactly $d+1$ edges
of the tiling which are incident with $v$.
Each of the $d+1$ parallelotopes contains exactly $d$
of these edges. Therefore $v$ is a simple vertex
of each parallelotope, and {\em each $k$ edges
incident with $v$ are the edges of a face
of the tiling of dimension $k$, for $k=1,\dots,d$.}
Since the vertex $v$ was chosen arbitrarily,
it follows that each $(d-k)$-dimensional face
of the tiling is on the boundary of exactly $k+1$
parallelotopes, $k=1,\dots,d$. In particular,
each $(d-2)$-face is hexagonal, since it has
$3$ parallelotopes in its star.

We denote the edge vectors of the $d+1$ edges
by $e_1$,\dots,$e_{d+1}$ so that the edges
can be written as $[v,v+e_1]$, $[v,v+e_2]$,\dots,
$[v,v+e_{d+1}]$. Choosing the numbering
appropriately, we have the vertex cone
of parallelotope $P_j$ at $v$ equal to
\begin{equation}
C_j = v + \cone\lbrace e_i  :  i\ne j \rbrace
\end{equation}
By $\cone(A)$ we mean the set of all linear 
\index{cone}
combinations of vectors from $A$ with nonnegative
coefficients. Cones $C_1,\dots,C_{d+1}$ form
a tiling of the space (because in a small
neighborhood of $v$ they coincide with
the parallelotopes).

We now prove that there is a canonical scaling
of $\St(v)$. Consider the following polyhedron:
\begin{equation}
\Delta = \lbrace x : x\cdot e_i \le 1, i=1,\dots,d+1 \rbrace.
\end{equation}
We have $\int(\Delta)\ne\emptyset$, for $0\in\int(\Delta)$.
The polyhedron $\Delta$ is bounded. For, assuming it 
is not bounded, there is a half-line $\lbrace ut: t\ge 0\rbrace \subset \Delta$,
$u\in{\mathbb R}^d$, $u\ne 0$. We then have $e_i\cdot u \le 0$
for all $i=1,\dots,d+1$ which is a contradiction of the
fact that the cones $C_j$ cover the space. 

Therefore $\Delta$ is a simplex with vertices $v_i$
given by $v_i\cdot e_j = 1$ for $i\ne j$, 
$v_i\cdot e_i < 1$. 

We construct the canonical scaling to $\St(v)$
using simplex $\Delta$. The affine hull of each
facet $F_{kl} = P_k \cap P_l$ in the star can be written as 
$v + \lin\lbrace e_i: i\ne k,l \rbrace$.
The edge vector $v_k - v_l$ of the simplex $\Delta$ 
is therefore orthogonal to $F_{kl}$. We
let 
\begin{equation}
s(F_{kl}) = |v_k - v_l|.
\end{equation}
Let ${\bf n}_{kl} = \frac{v_k - v_l}{|v_k - v_l|}$. 
Then ${\bf n}_{kl}$ is a unit normal vector to $F_{kl}$.
To prove that the scale factors we have assigned 
indeed form a canonical scaling, we need to
check conformance to definition
\ref{candesc-definition} on page \pageref{candesc-definition}.
That is, we need to establish the following equation:
\begin{equation}
s(F_{kl}){\bf n}_{kl} + s(F_{lm}){\bf n}_{lm} + s(F_{mk}){\bf n}_{mk} = 0
\end{equation}
for all distinct $k,l,m \in 1\dots d+1$. 
This equation obviously holds; we just plug 
in the definitions of the scale factors
and unit normals to check it. 

We now construct a canonical scaling 
of the tiling. We have proved that each $(d-2)$-face
$F^{d-2}$ of the tiling is hexagonal.
Let the facets in its star be $F$, $G$ and $H$,
and let $\alpha = \alpha_{F^{d-2}}$ be its canonical scaling. 
We apply the quality 
translation theorem (theorem \ref{quality-translation-theorem} 
on page \pageref{quality-translation-theorem}) 
to the $(d-1)$-skeleton of the tiling with 
the following gain function:
\begin{equation}
T[F,G] = \frac{\alpha(G)}{\alpha(F)}.
\end{equation}
To apply the theorem, we need to check 
that the gain along an arbitrary $(d-3)$-primitive
circuit $[F_1,\dots,F_n,F_{n+1}]$, where $F_{n+1} = F_1$,
is equal to $1$. Suppose that the circuit is in the star
of face $F^{d-3}$ of the tiling. We have 
\begin{equation}
\label{gain-equation-primitive-proof}
T[F_1,\dots,F_n,F_1] = 
\frac{\alpha_1(F_2)}{\alpha_1(F_1)} 
\frac{\alpha_2(F_3)}{\alpha_2(F_2)}\cdot\dots\cdot
\frac{\alpha_n(F_1)}{\alpha_n(F_n)}
\end{equation}
where $\alpha_i$ is the canonical scaling
of the star of $(d-2)$-face $F^{d-2}_i$ joining $F_i$ and $F_{i+1}$.
We have $\St(F^{d-3}) \subset \St(v)$, where $v$ is any
vertex of $F^{d-3}$. Therefore $\St(F^{d-3})$  
has a canonical scaling $s$, which can be obtained
by restricting the canonical scaling to $\St(v)$.

Since $\St(F_i^{d-2}) \subset \St(F^{d-3})$,
we have by equation \ref{hexagonal-scale-equation} on 
page \pageref{hexagonal-scale-equation}
\begin{equation}
\frac{\alpha_i(F_{i+1}) } {\alpha_i(F_i) } = \frac{s(F_{i+1})} {s(F_i)}, 
i = 1,\dots,n.
\end{equation}
Factors in the numerator
and denominator in equation \ref{gain-equation-primitive-proof}
cancel each other and we have $T[F_1,\dots,F_n,F_1] = 1$.
By the quality translation theorem, the tiling
has a canonical scaling, and therefore
it is affinely equivalent to a DV-tiling, 
by theorem \ref{equivalence-proposition}
on page \pageref{equivalence-proposition}.
This proves the Voronoi conjecture for primitive tilings.
\end{proof}

\end{subsection}
\end{section}

\begin{section}{Dual cells}
\label{dual-cell-combinatorics-section}
\index{dual cell}

In the previous section, we introduced canonical scalings
of a tiling, whose existence is equivalent to the 
Voronoi conjecture. Our tool for constructing a
canonical scaling is the quality translation theorem
(theorem \ref{quality-translation-theorem}
on page \pageref{quality-translation-theorem}).
Applying the theorem to the $(d-1)$-skeleton
of the tiling requires an understanding of $(d-3)$-primitive
combinatorial circuits, or circuits of facets in the stars 
of $(d-3)$-faces. 

In this section, we introduce {\em dual cells},
objects that are very convenient
for describing stars of faces in the tiling.
We give the classification of stars of faces
of dimension $(d-2)$ and $(d-3)$.

\begin{subsection}{Definition of dual cells}
A DV-tiling can be equipped with a metrically dual tiling, known 
as a {\em Delaunay tiling}, or {\em L-tiling}. Vertices of the dual 
tiling are the centers of parallelotopes; the faces are orthogonal 
to the corresponding faces of the tiling\footnote{See Delaunay's 
paper \cite{bib-delaunay-empty-sphere} for details on the  
L-tiling.}.

However, we do not know whether an arbitrary parallelotope 
is a DV-domain or is affinely equivalent to one. In fact, proving
that for 3-irreducible parallelotope tilings is the objective 
of our work. The following definition works whether or not
we have a DV-tiling.

\begin{definition}
Let $0\le k \le d$.  Given a face $F^{d-k}$ of the tiling, the 
{\em dual cell} $D^k$ corresponding to $F^{d-k}$ is the convex 
hull of centers of parallelotopes in the star of $F^{d-k}$. 
\end{definition}
\index{dual cell}

It is so far an open problem to show that $\dim(D^k) + \dim(F^{d-k}) = d$ 
(or, equivalently, that $\dim(D^k) = k$). Therefore the number $k$ 
will be called the {\em combinatorial dimension} of $D^k$.
A {\em dual $k$-cell} is a dual cell of 
combinatorial dimension $k$. Dual $1$-cells are called {\em dual 
edges}. 
\index{combinatorial dimension}
\index{dual edge}

Suppose that $D_1$ and $D_2$ are dual cells. Then $D_1$ is called a 
{\em subcell} of $D_2$ if $\Vert(D_1)\subset \Vert(D_2)$. We
denote this relationship by the symbol ``$\prec$''.
It is important to
distinguish between subcells and faces of $D_2$ in the 
polytope-theoretic sense. We do not know so far if they are the 
same thing. 
\index{subcell}
\index{$\prec$}

A dual cell is called {\em asymmetric} if it does not have a center 
of symmetry.
\index{asymmetric dual cell}

We now prove some essential properties of dual cells.

\medskip

\noindent {\em 1.
The set of centers of parallelotopes in the star of face $F^{d-k}$ is in a 
convex position.}
We use  the notation $P(x)$ for the translate of the parallelotope 
of the tiling whose center is $x$, so that $2x - P(x) = P(x)$. Let $v$
be an arbitrary vertex of $F^{d-k}$. Consider the polytope
$Q = P(v)$ (note that $Q$ does not belong to the tiling).
The centers $c$ of parallelotopes $P \in\St(F^{d-k})$ 
are vertices of $Q$. Indeed, $v$ is a vertex of $P(c)$, therefore 
$c$ is a vertex of $P(v)$. The claim follows from the fact 
that the set of vertices of $Q$ is in a convex position.

\medskip

\noindent {\em 2. The correspondence between faces of the tiling 
and dual cells is $1-1$.}
Indeed, if $D$ is the dual cell corresponding to a face $F$ of the tiling, 
then
\begin{equation}
F = \bigcap_{v\in\Vert(D)} P(v).
\end{equation}
This result follows from the fact that a proper face of a polytope is 
the intersection of facets of the polytope which contain it, which in 
turn is a partial case of theorem 9 on page 54 of 
\cite{bib_mcmullen_convex_polytopes}.

\medskip

\noindent {\em 3. If $D_1$, $D_2$ are dual cells corresponding 
to faces $F_1$, $F_2$, then $\Vert(D_1)\subset \Vert(D_2)$ 
is equivalent to $F_2\prec F_1$.}
This result follows from the definition of the dual cell
and the previous formula.

\medskip

Results 2 and 3 can be reformulated as follows:

\medskip

\noindent {\em 4. Dual cells form a dual abstract complex to 
the tiling.} 

\medskip

One corollary is that if $D_1$ and $D_2$ are dual cells 
and their vertex sets intersect, then $\conv(\Vert(D_1)\cap\Vert(D_2))$
is a dual cell.
\end{subsection}

\begin{subsection}{Dual $2$, $3$-cells}
\label{dual-3-cells-section}

In this section, we classify all dual dual $2$- and $3$-cells.
First we quote the classification of fans of $(d-2)$- and 
$(d-3)$-dimensional faces. 

The fan of a face is defined as follows.
For each parallelotope $P$ in the star of a face $F^{d-k}$, take the cone 
${\cal C}_P$ bounded by the facet supporting inequalities for facets of $P$ which 
contain $F^{d-k}$. The cones ${\cal C}_P$ form a face-to-face partitioning 
of the $d$-dimensional space ${\mathbb R}^d$. All facets in the partitioning 
are parallel to the face $F^{d-k}$. By taking a section of the partitioning by 
a $k$-dimensional plane orthogonal to $F^{d-k}$, we get the {\em fan} 
of the tiling at face $F^{d-k}$. 
\index{fan}

Delaunay in \cite{bib_delaunay_4d} proves the following lemma.

\begin{lemma} 
\label{fans-proposition}
The fans of parallelotopes of any dimension $d$ at a face of dimension
$(d-2)$ can only be of the following two combinatorial types:
\medskip
\begin{center}
\resizebox{120pt}{!}{\includegraphics[clip=true,keepaspectratio=true]
{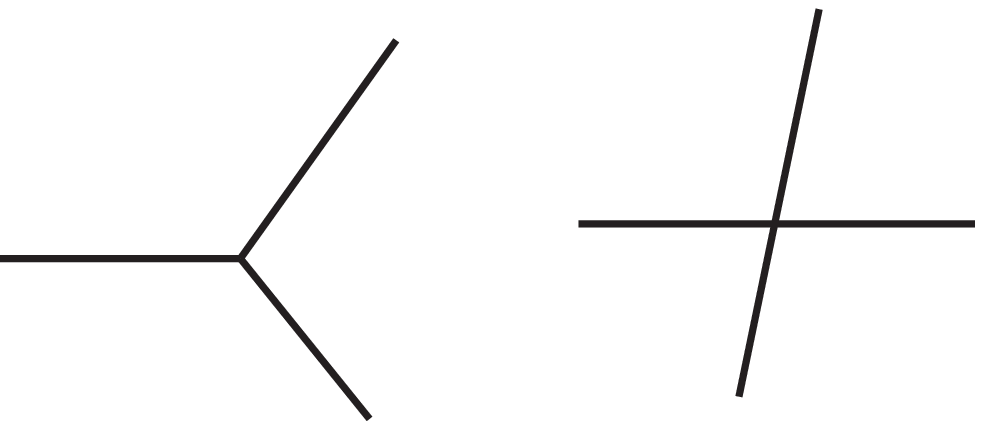}}
\end{center}

The fans of parallelotopes of any dimension $d$ at a face of dimension
$(d-3)$ can only be of the following 5 combinatorial types:
\medskip
\begin{center}
\resizebox{300pt}{!}{\includegraphics[clip=true,keepaspectratio=true]
{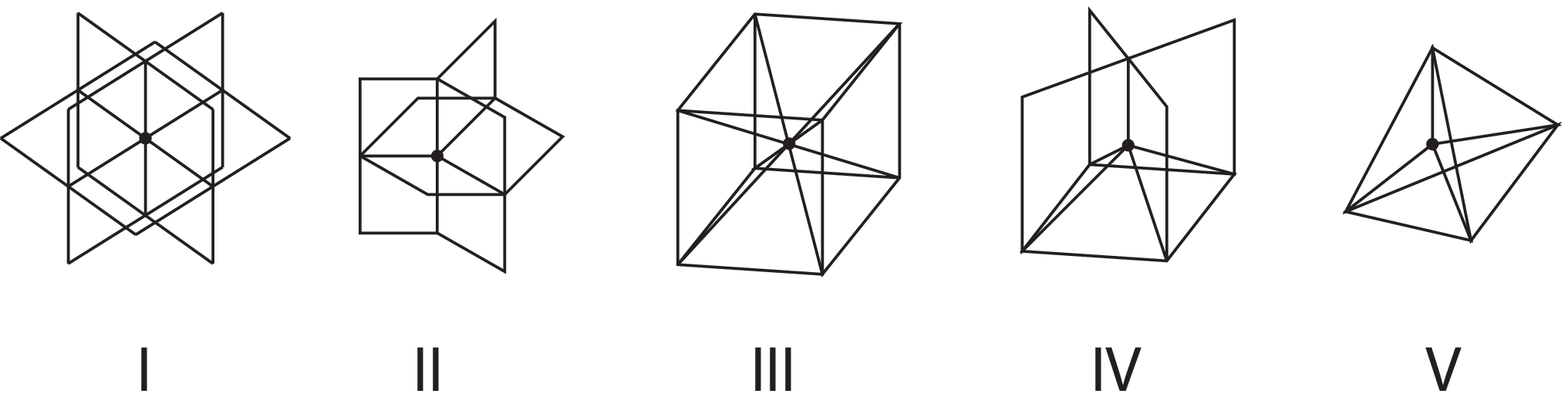}}
\end{center}
\end{lemma}

A fan of type (I) is simply the partitioning of $3$-space by $3$ linearly 
independent planes. A type (II) fan has $3$ half-planes meeting at a line, 
together with a plane which intersects the line. It is the fan at a vertex in the tiling 
of space by hexagonal prisms. A fan of type (III) has 6 cones with the vertex 
in the center of a parallelepiped; the cones are generated by the $6$ facets of the 
parallelepiped. A fan of type (IV) has $5$ cones. A fan of type (V) 
has the minimal number of $4$ cones.

\begin{lemma}
\label{dual-2-3-cells-lemma}
\label{dual-2-cells-lemma}
\label{dual_3_cells_theorem}
\label{3-reciprocal-existence-lemma}
The stars of $(d-2)$- and $(d-3)$-faces have canonical 
scalings. Dual cells $D^k$ of combinatorial 
dimension $k=2$ and $k=3$, corresponding 
to faces $F^{d-k}$ of the tiling, are given 
by the following table. Nonempty faces 
of $D^k$ are dual cells corresponding to faces 
in the star of face $F^{d-k}$.
\begin{table}
{\em
\begin{equation}
\begin{array}{|c|cccc|}
\hline
k & \text{Fan of $F^{d-k}$}   & D^k        
& \text{Canonical} & \text{Tiling locally} \\
&&& \text{scaling of $\St(F^{d-k})$} & \text{irreducible at $F^{d-k}$} \\
&&& \text{unique up to} & \\
&&& \text{a common multiplier} & \\
\hline
2           & \text{(a)}    & \text{triangle}         & \text{Yes}             & \text{Yes}   \\
            & \text{(b)}    & \text{parallelogram}    & \text{No}              & \text{No}  \\
\hline
3           & \text{I}         & \text{parallelepiped}              & \text{No}        & \text{No}  \\  
            & \text{II}        & \text{triangular prism}            & \text{No}        & \text{No}  \\
            & \text{III}       & \text{octahedron}                  & \text{Yes}       & \text{Yes}   \\
            & \text{IV}        & \text{pyramid over parallelogram}  & \text{Yes}       & \text{Yes} \\
            & \text{V}         & \text{simplex}                     & \text{Yes}       & \text{Yes} \\
\hline
\end{array}
\end{equation}
}
\caption{Dual 2- and 3-cells}
\end{table}
\end{lemma}

\begin{proof} Results of this lemma were established in
Delaunay's paper \cite{bib_delaunay_4d}.
\end{proof}

From this lemma, it follows that a tiling 
is $3$-irreducible if and only if all dual 
$3$-cells are simplices, octahedra 
or pyramids.
\end{subsection}

\end{section}

\begin{section}{Coherent parallelogram cells}
\label{topological-results-section}

We now expose the obstacle on the way to
constructing a canonical scaling for the tiling:
the {\em incoherent parallelogram dual cells}.
\index{incoherent parallelogram dual cells}

The most important result is theorem 
\ref{exposed-vertex-theorem} on page 
\pageref{exposed-vertex-theorem}, which implies
that incoherent parallelogram cells come
in groups of at least 5 (if they come at all). This helps 
to prove their nonexistence.

\begin{subsection}{Coherent parallelogram cells}
\label{combinatorial-results-section}

\begin{definition}
Consider a parallelotope tiling of $d$-dimensional
space, where $d\ge 4$. Let $\Pi$ be a parallelogram dual 
cell, $D^4$ a dual $4$-cell, $D_1^3$, $D_2^3$ pyramid dual 
cells with $\Pi\prec D_1^3 \prec D^4$,
$\Pi\prec D_2^3 \prec D^4$. Let $F^{d-2}$, $F^{d-4}$
$F_1^{d-3}$, $F_2^{d-3}$ be the corresponding faces 
of the tiling.

The parallelogram dual cell $\Pi$ is called {\em coherent}
with respect to $D^4$ if the canonical scalings
of $\St(F_1^{d-3})$, $\St(F_2^{d-3})$ agree on
$\St(F^{d-2})$. The face $F^{d-2}$ is called coherent
with respect to $F^{d-4}$ if $\Pi$ is coherent
with respect to $D^4$.
\end{definition}
Note that canonical scalings of $\St(F_1^{d-3})$ and $\St(F_2^{d-3})$
in the definition are unique because the corresponding
dual cells $D_1^3$, $D_2^3$ are pyramids, see lemma \ref{dual-2-3-cells-lemma}
on page \pageref{dual-2-3-cells-lemma}.

Information on canonical scalings can be found on page
\pageref{candesc-definition}.

The ``rhombus'' diagram of the faces below shows the
relationships between faces and dual cells in the 
definition (arrows indicate ``$\prec$'' relationships).

Note that faces $F_1^{d-3}$ 
and $F_2^{d-3}$ are uniquely defined by $F^{d-4}$ and 
$F^{d-2}$, since $F^{d-4}$ is a codimension $2$ face
of $F^{d-2}$. Therefore dual cells $D_1^3$, $D_2^3$
are uniquely defined by $D^4$ and $\Pi$.

%$3$-irreducibility of the tiling means that all 
%dual $3$-cells are simplices, octahedra and pyramids
%over parallelograms, therefore 

\begin{equation}
\label{rhombus_diagram}
\begin{array}{cccccccccc}
           &          & F^{d-4}     &          &           &           &          &     D^4     &          &                 \\
           & \swarrow &             & \searrow &           &           & \nearrow &             & \nwarrow &                 \\
 F_1^{d-3} &          &             &          & F_2^{d-3} &   D_1^3   &          &             &          &  D_2^3          \\
           & \searrow &             & \swarrow &           &           & \nwarrow &             & \nearrow &                 \\  
           &          & F^{d-2}     &          &           &           &          &     \Pi     &          &                  
\end{array}
\end{equation}

\begin{theorem}
\label{coherent-face-proposition}
\label{reciprocal-and-coherence-lemma}
Consider a $3$-irreducible tiling by parallelotopes.
The tiling has a canonical scaling if and only if 
all parallelogram dual cells $\Pi$ are coherent with respect
to all $D^4$ with $\Pi\subset D^4$.
\end{theorem}

\begin{proof} 
Note that saying: parallelogram dual cell
$\Pi \prec D^4$ is coherent  with respect 
to $D^4$ is equivalent to saying that 
$F^{d-2}$, a face of the tiling corresponding
to $\Pi$, is coherent to $F^{d-4}$, a face
of the tiling corresponding to $D^4$.

The necessity part is easy. Consider
faces of the tiling in diagram \ref{rhombus_diagram}.
If the tiling has a canonical scaling $s$, then 
canonical scalings of $\St(F_1^{d-3})$ 
and $\St(F_2^{d-3})$ can be obtained by restricting $s$:
\begin{equation}
\label{two-restrictions-equation}
\begin{split}
{s|}_{\St(F_1^{d-3})} \\
{s|}_{\St(F_2^{d-3})}
\end{split}
\end{equation}
If we in turn restrict these to 
$\St(F^{d-2}) = \St(F_1^{d-3}) \cap \St(F_2^{d-3})$, 
we get
the same canonical scaling of $\St(F^{d-2})$, namely,
\begin{equation}
{s|}_{\St(F^{d-2})}
\end{equation} 
therefore the canonical scalings in
(\ref{two-restrictions-equation}) agree on
$\St(F^{d-2})$ and $F^{d-2}$ is coherent 
with respect to $F^{d-4}$. Since the faces 
$F^{d-2}$ and $F^{d-4}$ can be chosen
arbitrarily, the necessity part of the theorem is proved.

Now we prove the sufficiency. We begin with a result
for quadruple face that we will need in the proof.
Take a quadruple face
$F^{d-2}$ of the tiling and consider
two arbitrary $(d-3)$-faces $F^{d-3},G^{d-3} \prec F^{d-2}$.
We observe that the canonical scalings
to $\St(F^{d-3})$ and $\St(G^{d-3})$ agree
on $\St(F^{d-2})$. Indeed, we can 
connect two $(d-3)$-faces $F^{d-3}$, $G^{d-3}$ 
on the boundary of $F^{d-2}$ by a combinatorial path on 
the same boundary:
\begin{equation}
F^{d-3} = F_1^{d-3},\dots,F_n^{d-3} = G^{d-3}
\end{equation}
where two consecutive $(d-3)$-faces $F_i^{d-3}$
and $F_{i+1}^{d-3}$ share a $(d-4)$-face $F_i^{d-4}$.
From the conditions of the theorem, it follows that
all dual cells corresponding to $F_i^{d-3}$, $i=1,\dots,n$
are pyramids and therefore the 
canonical scaling to $\St(F_i^{d-3})$
is unique up to a common multiplier.

From the condition of the theorem, $F^{d-2}$
is coherent with respect to $F_i^{d-4}$
for all $i=1,\dots,n-1$, which means that 
the canonical scaling of $\St(F_i^{d-3})$
agrees with the canonical scaling of
$\St(F_{i+1}^{d-3})$. By the chain argument,
we see that the canonical scalings
of $\St(F^{d-3}) = \St(F_1^{d-3})$ and 
$\St(G^{d-3}) = \St(F_n^{d-3})$ agree
on $\St(F^{d-2})$.

Now we use the quality translation theorem
to obtain the canonical scaling 
of the tiling. Let $F^{d-2}$ be a $(d-2)$-face.
It can be either quadruple or hexagonal.
We define the gain
function as follows: suppose that facets $F_1$, $F_2$
of the tiling are in the star of 
the $(d-2)$-face $F^{d-2}$. We set
\begin{equation}
\label{criterion-gain-equation}
T[F_1,F_2] = \frac{\alpha(F_2)}{\alpha(F_1)},
\end{equation}
where $\alpha = \alpha_{F^{d-3}}$ is the canonical 
scaling of $\St(F^{d-3})$ for an arbitrary
face $F^{d-3}$ with $F^{d-3}\prec F^{d-2}$,
so that $\St(F^{d-2}) \subset \St(F^{d-3})$.

The definition of $T[F_1,F_2]$ is consistent,
that is, the value of $\frac{\alpha(F_2)}{\alpha(F_1)}$ 
does not depend
on the choice of $F^{d-3}$ if $F^{d-2}$ is hexagonal. 
Indeed, if
face $F^{d-2}$ is hexagonal, then the canonical
scaling to $\St(F^{d-2})$ is unique up 
to a common multiplier. If  $\beta=\beta_{F^{d-2}}$
is some fixed canonical scaling to $\St(F^{d-2})$,
then the restriction of $\alpha$ to $\St(F^{d-2})$
must agree with it:
\begin{equation}
\frac{\alpha(F_2)}{\alpha(F_1)} = \frac{\beta(F_2)}{\beta(F_1)},
\end{equation}
therefore the definition 
of $T[F_1,F_2]$ does not depend on the choice of $F^{d-3}$.

Suppose now that face $F^{d-2}$ is quadruple. As we
argued above, if we have another face $G^{d-3}\prec F^{d-2}$,
then the canonical scalings $\alpha=\alpha_{F^{d-3}}$, 
$\alpha' = \alpha_{G^{d-3}}$ to $\St(F^{d-3})$ 
and $\St(G^{d-3})$ agree on $\St(F^{d-2})$, therefore
\begin{equation}
\frac{\alpha(F_2)}{\alpha(F_1)} = \frac{\alpha'(F_2)}{\alpha'(F_1)}
\end{equation}
We have proved that the value of $T[F_1,F_2]$
is well-defined.

By construction, the gain function is equal to $1$
on $(d-3)$-primitive circuits (the reader may find
a detailed argument in the proof of theorem
\ref{voronoi-theorem} on page \pageref{voronoi-theorem}).
Application of the quality translation theorem gives
an assignment $s$ of positive numbers to 
facets of the tiling. 

Equation (\ref{criterion-gain-equation}) guarantees
that the restrictions of $s$ to stars of 
$(d-2)$-faces are canonical scalings. Therefore 
the conditions of definition \ref{candesc-definition}
on page \pageref{candesc-definition} are satisfied
by $s$. We have proved that the tiling has a canonical
scaling.
\end{proof}
\end{subsection}

\begin{subsection}{Sufficient conditions for coherency
of a parallelogram dual cell}

\begin{theorem}\label{exposed-vertex-theorem}
Consider a $3$-irreducible normal parallelotope tiling. 
 Let $D^4$ be a dual cell, $\Pi\subset D^4$ a parallelogram subcell. 
Suppose $\Pi$ has a vertex $v$ such that all parallelogram subcells 
$\Pi'\subset D^4$ with $v\in \Pi'$ are coherent except, perhaps, $\Pi$. 
Then $\Pi$ is coherent. In particular, if $\Pi$ is the only
parallelogram in $D^4$ that contains $v$, then $\Pi$ is coherent.
\end{theorem}

Speaking informally, incoherent parallelograms ``come in droves'':
each vertex of an incoherent parallelogram must also be a vertex
of another incoherent parallelogram.

\medskip

\begin{proof} Let $F^{d-4}$ be the face 
of the tiling corresponding to $D^4$, 
$P$ the parallelotope of the tiling with center $v$, 
$F^{d-2}$ the face corresponding to $\Pi$,  
$F_1^{d-3}$, $F_2^{d-3}$ faces with the property 
that $F^{d-4} \prec F_1^{d-3},F_2^{d-3} \prec 
F^{d-2}$. Let $D_1^3$, $D_2^3$ be the dual 
cells corresponding to faces $F_1^{d-3}$,
$F_2^{d-3}$. By our assumption about the tiling, 
they are pyramids. The relationship of the introduced 
polytopes can be seen in diagram 
\ref{two-pyramid-diagram} on page 
\pageref{two-pyramid-diagram}, arrows indicate inclusion.

Let $I_1$, $I_3$ be the edges of $\Pi$ containing 
$v$, and let $I_2$,$I_4$ be the edges of pyramids 
$D_1^3$, $D_2^3$ containing $v$ and different 
from $I_1$,$I_3$. Let $F_1,\dots,F_4$ be the 
facets of the tiling corresponding to dual 
edges $I_1,\dots,I_4$. Take canonical scalings 
$s_1$ and $s_2$ to $\St(F_1^{d-3})$ 
and  $\St(F_2^{d-3})$. 
To show that they agree on the star of $F^{d-2}$, 
it is sufficient to check that they assign 
proportional scale factors
to facets $F_1$ and $F_3$, ie.,
\begin{equation}
\label{pi-coherent-equation}
\frac{s_1(F_1)}{s_1(F_3)} = \frac{s_2(F_1)}{s_2(F_3)}
\end{equation}

To prove this result, we use the quality translation 
theorem (theorem \ref{quality-translation-theorem}
on page \pageref{quality-translation-theorem}). Consider
face collections $K'$ and $K$:
\begin{equation}
\begin{split}
K' &= \lbrace F: F^{d-4} \prec F \prec P \rbrace, \\
K &= K' \setminus\lbrace 
F^{d-4}, F_1^{d-3}, F_2^{d-3}, F^{d-2}, P\rbrace.
\end{split}
\end{equation}

We define the gain function on combinatorial paths in $K$.
Consider a path $[G_1,G_2]$ on $K$, where facets 
$G_1$ and $G_2$ share a face $G^{d-2} \ne F^{d-2}$.
Note that $G^{d-2}$ is coherent with respect to $F^{d-4}$:
since $\Pi$ is the only parallelogram cell with 
$v \prec \Pi \prec D^4$ which may a priori be incoherent 
with respect to $D^4$,
$F^{d-2}$ is the only potentially incoherent $(d-2)$-face 
with $F^{d-4} \prec F^{d-2} \prec P$.

There are exactly two $(d-3)$ faces 
$G_1^{d-3}$, $G_2^{d-3}$
 with $F^{d-4} \prec G_i^{d-3} \prec G^{d-2}$. 
We define $T[G_1,G_2]$ by
\begin{equation}
\label{2nd-gain-equation}
T[G_1,G_2] = \frac{s(G_2)}{s(G_1)} = 
\frac{s'(G_2)}{s'(G_1)}.
\end{equation}
where $s$, $s'$ are 
the canonical scalings of $\St(G_1^{d-3})$ and $\St(G_2^{d-3})$. 
The two last fractions are equal because the 
canonical scalings $s$, $s'$ agree on the star 
of $G^{d-2}$. If $G^{d-2}$ is hexagonal, then
this result follows from the uniqueness (up to a common multiplier)
of a canonical scaling to $\St(G^{d-2})$. If $G^{d-2}$ 
is quadruple, then it is coherent with respect to $F^{d-4}$, 
therefore $s$ and $s'$ agree 
on $\St(F^{d-2})$.

Note that equation (\ref{2nd-gain-equation})
implies that the gain on $(d-3)$-primitive circuits
is equal to $1$ (see a detailed argument in 
the proof of theorem \ref{voronoi-theorem} on page 
\pageref{voronoi-theorem}).

The statement that the parallelogram $\Pi$ is 
coherent (equation (\ref{pi-coherent-equation})) 
can be rewritten as:
\begin{equation}
T[ F_1,F_2,F_3,F_4,F_1 ] = 1.
\end{equation}
Indeed, $F_1,F_2,F_3\in\St(F_1^{d-3})$ and
$F_1,F_4,F_3\in\St(F_2^{d-3})$, hence
$$\frac{s_1(F_3)}{s_1(F_1)} = \frac{s_1(F_3)}{s_1(F_2)}\frac{s_1(F_2)}{s_1(F_1)} = T[ F_1,F_2,F_3 ]$$ and 
$$\frac{s_2(F_1)}{s_2(F_3)} = \frac{s_2(F_1)}{s_1(F_4)}\frac{s_1(F_4)}{s_2(F_3)} = T[ F_3,F_4,F_1 ].$$

By theorem 16 on page 71 of \cite{bib_mcmullen_convex_polytopes},
$K'=\lbrace F: F^{d-4} \prec F \prec P \rbrace$ 
is combinatorially isomorphic to a $3$-dimensional 
polytope $Q$ (ie., the lattices of the two complexes 
are isomorphic). Denote the isomorphism by $\delta$. 
$I = \delta(F^{d-2})$ is an edge of $Q$. All other 
edges of $Q$ are the images of coherent $(d-2)$-faces. 
Facets of $Q$ are the images of facets in $K'$. 
Under the isomorphism, complex $K'$  is mapped 
onto the boundary of $Q$ minus the edge $I$ 
and its two endpoints. 

Removing the edge from the boundary of $Q$ 
leaves it simply and strongly connected,
therefore the quality translation theorem 
(theorem \ref{quality-translation-theorem}
on page \pageref{quality-translation-theorem})
can be applied to $K'$. It follows that the gain on the 
circuit $[ F_1,F_2,F_3,F_4,F_1 ]$ is equal to $1$, 
and the theorem's result follows\footnote{Strictly
speaking, the boundary of $Q$ with an edge and its
endpoints removed is not a QRR-complex, but
we can cut away the edge together with a small
neighborhood, producing a $2$-dimensional
polyhedral QRR-complex.}.
\end{proof}

\begin{spacing}{1}
\begin{equation}
\label{two-pyramid-diagram}
\begin{array}{ccccc|ccccc|ccccc}
       &          & D^4         &          &        &          &           & F^{d-4}    &          &            &             &           & \emptyset  &          &            \\ 
       & \nearrow &             & \nwarrow &        &          & \swarrow  &            & \searrow &            &             & \swarrow  &            & \searrow &            \\
 D_1^3 &          &             &          & D_2^3  & F_1^{d-3}&           &            &          & F_2^{d-3}  & v_1         &           &            &          & v_2        \\
       & \nwarrow &             & \nearrow &        &          & \searrow  &            & \swarrow &            &             & \searrow  &            & \swarrow &            \\  
       &          & \Pi         &          &        &          &           & F^{d-2}    &          &            &             &           & I          &          &            \\ 
       & \nearrow &             & \nwarrow &        &          & \swarrow  &            & \searrow &            &             & \swarrow  &            & \searrow &            \\
 I_1   &          &             &          &  I_3   & F_1      &           &            &          & F_3        & \delta(F_1) &           &            &          & \delta(F_3)\\
       & \nwarrow &             & \nearrow &        &          & \searrow  &            & \swarrow &            &             & \searrow  &            & \swarrow &            \\  
       &          & v           &          &        &          &           & P          &          &            &             &           & Q          &          &            \\
%      &          &             &          &        &          &           &            &          &            &             &           &            &          &            \\
\end{array}
\end{equation}
\end{spacing}

\begin{theorem} 
\label{bipyramid-coherence-theorem}
Let $D^4$ be a dual $4$-cell in a parallelotope tiling. Suppose 
that $D^4$ is a centrally  symmetric bipyramid over a
parallelepiped $C$, and that subcells of $D^4$ coincide with its faces. Then
all parallelogram dual subcells of $D^4$ are coherent with respect 
to $D^4$.
\end{theorem}

\begin{proof} 
Polytope $D^4$ is the convex hull of a parallelepiped $C$
and two vertices $v,v'$ outside the space
spanned by $C$, whose midpoint is 
the center of $C$. The facets of $D^4$
are 12 pyramids over parallelograms.

All edges of $D^4$ are dual $1$-cells.
$D^4$ has $28$ edges: $8$ edges incident
with $v$, $8$ edges incident with $v'$,
and $12$ edges of $C$. Let ${\cal F}$, ${\cal F}'$,
and ${\cal M}$ be the collections of facets 
of the tiling, corresponding to these $3$ groups
of dual edges.

Consider the complex $K$, the union of
(closed) pyramid facets of $D^4$ meeting at $v$.
Complex $K$ contains 6 pyramids over
parallelograms. Each pyramid $D$ is a dual cell.
Let $F^{d-3}$ be the corresponding face
of the tiling. By lemma \ref{dual-2-3-cells-lemma}
(page \pageref{dual-2-3-cells-lemma}),
there is a canonical scaling of
$\St(F^{d-3})$, which we will denote
$s_{D}$. We will fix $s_{D}$ for each
of the pyramids in $K$.

We now find a positive multiplier $m(D)$ for
each pyramid $D$ so that the 6 canonical scalings
\begin{equation}
m(D) s_{D}
\end{equation}
together define a canonical scaling 
on the collection
\begin{equation}
{\cal F} \cup {\cal M}.
\end{equation}

In other words, we require that whenever
two pyramids $D_1$, $D_2$ share a 
triangular dual cell $T$, then
for each facet $F$ corresponding
to one of the edges of $T$,
we have
\begin{equation}
\begin{split}
m(D_1) s_{D_1}(F) = m(D_2) s_{D_2}(F), \\
\text{or} \\
\frac{s_{D_1}(F)}{s_{D_2}(F)} = \frac{m(D_2)}{m(D_1)}.
\end{split}
\end{equation}
We define the gain function on $K$ as follows:
\begin{equation}
T[D_1,D_2] = \frac{s_{D_1}(F)}{s_{D_2}(F)}.
\end{equation}
The definition does not depend on the choice
of $F$. This is easy to see. Let $F_1^{d-3}$, 
$F_2^{d-3}$, $F^{d-2}$ be the faces of 
the tiling corresponding to dual
cells $D_1$, $D_2$, $T$. Since $T$ is a
triangle, face $F^{d-2}$ is hexagonal
and the canonical scaling to $\St(F^{d-2})$
 is unique up to a factor. Since $s_{D_1}$
and $s_{D_2}$ are canonical scalings
to $\St(F_1^{d-3})$ and $\St(F_2^{d-3})$,
and $\St(F^{d-2}) = \St(F_1^{d-3}) \cap \St(F_2^{d-3})$,
for each two facets $F,F'\in \St(F^{d-2})$ we
have (see equation (\ref{hexagonal-scale-equation})
on page \pageref{hexagonal-scale-equation}):
\begin{equation}
\frac{s_{D_1}(F)}{s_{D_2}(F)} = \frac{s_{D_1}(F')}{s_{D_2}(F')}.
\end{equation}
We have proved that the definition 
of $T[D_1,D_2]$ does not depend on the
choice of $F\in \St(F^{d-2})$.

The dimension of complex $K$ is $3$. 
To apply the quality translation theorem
(theorem \ref{quality_translation_theorem}
on page \pageref{quality_translation_theorem}),
we need to check that a $1$-primitive circuit
$[D_1,\dots,D_n,D_1]$ has gain 1. Indeed, since a $1$-primitive
circuit in $K$ is in the star of an edge, the facet $F$
corresponding to that edge belongs to the stars
of all faces $F_i^{d-3}$ corresponding to
$D_i$, and the gain function can be calculated
using the value of $s_{D_i}(F)$ alone
as
\begin{equation}
T[D_1,\dots,D_n,D_1] =
\frac{s_{D_1}(F)}{s_{D_2}(F)} \frac{s_{D_2}(F)}{s_{D_3}(F)}\cdot\dots\cdot
\frac{s_{D_n}(F)}{s_{D_1}(F)} = 1.
\end{equation}

Clearly $K$ is a QRR-complex. Therefore, by 
the quality translation theorem, we obtain
scale factors $m(D)$ with the required
property.

We now have canonical scalings $m(D) s_{D}$
which together define a common canonical 
scaling $s$ on 
${\cal F} \cup {\cal M}$.

Each 4 facets of the tiling
corresponding to a group of parallel
edges of $C$ receive the same scale
factor $s$. To prove this, it is sufficient
to check that parallel facets $F_1$, $F_2$
corresponding to parallel edges 
of a parallelogram cell get the same scale 
factor. Let $D\in K$ be a pyramid which
contains the parallelogram, and let $F^{d-3}$
be the corresponding face of the tiling.
We have
\begin{equation}
\frac{s(F_1)}{s(F_2)} = \frac{s_D(F_1)}{s_D(F_2)} = 1,
\end{equation}
because $s_D$ is a canonical scaling
to $\St(F^{d-3})$.

Note that $D^4$ is centrally symmetric, 
as is the star $\St(F^{d-4})$. The image
of facet collection ${\cal F} \cup {\cal M}$
under the central symmetry is 
${\cal F}' \cup {\cal M}$.

We have proved in the last paragraph
that parallel facets in ${\cal M}$
have the same scale factors $s$. This
means that $s$ can be continued to the whole
set 
\begin{equation}
{\cal F} \cup {\cal F}' \cup {\cal M}
\end{equation}
by the central symmetry $*$ of $D^4$, by assigning
$s(F) = s(*(F))$ for 
$F \in {\cal F}' \cup {\cal M}$. 
It is easy to see that the continuation 
of $s$ is a canonical scaling
of $\St(F^{d-4})$.

It follows that all parallelogram subcells
of $D^4$ are coherent. This completes the
proof of the theorem.
\end{proof}
\end{subsection}
\end{section}

\begin{section}{Geometric results on dual cells}
\label{affine-results-section}
We defined dual cells in section 
\ref{dual-cell-combinatorics-section}
and used the dual cell terminology 
in the previous section to state
some combinatorial properties 
of incoherent parallelogram 
cells.

It turns out that dual cells carry not only combinatorial, 
but also geometric information about 
the parallelotope tiling. In this section, we prove 
some geometric, or polytope-theoretic 
properties of dual cells. The results
are applicable to any normal tiling by
parallelotopes and we believe they
may be useful for further research
in parallelotope theory.

The main result of this section
deals with dual cells that are
affinely equivalent to cubes of arbitrary
dimension. In particular, we prove
that a pair of parallelogram subcells
in an arbitrary
dual $4$-cell is either 
\begin{itemize}
\item complementary
(parallelograms share a vertex and
span a $4$-space), 
\item adjacent (parallelograms
share an edge and span a $3$-space),
\item translate (parallelograms are translated
copies of each other and span
a $3$-space), or 
\item skew (parallelograms
have exactly one common edge
direction and span a $4$-space).
\end{itemize}
This result (corollary
\ref{two_parallelograms_corollary}
on page \pageref{two_parallelograms_corollary})
follows from the more general 
theorem \ref{two_cubes_theorem}
on page \pageref{two_cubes_theorem}.

Throughout the section, we consider a normal 
tiling of the space ${\mathbb R}^d$ by parallelotopes, 
and denote by $\Lambda$ the lattice formed by the
centers of parallelotopes (known as the {\em lattice 
of the tiling}). We assume that one of the parallelotopes 
of the tiling is centered at $0$.

\begin{subsection}{A series of lemmas on dual cells}

\begin{lemma} \label{translate_complex_main_lemma} 
Let $D^k$ be the dual cell corresponding to a face $F^{d-k}$ 
of the tiling, $0 < k \le d$, $d \ge 1$. Suppose that $t\in\Lambda$, 
$t\ne 0$, and that $h$ is the projection along $F^{d-k}$ 
onto a complementary affine space $L$. Then: 
\begin{enumerate} 
\item $|h(\Vert(D^k))| = |\Vert(D^k)|$.  
\item If $D^k \cap (D^k + t) \ne\emptyset$, then 
$D = D^k \cap (D^k + t)$ is a dual cell and a face 
of both $D^k$ and $D^k + t$. There is a hyperplane 
$N$ in ${\mathbb R}^d$ which separates\footnote{Hyperplane $N$
separates two convex sets $A$ and $B$
if $\relint(A)$ and $\relint(B)$ belong
to different open halfspaces bounded by
$H$.} $D^k$ and $D^k + t$, with 
\begin{equation}
\label{separating-hyperplane-equation}
\begin{split}
D = N \cap D^k = N \cap (D^k + t), \\
\lin(F^{d-k} - F^{d-k}) \subset N - N
\end{split}
\end{equation}
\item The same hyperplane $N$ also separates $h(D^k)$
and $h(D^k + t)$.
\item $\Vert(h(D^k)) = h(\Vert(D^k))$.
\end{enumerate}
\end{lemma}

\begin{proof}
For $x\in{\mathbb R}^d$, $P(x)$ will denote 
the translated copy of the parallelotope $P$
so that $x$ is the center of $P(x)$.

\bigskip

As a preparation for the proof, we find a way to
{\em inscribe an inverted (centrally symmetric) 
copy of the dual cell into the parallelotope}.  
\index{inscribing a dual cell}
Fix a parallelotope $P\in \St(F^{d-k})$. 
We introduce an equivalence relation
$\sim$ on faces of the tiling. Two faces $F_1$, $F_2$ are equivalent 
if there is a vector $v\in\Lambda$ so that $F_2 = F_1 + v$. 
Consider the equivalence class which contains face $F^{d-k}$. Fix a point 
$x\in \relint(F^{d-k})$ and for every equivalent face $F = F^{d-k} + v$ 
assign the point $x(F) = x + v$ to $F$. If $F^{d-k}$
is centrally symmetric, then we choose $x$ to
be the center of symmetry of $F^{d-k}$;
we will need this further in the paper.

Let $c$ be the center of symmetry of $P$. The polytope 
$$S = \conv\lbrace x(F) : F \sim F^{d-k}, F\prec P\rbrace$$ is the image 
of $D^k$ under the central symmetry $*$ with center $\frac{x + c}{2}$. 
The transformation $*$ is defined as follows:
\begin{equation}
*(y) = x + c - y.
\end{equation}
Indeed, for each parallelotope $Q\in \St(F^{d-k})$ 
with center $c_1$, $F^{d-k} + (c - c_1)$ is a face of $P$, and
$x(F^{d-k} + (c - c_1)) = x + (c - c_1) = *(c_1)$. Conversely, 
if $F = F^{d-k} + v$ is a face of $P$, then the parallelotope 
with center $c - v$ belongs to $\St(F^{d-k})$ so 
$c - v \in \Vert(D^k)$. Note that $c - v = *(x + v) = *(x(F))$. 
Therefore sets $\lbrace x(F): F\sim F^{d-k}, F\prec P \rbrace$ and 
$\lbrace y: P(y)\in\St(F^{d-k}) \rbrace$ are mapped onto 
each other by $*$. Since $S$ and $D^k$ are their convex hulls,
we have
\begin{equation}
*(S)=D^k
\end{equation}
Since all vertices of $S$ are contained in $P$,
\begin{equation}
S \subset P
\end{equation}

1. First we prove that $|h(\Vert(D^k))| = |\Vert(D^k)|$. 
If this is not true, then there are two 
distinct parallelotopes $P$, $P'$ in the star 
of $F^{d-k}$ such that $P' = P + \lambda$ and  
$\lambda \in \lin(F^{d-k}-F^{d-k})$. Since $F^{d-k}$ is a face 
of $P + \lambda$, $F^{d-k} - \lambda$ is a face of $P$. This is 
only possible if $\lambda=0$, so $P = P'$.

\bigskip

2. Now we prove the second statement of the lemma.
Suppose that $D^k \cap (D^k + t) \ne\emptyset$, 
$t\in\Lambda$, $t\ne 0$. 

We have $S\cap (S - t) \ne\emptyset$, 
so $B = P\cap (P - t)\ne\emptyset$. 
Take a hyperplane $H$ supporting the polytope $P$ 
so that $H\cap P = B$. 

Let $m$ be the midpoint between the centers 
of $P$ and $P - t$. Remember that 
the parallelotopes of the tiling are centrally 
symmetric. It follows that the central symmetry
at $m$ maps $P$ to $P - t$, and vice versa,
so $m$ is the center of symmetry of $B$. 
Therefore $m$ belongs to $H$. 
Applying the central symmetry at $m$, we see
that  $H$ supports  $P - t$ as well, 
and $H\cap (P - t) = B$. The hyperplane $H$ 
separates polytopes $P$ and $P - t$
(that is, they are contained in different 
half-spaces bounded by $H$).

We prove that $S\cap (S-t)$ is a face
of both of the polytopes $S$ and $S - t$.
Let $H^+$, $H^-$ be the open half-spaces bounded
by $H$, whose closures contain $P$ and $P - t$ respectively.
We decompose the set of vertices of polytopes
$S$ and $S - t$ as follows:
\begin{equation}
\begin{split}
\Vert(S)\cup\Vert(S - t) = 
\lbrace x(F) : F \sim F^{d-k}, F \prec P \text{ or } F \prec (P - t) \rbrace \\ = \\
\lbrace x(F) : F \sim F^{d-k}, F \prec P, F\not\prec (P - t) \rbrace \cup \\
\lbrace x(F) : F \sim F^{d-k}, F \not\prec P, F\prec (P - t) \rbrace \cup \\
\lbrace x(F) : F \sim F^{d-k}, F \prec P \text{ and } F\prec (P - t) \rbrace \\ = \\
\lbrace x(F) : F \sim F^{d-k}, F \prec P, F\not\prec B \rbrace \cup \\
\lbrace x(F) : F \sim F^{d-k}, F\prec (P - t), F\not\prec B \rbrace \cup \\
\lbrace x(F) : F \sim F^{d-k}, F \prec B \rbrace \\=\\ A^+ \cup A^- \cup M.
\end{split}
\end{equation}
Note that for faces $F\sim F^{d-k}$ of polytopes $P$
and $P - t$, the inclusion $x(F)\in H$ is 
equivalent to $F \prec B$. Therefore 
set $A^+$ is contained in $H^+$, 
set $A^-$ is contained in $H^-$, set $M$ is contained
in $H$. Neither of the sets $A^+$, $A^-$, $M$ 
is empty. Indeed, if $M$ is empty, then
$S\cap (S-t) = \emptyset$; if, say, $A^+$
is empty, then $A^-$ is also empty and
all faces of $P$ equivalent
to $F^{d-k}$ are contained in $B$, which 
is impossible (if $F\prec B$, you can always 
take a facet containing $B$ and take $F - u \prec P$,
where $u$ is the facet vector).

We have $\Vert(S) = A^+ \cup M$
and $\Vert(S-t) = A^- \cup M$. It follows that 
$\conv(M)$ is a face of both of the polytopes $S$ and $S-t$,
$\conv(M) = S \cap (S-t)$. Also $H$ supports 
$S$ and $S - t$ and separates them.

The condition
\begin{equation}
\lin(F^{d-k} - F^{d-k}) \subset H - H
\end{equation}
holds because the hyperplane $H$ contains a face
$F\sim F^{d-k}$.

Let $N = *(H)$.
Applying the central
symmetry $*$ to polytopes $S$ and $S-t$, 
we get $D^k$ and $D^k + t$. Statement 2
of the lemma follows from the results
we have proved for $S$ and $S - t$.

\bigskip

3. To prove statement 3, note that hyperplane $*(H)$ 
is parallel to $F^{d-k}$, therefore projection 
$h$ maps the halfspaces $*(H^+)$ and $*(H^-)$ into
themselves. Since $\relint(D^k)$ and 
$\relint(D^k + t)$ belong to different open
halfspaces $*(H^+)$ and $*(H^-)$, 
sets $\relint(h(D^k)) = h(\relint(D^k))$
and $\relint(h(D^k+t)) = h(\relint(D^k+t))$ 
are contained in the same two open halfspaces. 

\bigskip

4. Finally, we prove that 
$\Vert(h(D^k)) = h(\Vert(D^k))$.
Each vertex $v$ of $S$ belongs 
to the relative interior of a face 
of $P$ which is a translate of $F^{d-k}$,
therefore $h(v)$ is a vertex of $h(P)$. 
This means that for each
vertex $v$ of $D^k$, $h(v)$ is 
a vertex of $h(*(P))$. It follows 
that $h(v)$ is a vertex of $h(D^k)$. Therefore 
$h(\Vert(D^k)) \subset \Vert(h(D^k))$.
This, together with statement 1 of the lemma, 
establishes statement 4 and completes 
the proof of the lemma.
\end{proof}

\begin{corollary} 
\label{no-vertex-inside-corollary}
Let $D$ be a dual cell. Then $\Lambda \cap D = \Vert(D)$.
\end{corollary}

\begin{proof} The inclusion $\Vert(D) \subset \Lambda\cap D$
is trivial. We need to prove $ \Lambda\cap D \subset \Vert(D)$.

Take $v\in \Lambda\cap D$. Let $y$ be any vertex of $D$.
Dual cells $D$ and $D + (v - y)$ intersect, 
$v$ is a vertex of $D + (v - y)$, and it belongs to the intersection 
$D \cap D + (v - y)$, which is a face of polytopes $D$ and $D + (v - y)$ 
by statement 2 of lemma \ref{translate_complex_main_lemma}. Therefore $v$ 
is a vertex of $D$.
\end{proof}

\begin{corollary} Let $D_1$, $D_2$ be dual cells. Then $D_1$ 
is a subcell of $D_2$ if and only if $D_1 \subset D_2$.
\end{corollary}

\begin{proof} By definition $D_1$ is a subcell of $D_2$ if and only 
if $\Vert(D_1) \subset \Vert(D_2)$. So, if that holds, then 
$D_1 \subset D_2$. 

Conversely, suppose that $D_1 \subset D_2$. We need to prove 
that $\Vert(D_1) \subset \Vert(D_2)$. We have 
$\Vert(D_1) \subset \Lambda \cap D_2$. Since by the previous
corollary $\Lambda \cap D_2 = \Vert(D_2)$, each 
vertex of $D_1$ is also a vertex of $D_2$.
\end{proof}

\begin{corollary} 
\label{parity-class-representation-corollary}
Let $D$ be a dual cell. Each parity class in the lattice $\Lambda$ 
of the tiling (a class modulo $2 \Lambda$) is represented at most 
once among vertices of $D$.
\end{corollary}

\begin{proof} Suppose that $\lambda_1$ and $\lambda_2$ are 
distinct vertices of $D$, and $\lambda_1 - \lambda_2 \in 
2 \Lambda$. Then $\lambda = \frac{\lambda_1 + \lambda_2}{2}$ 
is a lattice point, and it belongs to $D$ so by corollary 
\ref{no-vertex-inside-corollary} it is a vertex of $D$ 
which it cannot be.
\end{proof}

The next two lemmas exploit the central symmetries of a normal 
parallelotope tiling.

\begin{lemma} 
\label{central-symmetry-criterion-lemma}
A dual cell of combinatorial dimension $k$, $0 < k \le d$, is 
centrally symmetric if and only if the corresponding face 
of the tiling is the intersection of two parallelotopes. The dual cell
and the corresponding face of the tiling have the same center of symmetry. 
\end{lemma}

\begin{proof}
Suppose that $F^{d-k} = P_1\cap P_2$. Let $c_1$, $c_2$ be 
the centers of the parallelotopes $P_1$, $P_2$. The central symmetry $*$ 
at the point $c=\frac{c_1 + c_2}{2}$ is a symmetry of the whole tiling 
and it maps $F^{d-k}$ onto itself, so if a parallelotope $P$ contains 
$F^{d-k}$, then so does $*(P)$. In terms of dual 
cells, if $v\in\Vert(D)$, then $*(v)\in\Vert(D)$. This proves the sufficiency.

To prove the necessity, consider a centrally symmetric dual cell $D^k$. Let $*$ be 
the central symmetry of $D^k$ with center $c$. Note that 
$$F^{d-k} = \bigcap_{ v\in\Vert(D^k) } P(v)$$ so $c$ is the center 
of symmetry of $F^{d-k}$. Let $v\in \Vert(D^k)$. Then the intersection 
of parallelotopes $P_1$ and $P_2$ centered at $v$ and $*(v)$ respectively
is exactly $F^{d-k}$. Indeed, $c$ is the center of symmetry of both 
polytopes $P_1\cap P_2$ and $F^{d-k}$, so their relative 
interiors intersect. Since the tiling is normal (face-to-face), 
$P_1\cap P_2 = F^{d-k}$. This proves the necessity.
\end{proof}

Two vertices $v_1,v_2$ of a centrally symmetric dual cell $D$
are called {\em diametrically opposite} if $y_1 = *(y_2)$,
where $*$ is the central symmetry transformation of $D$.

\begin{lemma} 
\label{centrally_symmetric_subsets_lemma}
Let $D$ be the dual cell corresponding to a face $F$ of the tiling. 
Suppose that $X\subset \Vert(D)$ is a centrally symmetric nonempty 
subset. Then there is a centrally symmetric dual cell $D_1 \subset D$
with the same center of symmetry as $X$ 
such that $X\subset \Vert(D_1)$. 
\end{lemma}

\begin{proof}
Consider the intersection $F_1$ of parallelotopes centered 
at points of $X$. The center of symmetry of $X$ is also the center 
of symmetry of $F_1$, and of the tiling. Therefore the dual cell 
$D_1$ corresponding to $F_1$ is centrally symmetric. By construction, 
$X\subset\Vert(D_1)$. We also have $F\subset F_1$, therefore 
$D_1 \subset D$.
\end{proof}

\begin{lemma} 
\label{centrally-symmetric-face-lemma}
A centrally symmetric nonempty face of a dual cell $D$ is a dual cell.
\end{lemma}

\begin{proof}
Let $D'$ be a face of $D$, invariant under a central symmetry $*$. 
Then $\Vert(D') = \Vert(D) \cap \Vert(*(D))$, therefore $D'$ is 
a dual cell, for the reason that the dual cells form a combinatorial
complex dual to the tiling.
\end{proof}

\begin{definition} A point $v$ on the boundary of a (convex) polytope 
$Q$ is said to be illuminated by a direction $u\in \lin(Q - Q)$,
if for some $t > 0$ the point $v + tu$ belongs to the interior of $Q$.
\end{definition}
\index{illumination}

\begin{center}
\resizebox{150pt}{!}{\includegraphics[clip=true,keepaspectratio=true]{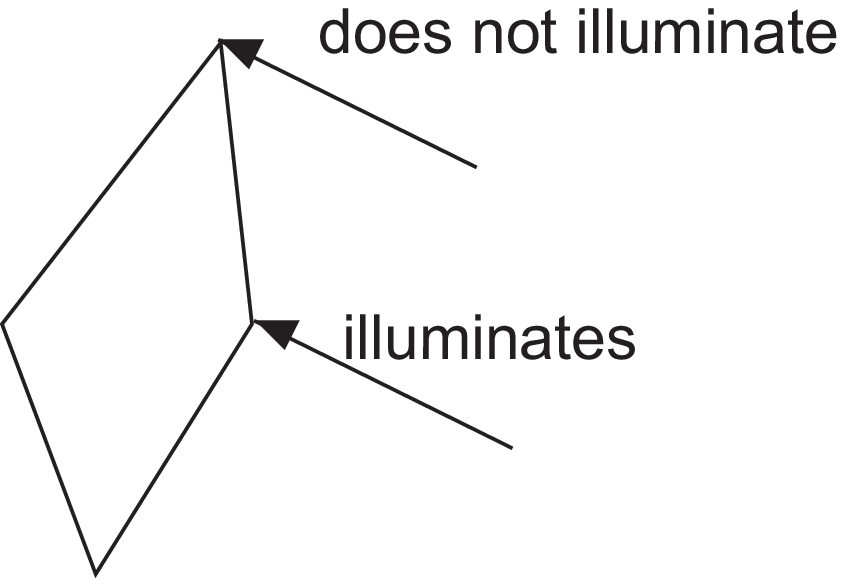}}
\label{illumination-figure}
\end{center}

The minimal number $c(Q)$ of directions needed to illuminate 
the boundary of a polytope $Q$ has been studied by several
authors (see \cite{bib-boltyanski-illumination}). It has been 
conjectured that $2^n$ directions are always sufficient
for an $n$-dimensional polytope $Q$; and that $c(Q) = 2^n$
if and only if $Q$ is affinely equivalent to the $n$-dimensional
cube. 

It is easy to see that if every vertex is illuminated by 
a direction from a given system, then the system 
illuminates all boundary points, so for every polytope $Q$, 
$\dim(Q)\ge 1$ we have $c(Q)\le |\Vert(Q)|$.

\begin{lemma} 
\label{dual-cells-skinny-lemma}
Let $1\le k \le d$. Consider a dual cell $D^k$ and a projection 
$h$ along the corresponding face $F^{d-k}$ of the tiling onto 
a complementary $k$-space. Then $c(D^k) = c(h(D^k)) = |\Vert(D^k)|$.
In other words, dual cells are {\em skinny}, see definition below.
\end{lemma}

\begin{proof} 
This directly follows from lemma \ref{translate_complex_main_lemma}. 
Let $Q$ be any of the polytopes $D^k$, $h(D^k)$. Suppose that 
one direction $u$ illuminates two distinct vertices $v_1$, $v_2$ 
of $Q$. Then $u$ illuminates vertex $v_1$ of polytope $Q + (v_1 - v_2)$,
so for some $\epsilon_1,\epsilon_2 > 0$ we have 
$v_1 + \epsilon_1 u \in \relint(Q)$,  
$v_1 + \epsilon_2 u \in \relint(Q + (v_1 - v_2))$. Let 
$\epsilon = \min\lbrace\epsilon_1,\epsilon_2\rbrace$, then the point 
$v_1 + \epsilon u$ belongs to $\relint(Q)$ and to 
$\relint(Q + (v_1 - v_2))$. By lemma \ref{translate_complex_main_lemma}, 
these relative interiors do not intersect, unless $v_1 = v_2$.
However, vertices $v_1$ and $v_2$ were chosen distinct. 
This contradiction proves the lemma. 
\end{proof}

\begin{lemma}
\label{illumination-conjecture-lemma}
Suppose that $\dim(D) = 3$, where $D$ is a dual cell. 
Then $|\Vert(D)| \le 8$, with equality taking place
if and only if $D$ is affinely equivalent to a $3$-cube.
\end{lemma}

\begin{proof} 
Assume without limiting the generality that $0$ 
is a vertex of $D$. Consider the lattice 
$\Lambda' = {\mathbb Z}(\Vert(D))$. 
Rank of lattice $\Lambda'$ is 3, because $\dim(D) = 3$.
If $|\Vert(D)| \ge 9$, then $\Vert(D)$ contains two 
vertices from the same parity class in $\Lambda'$,
and hence from the same parity class in $\Lambda$.
This contradicts lemma \ref{parity-class-representation-corollary}
on page \pageref{parity-class-representation-corollary}.
Therefore $|\Vert(D)|\le 8$. 

Now, $D$ is a lattice polytope with respect to
lattice $\Lambda'$ and its images under the action
of $\Lambda'$ pack the space $\lin(\Lambda')$. 
Also, $|\Vert(D)| = 8$. It follows that $D$ is a
fundamental parallelepiped for the lattice $\Lambda$.
\end{proof}

Below we will call a polytope $Q$ with $c(Q) = |\Vert(Q)|$ a 
{\em skinny} polytope.
\index{skinny polytope}

\begin{lemma} Faces of dimension at least $1$ of a skinny polytope 
are skinny.
\end{lemma}

\begin{proof}
Enough to prove the lemma for facets and then propagate 
the result by induction. Suppose that $Q$ is a skinny polytope 
and $F\subset Q$ is a facet. Suppose that directions 
$u_1,\dots,u_k\in \lin(F - F)$ illuminate all vertices of 
$F$, where $k < |\Vert(F)|$. Let $n$ be an inward normal 
vector to the facet $F$. Then for small enough $\epsilon > 0$ 
the directions $u_i + \epsilon n$ illuminate all vertices 
$\Vert(F)\subset \Vert(Q)$ on the boundary of $Q$. We can 
use one direction for each of the remaining vertices 
$\Vert(Q)\setminus\Vert(F)$. It follows that $Q$ is not skinny.
\end{proof}

In fact, a more general result holds.

\begin{lemma} \label{subset_ei_lemma}
If $Q$ is a skinny polytope and $X\subset \Vert(Q)$ 
is a subset with $|X| \ge 2$, then $conv(X)$ is skinny.
\end{lemma}

\begin{proof} Choose a point $y\in\relint(\conv(X))$ and 
take the face $F$ of $Q$ so that $y\in\relint(F)$. 
Face $F$ is defined uniquely.

We prove that $\conv(X)\subset F$. If $F = Q$, then the
result is immediate. Suppose that $F$ is a proper face 
of $Q$. Let $H$ be a hyperplane supporting face $F$ on $Q$. Then $X\subset H$. 
Indeed, all points of $X$ are vertices of $Q$ so they belong 
to the same closed halfspace $H^+$ of $H$. On the other hand, 
if there is a point $y'\in\conv(X)$ which belongs to the 
interior of halfspace $H^+$, then (since $y\in\relint(\conv(X))$) 
there is a point on the line going through $y$ and $y'$ 
which belongs to the interior of the other halfspace $H^-$,
which is a contradiction. Hence 
$\conv(X)\subset H\cap P = F$. 

Since $\conv(X)$ and $F$ are star sets with respect 
to $y$,\footnote{A set $A$ is called star set with 
respect to $y$ if any point in $A$ can be connected
with $y$ by a line segment contained in $A$.} we have 
$\relint(\conv(X))\subset \relint(F)$. If the number of directions
sufficient to illuminate $\conv(X)$ is less than $|X|$, then we 
can use the same set of directions plus one direction
for each vertex in $\Vert(F)\setminus X$ to illuminate $F$, 
and therefore $F$ is not skinny. Using the previous lemma,
we derive that $Q$ is not skinny which is a contradiction. 
The lemma is proved.
\end{proof}

%Although the illumination conjecture which claims that 
%$c(Q) \le 2^{\dim(Q)}$ is not resolved yet, we can prove
%that cubes are ``maximal'' skinny polytopes, in the following sense.

\begin{lemma} 
\label{cube-maximality-lemma}
Let $Q=[0,1]^n$, $n\ge 1$, $X \subset \aff(Q)$, so that 
$\conv(Q\cup X)$ has vertex set 
$\Vert(Q)\cup X$. Then 
$\conv(Q\cup X)$ is {\em not} skinny.
\end{lemma}

\begin{proof}
Let $x\in X$, $t$ a vector so that $x + t \in\relint(Q)$, 
chosen not to be parallel to any facet of $Q$. Then light direction
$t$ illuminates one of the vertices of $Q$, and hence one 
of the vertices of $\conv(Q\cup \lbrace x \rbrace)$, which 
means that the polytope $\conv(Q\cup X)$ is not skinny.
\end{proof}

\end{subsection}

\begin{subsection} {A lemma on the projections of polytopes}
\label{projections-of-polytopes-section}
Consider a $d$-dimensional parallelotope $P$ and the projection 
$h$ along a linear space $L^k$ onto a complementary linear 
space $L^{d-k}$. Call faces $F^s$ of $P$ with 
$L^k\subset\lin(F^s - F^s)$ {\em commensurate} 
\index{commensurate faces}
with the projection 
$h$. If $F^s$ is commensurate,
the polytope $h(F^s)$ has dimension $s - k$ and is a face 
of $h(P)$. Different commensurate faces of $P$ have different 
images under $h$. 

\begin{lemma} 
\label{projection-of-relint-lemma}
Let $M$,$N$ be affine spaces, $f:M \to N$ an affine mapping, 
$P\subset M$ a convex set. Then $\relint(f(P)) = f(\relint(P))$.
\end{lemma}

\begin{proof} We may assume that $M = \aff(P)$, $N = f(M)$. 
Then $N = \aff(f(P))$. The set $f(\int(P))$ is open, convex
and dense in $f(P)$, which implies that $f(\int(P)) = \int(f(P))$.
\end{proof}

\end{subsection}

\begin{subsection}{Dual cells affinely equivalent to cubes}
\label{affine-cubes-cells-section}
In this section, we prove some
properties of dual cells
which are affinely equivalent to cubes
of arbitrary dimension. Readers familiar
with Delaunay tilings would note that 
some of these properties hold
almost automatically, if you assume 
that the parallelotope tiling is 
a DV-tiling and the dual cells are
the Delaunay polytopes. 

First, a few notes on terminology. {\em Edge vectors} 
of a polytope are the vectors $v_1 - v_2$ for vertices $v_1,v_2$ 
such that $[v_1,v_2]$ is an edge. A $k$-dimensional nonempty 
proper face of a polytope is called primitive (in relation to the polytope) if 
it is contained in exactly $d-k$ facets of the polytope. 
By an $s$-cube we mean a cube of $s$ dimensions.

\begin{lemma} 
\label{preliminary-cube-subcells-lemma}
Let $D$ be a dual cell affinely equivalent to an $s$-cube. 
Then each face of $D$ is also a subcells. All $1$-subcells 
of $D$ are edges of $D$.
\end{lemma}

\begin{proof}
For every face $D'\subset D$ one can always pick a vector 
$\lambda\in\Lambda$ so that $D' = D \cap (D + \lambda)$.
This, by statement 2 of lemma \ref{translate_complex_main_lemma}, 
implies that $D'$ is a subcell of $D$. 
This proves the first statement of the lemma.

Next, if $[x,y]$ is a dual 
$1$-cell, $x,y\in \Vert(D)$, let $D'$ be the minimal face of $D$ 
which contains both $x$ and $y$. Suppose that $|\Vert(D')| > 2$. 
Then the face $F'$ of the tiling dual to $D'$ is of dimension less 
than $d-1$, and $\frac{x+y}{2}$ is the center of symmetry of $F'$. 
However $\frac{x+y}{2}$ is also the center of symmetry of a
facet corresponding to the dual $1$-cell $[x,y]$ which is 
a contradiction proving the second statement.
\end{proof}

\begin{definition} Let $F$ be a face of a parallelotope $P$ 
of the tiling. The {\em associated collection} $P_F$
of $F$ on the parallelotope $P$ is the set of all faces $F'\subset P$ such that 
$F + \lambda \subset F'$ for some $\lambda\in\Lambda$.
\end{definition}
\index{associated collection}

For example, if $F^{d-2}$ is a quadruple $(d-2)$-face 
of parallelotope $P$, then the associated collection $P_{F^{d-2}}$
contains $4$ translated copies of $F^{d-2}$, $4$ facets of $P$ 
and $P$ itself.

\begin{lemma} \label{description-of-cube-subcells-lemma}
Let $D$ be a dual cell affinely equivalent to 
an $s$-cube, $s\ge 1$. Let $F$ be the corresponding 
face of the tiling, $P$ a parallelotope of the tiling, $F\subset P$, 
and let $t_1,\dots,t_s$ be linearly independent edge vectors of $D$.
Let $f : {\mathbb R}^d \to M$ be a linear projection 
onto a linear space $M$ such that $\ker(f)\subset \lin(F - F)$. 
\footnote{By $\ker$ we denote the null space of a mapping.
\index{ker}
\index{kernel}
}
Then the following statements hold:
\begin{enumerate}
\item $\dim(F) = d - s$
\item All elements of $P_F$ are primitive faces of $P$
\item $\lin(F - F)$ and $\lin(D - D)$ are complementary 
\item $\aff(D) \cap P \subset \bigcup_{G\in P_F} \relint(G)$
\item $f(\aff(D)) \cap f(P) \subset \bigcup_{G\in P_F} \relint(f(G))$
\item The facet vectors of facets in $P_F$ are 
$\lbrace \pm t_1,\dots,\pm t_s \rbrace$
\item All subcells of $D$ are faces of $D$.
\item If $v$ is the center of $P$ and $D=v+[0,t_1]\oplus\dots\oplus[0,t_s]$, 
then $F - ([0,t_1]\oplus\dots\oplus[0,t_s])\subset P$, where
symbol $\oplus$ stands for the direct
Minkowski sum.
\end{enumerate}
\end{lemma}

\begin{proof}
Without limiting the generality, assume that $0$ is a vertex 
of $D$, $P$ is the parallelotope of the tiling with center $0$, 
and $D = [0, t_1] \oplus \dots \oplus [0, t_s]$. By lemma 
\ref{preliminary-cube-subcells-lemma}, the edges of $D$ are its subcells. 
Let $F_1,\dots,F_s$ be the facets of $P$ corresponding 
to dual $1$-cells $[0, t_1],\dots,[0, t_s]$, and let $f_1,\dots,f_s$ 
be the outward normals to those facets of $P$, so that facet
$F_j$ is defined by equation $f_j \cdot x = 1$.  Let 
$L = \lin\lbrace t_1,\dots, t_s \rbrace$.

\medskip

\noindent {\em Inscribing a copy of $D$ 
into the parallelotope.}
We can inscribe a copy of $D$ into the parallelotope $P$ 
by taking $E = \frac{1}{2}\sum_{i=1}^{s} t_i - D$. 
Proof of lemma \ref{translate_complex_main_lemma}
(page \pageref{translate_complex_main_lemma}) 
describes this operation in detail: 
polytope $E$ is denoted by $S$ there.

Vertex $\frac{1}{2}\sum_{i=1}^s \epsilon_i t_i$ 
of $E$ is the center of a $(d-s)$-face $F^{\epsilon_1\dots\epsilon_s}$
in the associated collection $P_F$,
where $\epsilon_i\in\lbrace\pm 1\rbrace$. 
Faces $F^{\epsilon_1\dots\epsilon_s}$ are
all the faces of $P$ equivalent to $F$ (in terms of the proof
of lemma \ref{translate_complex_main_lemma}), in other words,
all $(d-s)$-faces in the associated collection $P_F$.

Vertices $\frac{1}{2}\sum_{i\ne j} \epsilon_i t_i + \frac{1}{2} t_j$
and $\frac{1}{2}\sum_{i\ne j} \epsilon_i t_i - \frac{1}{2} t_j$
differ by $t_j$, the facet vector of $F_j$. Therefore the first 
point belongs to $F_j$, the second to $-F_j$. We have proved
that 
\begin{equation}
\frac{1}{2}\sum_{i=1}^s \epsilon_i t_i \in \epsilon_j F_j.
\end{equation}
Suppose that $\epsilon_j = 1$. By varying $\epsilon_i$ 
for all $i\ne j$, we get only points of $F_j$.
Therefore the vectors between these points
are parallel to $F_j$, so we have $f_j \cdot t_i = 0$.
Next, since $t_j$ is a facet vector of $F_j$ and 
equation $f_j \cdot x = 1$ defines the facet, we
have $f_j \cdot t_j = 2$. Summarizing, we have
\begin{equation}
\label{delta-equation}
f_i \cdot t_j = 2\delta_{ij}
\end{equation}
where $\delta_{ij}$ is the Kronecker delta. It follows
that
\begin{equation}
\label{E-equation}
E = L \cap \lbrace x : |f_i(x)|\le 1 \rbrace
\end{equation}

\medskip

\noindent {\em Proof of statements 1 and 3.}
By lemma \ref{preliminary-cube-subcells-lemma}, 
there are $s$ edges of $D$ incident
with vertex $0$, so there are exactly $s$
facets of $P$ containing $F$, namely, $F_1,\dots,F_s$.
The normal vectors to these facets are $f_1,\dots,f_s$.
Therefore $\lin(F - F) = \cap_{i=1}^s \ker(f_i)$.

Since $\dim(F) = d-s$, the vectors $f_1,\dots,f_s$ 
are linearly independent.
Equation \ref{delta-equation} implies that the linear 
space $\lin(F - F) = \cap_{i=1}^s \ker(f_i)$ is complementary 
to $\lin\lbrace t_1,\dots,t_s\rbrace = \lin(D-D)$. 
This proves that $\dim(D) + \dim(F) = 1$.

\medskip

\noindent {\em Proof of statement 2.}
The cube $D^{\epsilon_1\dots\epsilon_s} = [0, \epsilon_1 t_1] 
\oplus \dots \oplus [0, \epsilon_s t_s]$  
is the dual cell corresponding to $F^{\epsilon_1\dots\epsilon_s}$. 
By lemma \ref{preliminary-cube-subcells-lemma}, there are 
exactly $s$ $1$-subcells of $D^{\epsilon_1\dots\epsilon_s}$  
containing $0$, which means that there are exactly $s$ facets
$$\epsilon_1 F_1,\dots, \epsilon_s F_s$$ of $P$ which contain 
the face $F^{\epsilon_1\dots\epsilon_s}$. This means
that $F^{\epsilon_1\dots\epsilon_s}$ is primitive:
its dimension is equal to $d$ minus the number 
of facets it belongs to.
It implies that all faces which contain $F^{\epsilon_1\dots\epsilon_s}$
in their boundaries are primitive faces of $P$.
Therefore all faces in $P_F$ are primitive.

\medskip

\noindent {\em Proof of statement 4.} 
We have $\aff(D) = L$, and
\begin{equation} 
E = L\cap P
\end{equation}
Indeed, 
on one hand, $E\subset L$ and $E\subset P$, on the other hand, 
$P$ is bounded by inequalities $|f_i(x)\le 1|$, $i=1,\dots,s$ 
and the same set of inequalities defines $E$ in the space $L$, 
therefore $L\cap P \subset E$. We need to prove that
$E \subset \bigcup_{G\in P_F} \relint(G)$. We will prove 
that the relative interior of each face of $E$ is contained
in the set on the right hand side.

Take a face $E'$ of $E$. Its center is also the center
of a uniquely defined dual cell $D'$, 
$0\in D' \subset D^{\epsilon_1\dots\epsilon_s}$.
Let $F' \in P_F$ be the face of the tiling
corresponding to $D'$. It has the same
center of symmetry as $D'$.

We have $E' \subset P$. Since $E'$ and $F'$
share a center of symmetry, it follows that
$E' \subset F'$ (draw a supporting hyperplane
to face $F'$ on $P$ to prove it) and 
$\relint(E')\subset \relint(F')$. This proves
statement 4.

\medskip

\noindent {\em Proof of statement 5.} 
Without limiting the generality, we can 
assume that $L \subset M$, where $M$
is the image space of the linear projection $f$.
Then $f(\aff(D))\cap f(P) = f(L)\cap f(P) =
L \cap f(P)$. We have:
\begin{equation}
E = L \cap f(P).
\end{equation}
Indeed, since $E \subset L$, $E\subset P$,
and $f(E) = E$, we have $E\subset L \cap f(P)$.
On the other hand, $P$ is bounded by
inequalities $|f_i(x)| \le 1$, $i=1,\dots,s$,
and $E$ is equal to the set in 
equation \ref{E-equation}, therefore
$L \cap f(P) \subset E$.

We have $f(E) = E$.
Let $E'$ be a face of $E$. Since
by the previous statement 
$\relint(E')\subset\relint(F')$ for some
face $F' \in P_F$, we have
\begin{equation}
\relint(E') = f(\relint(E')) \subset f(\relint(F')),
\end{equation}
which is equivalent to
\begin{equation}
\relint(E')=\relint(f(E')) \subset \relint(f(F')).
\end{equation}
This establishes statement 5.

\medskip

\noindent {\em Proof of statement 6.}
We have enumerated all facets in the associated 
collection $P_F$: they are $\pm F_1,\dots,\pm F_s$.
Their facet vectors are 
$\lbrace \pm t_1,\dots,\pm t_s \rbrace$. 
This proves statement 6. 

\medskip

\noindent {\em Proof of statement 7}. 
Sufficient to prove that all subcells
of $D$ which contain vertex $0$ are
faces of $D$. Since face $F$ is a 
primitive face of parallelotope $P$,
there are exactly $2^s$ faces 
of $P$ in $\St(F)$. There are also
exactly $2^s$ faces of $D$ 
containing $0$, and all of them
are dual cells, by lemma 
\ref{preliminary-cube-subcells-lemma}.
There are no other subcells
of $D$ containing $0$, because
there is a 1-1 correspondence
between faces $F'$ with $F \prec F' \prec P$
and dual cells $D'$ with $0 \prec D' \prec D$.

\medskip

\noindent {\em Proof of statement 8}. We need to prove that 
$F \oplus ( - [0,t_1]\oplus\dots\oplus[0,t_s])\subset P$. 
We have
\begin{equation}
F \oplus ( - [0,t_1]\oplus\dots\oplus[0,t_s]) = 
\conv\left( \bigcup_{\epsilon_i\in\lbrace 0,1 \rbrace} (F - \sum_{i=1}^s \epsilon_i t_i) \right).
\end{equation}
Indeed, the polytopes on the left and the right hand side 
have the same set of vertices. The polytope on the right 
hand side is contained in $P$ since all polytopes 
$F - \sum_{i=1}^s \delta_i t_i$ are
faces of $P$. This proves statement 8.
\end{proof}

\begin{lemma} 
\label{boundary-section-lemma}
Let $Q$ be a polytope, $\dim(Q) = n \ge 0$, $L\subset\aff(Q)$ 
an affine space, $L\cap \relint(Q)\ne\emptyset$.  Then $L$ 
intersects the relative interior of a face $F \subset Q$ 
where $\dim(F)\le n - \dim(L)$.
\end{lemma}

\begin{proof} If $\dim(L) = 0$ or $1$, then the result 
is trivial since $L$ intersects $P$ or the boundary of $P$ respectively. 
This is the base of induction over $\dim(L)$. If $\dim(L) > 1$, 
consider a proper face $F'$ of $P$ such that 
$L\cap \relint(F')\ne\emptyset$. If $\dim(F') + \dim(L) \le n$, 
the claim is proved. If $\dim(F') + \dim(L) > n$, 
the affine space $L' = L \cap \aff(F')$ intersects 
$\relint(F')$ and has dimension $\dim(L') \ge  \dim(F') + \dim(L) - n$. 
By induction this means that $L'$ intersects the relative 
interior of a face $F''$ of $F'$ with  
$\dim(F'')\le \dim(F') - \dim(L')\le \dim(F') - \dim(F') - \dim(L) + n = n - \dim(L)$.
Since $F''$ is a face of $F'$, it is also a face of $P$. 
This completes the proof of the lemma.
\end{proof}

\end{subsection}

\begin{subsection} {Two subcells affinely equivalent to cubes}
The following theorem and its corollary are key to the 
proof of our main result.
\begin{theorem} 
\label{two_cubes_theorem}
Let $D$ be the dual cell corresponding to face $F$ 
of the tiling. Let $D_1,D_2$ be subcells of $D$. 
Suppose that $D_1$, $D_2$ are affinely equivalent to cubes 
of dimensions $\dim(D_1),\dim(D_2)\ge 1$. Let $h$ be the projection
along $F$ onto a complementary linear space $M$.
\begin{enumerate}
\item If $\aff(D_1) \cap \aff(D_2) \ne\emptyset$, then 
$D_1\cap D_2$ is a nonempty face of both 
$D_1$ and $D_2$ and $\aff(D_1\cap D_2) = \aff(D_1) \cap \aff(D_2)$.
\item If $\aff(D_1) \cap \aff(D_2) = \emptyset$ and 
$\dim(D_1) + \dim(D_2) \ge \combdim(D)$,
then $D_1$, $D_2$ have at least 
$\dim(D_1) + \dim(D_2) - \combdim(D) + 1$ linearly independent 
common edge directions.
\item Projection $h$ restricted to $\aff(D_1\cup D_2)$ is 1-1.
\item The number of linearly independent common
edge directions of $D_1$ and $D_2$ is equal to 
$\dim( \lin(D_1 - D_1) \cap \lin(D_2 - D_2) )$.

\end{enumerate}
\end{theorem}

\begin{proof}
Let $\Lambda_n = {\mathbb Z}_{\aff}(\Vert(D_n))$, 
$K_n = \aff(D_n)$, and let $F_n$ be the face of the tiling 
corresponding to the dual cell $D_n$, $n=1,2$. Let $P$ 
be the polytope of the tiling whose center is $0$, $d_n = \dim(D_n)$.
Let $T_n = \lbrace \pm t_1^{n},\dots,\pm t_{d_n}^{n} \rbrace$ 
where $t_1^{n},\dots,t_{d_n}^{n}$ are the edge vectors of 
$D_n$ (we have $|T_n|=2 d_n$), $v_n$ 
a vertex of $D_n$ so that
$D_n = v_n + [0,t_1^n] \oplus \dots \oplus [0,t^n_{d_n}]$.
Let $C_n$ be the convex hull of $0,1$-combinations 
of $t_1^{n},\dots,t_{d_n}^{n}$,
ie. $C_n = [0,t_1^n] \oplus \dots \oplus [0,t^n_{d_n}]$. 
$C_n$ is a dual cell, a translate of $D_n$. 
Let $\Delta_n = {\mathbb Z}(T_n)$, $L_n = \lin(T_n)$. 
The symbol ${\mathbb Z}(A)$ denotes the lattice generated
by $A$, that is, the set of all finite linear
combinations of elements of $A$; ${\mathbb Z}_{\aff}(A)$ 
denotes the affine lattice generated by set $A$.

For $\lambda\in\Lambda$, $P(\lambda)$ will stand 
for the parallelotope of the tiling
whose center is $\lambda$. Let $F_n$ be the face 
of the tiling corresponding to the dual cell
$D_n$, $n=1,2$. Consider sets 
\begin{equation}
\begin{split}
Q_n = F_n + \lin(D_n - D_n), \cr
U_n = \bigcup_{\lambda\in\Lambda_n} P(\lambda), \cr
n=1,2, \cr
\end{split}
\end{equation}

We begin by proving supplementary results (A) - (E).

\bigskip

\noindent (A) {\it The following statements are equivalent:
\begin{enumerate}
\item $\Lambda_1 \cap \Lambda_2 \ne\emptyset$
\item $K_1 \cap K_2 \ne\emptyset$
\item $\int(Q_1) \cap \int(Q_2) \ne\emptyset$
\item $\int(U_1) \cap \int(U_2) \ne\emptyset$
\end{enumerate}}

\begin{proof}
Indeed, we have: 
$\Lambda_n \subset K_n \subset \int(Q_n) \subset \int(U_n)$, $n=1,2$. 
The first inclusion is trivial. We need to prove the statement 
$K_n\subset \int(Q_n)$.  By statement 3 of lemma
\ref{description-of-cube-subcells-lemma}, the linear spaces 
$\lin(F_n - F_n)$ and $\lin(D_n - D_n)$ are complementary.
Let $c_n$ be the center of symmetry of $F_n$.
We have $c_n\in\relint(F_n)$, therefore for every 
$x\in\lin(D_n - D_n)$, $c_n + x \in\relint(F_n) + \lin(D_n - D_n) = 
\int(\lin(D_n - D_n) + F_n) = \int(Q_n)$. We conclude that 
$K_n=\aff(D_n)=\lin(D_n - D_n) + c_n\subset\int(Q_n)$.
To prove the last inclusion, $\int(Q_n)\subset \int(U_n)$, 
note that by statement 8 of lemma \ref{description-of-cube-subcells-lemma},
the following inclusion holds: $F_n - C_n \subset P(v_n)$, so 
\begin{equation}
\begin{split}
Q_n = F_n + \lin(D_n - D_n) = F_n - C_n + {\mathbb Z}(t_1^n,\dots,t_{d_n}^n) \subset \\
P(v_n) + {\mathbb Z}(t_1^n,\dots,t_{d_n}^n) = P + v_n + {\mathbb Z}(t_1^n,\dots,t_{d_n}^n) = \\
P + \Lambda_n = U_n.
\end{split}
\end{equation}
This implies that $\int(Q_n)\subset\int(U_n)$.

The three inclusions prove that 
$1 \Rightarrow 2 \Rightarrow 3 \Rightarrow 4$. Finally if 
$x\in \int(U_1) \cap \int(U_2)$, then there is a number 
$\epsilon > 0$ so that $B_\epsilon(x)\subset U_2$. Since 
$x\in U_1$, there is a point $\lambda\in\Lambda_1$ so that 
$x\in P(\lambda)$. Take $y\in \int(P(\lambda))$, $|y-x| < \epsilon$.
By choice of $\epsilon$, $y\in U_2$. There is a point 
$\lambda'\in\Lambda_2$ so that $y\in P(\lambda')$.
We therefore have $\int(P(\lambda)) \cap P(\lambda') \ne \emptyset$.
From the face-to-face property of the tiling it follows that 
$P(\lambda) = P(\lambda')$, hence $\lambda=\lambda'$
and $\Lambda_1 \cap \Lambda_2 = \emptyset$.
This proves that $4 \Rightarrow 1$. \end{proof}

\bigskip

\noindent (B) {\it The following numbers are equal: 
$\rank(T_1 \cap T_2)$, $\dim(L_1 \cap L_2)$, 
$d - \min\lbrace \dim(G) : G \in P_{F_1} \cap P_{F_2}\rbrace$.}

\begin{proof}
We prove this in a chain of inequalities. 
Obviously $\rank(T_1 \cap T_2) \le \dim(L_1 \cap L_2)$. 
Next, let $L = L_1 \cap L_2$. By lemma \ref{boundary-section-lemma} 
$L$ intersects the relative interior of a face $G$ of $P$ 
with $\dim(G) \le d - \dim(L)$. By statement 4 of lemma 
\ref{description-of-cube-subcells-lemma}, since $L\subset L_1$,  
$G\in P_{F_1}$. Similarly, since $L\subset L_2$, $G\in P_{F_2}$. 
Thus 
$\dim(L_1\cap L_2) = \dim(L) \le 
d - \min\lbrace \dim(G) : G \in P_{F_1} \cap P_{F_2}\rbrace$.

Finally let $G\in P_{F_1} \cap P_{F_2}$ be a face of 
minimal dimension $k$. By the definition of associate collection,
all facets which contain $G$ belong to both $P_{F_1}$ 
and $P_{F_2}$, and by statement 6 of lemma \ref{description-of-cube-subcells-lemma}
the facet vectors corresponding to those facets belong 
to $T_1$ and to $T_2$. Since by statement $2$ of lemma 
\ref{description-of-cube-subcells-lemma} $G$ is a primitive 
face of $P$, $G$ belongs to exactly $d-k$ facets of $P$. Hence
$P_{F_1} \cap P_{F_2}$ contains at least $d - k$ pairs 
of parallel facets of $P$ and $\rank(T_1\cap T_2) = 
\frac{1}{2}|T_1\cap T_2| \ge d - k = d - \dim(G) = 
d - \min\lbrace \dim(G) : G \in P_{F_1} \cap P_{F_2}\rbrace$ 
which closes the chain of inequalities.
\end{proof}

\bigskip

\noindent (C) ${\mathbb Z}(T_1)\cap {\mathbb Z}(T_2) = 
{\mathbb Z}(T_1\cap T_2)$.

\begin{proof}
We have $\dim(\lin(L_1\cup L_2)) = \dim(L_1) + \dim(L_2) - \dim(L_1\cap L_2)$, $\dim(L_1) = \frac{1}{2}|T_1|$,
$\dim(L_2) = \frac{1}{2}|T_2|$, and $\rank(T_1 \cap T_2) = \frac{1}{2}|T_1 \cap T_2| = \dim(L_1\cap L_2)$.
Together with $|T_1\cup T_2| = |T_1| + |T_2| - |T_1\cap T_2|$ this implies that $\dim(\lin(L_1\cup L_2)) =
\frac{1}{2}|T_1\cup T_2|$. This means that an orientation of $T_1\cup T_2$ (a set of vectors $A$ with
$\frac{1}{2}|T_1\cup T_2| = |A|$ and $T_1\cup T_2 = A \cup -A$) is a basis for $\lin(L_1\cup L_2)$.
Let $A_1\ = A\cap T_1$, $A_2 = A\cap T_2$. $A_1$, $A_2$ are bases for ${\mathbb Z}(T_1)$, ${\mathbb Z}(T_2)$
respectively. 

We clearly have ${\mathbb Z}(T_1\cap T_2) \subset {\mathbb Z}(T_1)\cap {\mathbb Z}(T_2)$. On the other hand,
take $z\in {\mathbb Z}(T_1)\cap {\mathbb Z}(T_2)$. Since $z\in{\mathbb Z}(T_1)$, we have 
$z = \sum_{a\in A_1} z_a a$. Similarly $z = \sum_{a\in A_2} z_a' a$. However, $A$ is a basis for $L_1 \cup L_2$,
therefore these two representations are the same, so $z_a = 0$ for $a\in A_1\setminus A_2$, $z_a' = 0$
for $a\in A_2\setminus A_1$, and $z_a = z_a'$ for $a\in A_1\cap A_2$. Hence $z\in{\mathbb Z}(A_1 \cap A_2)=
{\mathbb Z}(T_1 \cap T_2)$.
\end{proof}

\bigskip

\noindent (D) $\dim(h(L_1) \cap h(L_2)) = \dim(L_1 \cap L_2)$. 

\begin{proof}
%Since linear spaces
%$\lin(D_1 - D_1)$ and $\lin(F_1 - F_1)$ are complementary, 
%and since $\lin(F - F) \subset \lin(F_1 - F_1)$, 
%$T_1 \cap T_2 \subset \lin(D_1 - D_1)$, 
%we have $\rank(h(T_1 \cap T_2)) = \rank(T_1 \cap T_2)$, and 
%
%
Let $N = h(L_1) \cap h(L_2)$. By lemma 
\ref{boundary-section-lemma} $N$ intersects 
the relative interior of a face $A$ of $h(P)$ 
with $\dim(A) \le d - \dim(F) - \dim(N)$. By statement 5 
of lemma \ref{description-of-cube-subcells-lemma}, 
since 
$N\subset h(L_1)$,  $A \in h(P_{F_1})$.  
But $N\subset h(L_2)$, therefore $A\in h(P_{F_2})$. 
Thus $\dim(h(L_1)\cap h(L_2)) = \dim(N) \le 
d - \dim(F) - \min\lbrace \dim(X) : X \in h(P_{F_1}) \cap h(P_{F_2})\rbrace$.
From section \ref{projections-of-polytopes-section} 
and paragraph (B) of this proof, we get 
\begin{equation}
\begin{split}
\label{two-collections-equation}
d - \dim(F) - \min\lbrace \dim(X) :  X \in h(P_{F_1}) \cap h(P_{F_2})\rbrace = \\
d - \min\lbrace \dim(G) : G \in P_{F_1} \cap P_{F_2}\rbrace = \dim(L_1 \cap L_2).
\end{split}
\end{equation}
(since all faces in $P_{F_1}$ and $P_{F_2}$ 
are commensurate with the projection $f$).

This produces $\dim(h(L_1)\cap h(L_2)) \le \dim(L_1 \cap L_2)$.
On the other hand, $L_1 \cap L_2 \subset L_1$
and the null space of projection $f$ is contained
in $\lin(F_1 - F_1)$, which is complementary to
$L_1$ by lemma 
\ref{description-of-cube-subcells-lemma}, 
statement 3. Therefore $\dim(h(L_1 \cap L_2)) = 
\dim(L_1 \cap L_2)$. We also have 
$h(L_1 \cap L_2) \subset h(L_1) \cap h(L_2)$, 
so, combining these two results, we have 
$\dim(L_1 \cap L_2) \le \dim(h(L_1) \cap h(L_2))$. 

We have proved that $\dim(h(L_1) \cap h(L_2)) = \dim(L_1 \cap L_2)$.
\end{proof}

\bigskip

\noindent (E) {\it Mapping $f$ acts on the linear space 
$\lin(L_1 \cup L_2)$ bijectively.}

\begin{proof}
From $D$, we have 
\begin{equation*}
\begin{split}
\dim(h(\lin(L_1 \cup L_2))) =& \\
\dim(\lin(h(L_1 \cup L_2))) =& \\
\dim(\lin(h(L_1) \cup h(L_2))) =& \\
\dim(h(L_1)) + \dim(h(L_2)) - \dim(h(L_1) \cap h(L_2)) =& \\
\dim(L_1) + \dim(L_2) - \dim(L_1 \cap L_2) =& \\
\dim(\lin(L_1 \cup L_2)).
\end{split}
\end{equation*}
This proves our statement.
\end{proof}

\bigskip

We are now ready to prove the theorem. First we 
prove statement 1. We are given the condition
that $K_1\cap K_2\ne\emptyset$. Then, by the chain 
of equivalences proved above, $\Lambda_1\cap
\Lambda_2 \ne\emptyset$.  Without limitation of generality 
we may assume that $0\in\Lambda_1 \cap \Lambda_2$,
so $\Lambda_n = {\mathbb Z}(T_n)$, $K_n = L_n$, 
$\dim(K_1\cap K_2) = \rank(\Lambda_1\cap\Lambda_2)$. The lattice
$\Lambda_1\cap\Lambda_2$ is generated by 
$T_1\cap T_2$.

We now prove that the affine cubes $D_1$ and $D_2$ 
have a common face of dimension $\dim(K_1\cap K_2)$. 

Before giving a general proof,
we consider a special case
when $\dim(K_1) = \dim(K_2) = 2$, affine spaces $K_1$ 
and $K_2$ intersect over a point $z$, and parallelograms 
are located as shown in figure \ref{2-parallelograms-illumination-figure}.
We are proving that such a configuration is impossible. 
Choose vertices $v_1\in D_1$ and $v_2\in D_2$ so that 
vectors $v_1 - z$ and $v_2 - z$ illuminate
vertices $v_1$, $v_2$ of parallelograms $D_1$, $D_2$ respectively. 
For some $\epsilon > 0$ we have 
$v_1 + \epsilon (v_1 - z) \in \relint(D_1)$ and 
$v_2 + \epsilon (v_2 - z) \in \relint(D_2)$.
The quadrilateral $Q$ with vertices $v_1$, $v_2$,  
$v_1 + \epsilon (v_1 - z)$, $v_2 + \epsilon (v_2 - z)$
is shown on the right hand side of the picture. 
We then have $\relint(Q)\subset\relint(\conv(D_1 \cup D_2)))$.
This is because affine spaces $K_1$, $K_2$ are complementary. 
From the right hand side of the picture, we can also 
see that $\frac{1}{2}(v_1 - z + v_2 - z)$ illuminates 
both vertices $v_1$, $v_2$ on $Q$ and, consequently, 
on $\conv(D_1 \cup D_2)$ which contradicts the fact 
that polytope $D^4$ is skinny  (ie., no light direction 
illuminates two vertices) and that the convex hull 
of a subset in $\Vert(D^4)$
of dimension at least $1$ is also skinny.

Now we can prove statement 1 in the general case. 
Let $K=K_1\cap K_2 = L_1 \cap L_2$. The (unoriented)
basis in $K$ can be chosen as $T_1 \cap T_2$.
If at least one of $D_1$ and $D_2$ intersects $K$, 
eg. $D_1\cap K\ne\emptyset$,  then $D_1\cap K$
is a face of $D_1$. By lemma \ref{subset_ei_lemma}
on page \pageref{subset_ei_lemma},
polytope $\conv((D_1 \cap K) \cup D_2)$ is skinny,
and by lemma \ref{cube-maximality-lemma} we should 
$D_1\cap K\subset D_2$ (otherwise the polytope
is not skinny).

It follows that $D_2\cap K\ne\emptyset$, therefore
by a similar argument $D_2 \cap K\subset D_1$.
We have $D_1\cap K = D_2\cap K = D_1\cap D_2$.
Since $D_1\cap D_2$ is affinely equivalent to a cube with 
the same edge set $T_1\cap T_2$, we have proved
statement 1.

Suppose now that $D_1\cap K = D_2\cap K = \emptyset$. 
We prove that it is impossible by switching
to certain faces $D_1''$, $D_2''$ of cubes 
$D_1$, $D_2$ to come to a contradiction. 
We will observe that $\conv(D_1'' \cup D_2'')$ 
is not skinny, like in the example we've just
considered.
 
We have $K_1\not\subset K_2$ and $K_2\not\subset K_1$,
therefore $T_2 \setminus T_1 \ne\emptyset$ and 
$T_1 \setminus T_2 \ne\emptyset$.
Let $D_1'$ be any face of $D_1$ whose set of edge vectors 
is $T_1\setminus T_2$, and let $D_2' = D_2$. Since 
$T_1\setminus T_2\ne\emptyset$, we have $\dim(D_1') > 0$. 
Since the set of edge vectors of $D_1'$
is $T_1\setminus T_2$, the set of edge vectors of $D_2'$ 
is $T_2$, we know that the affine lattices ${\mathbb Z}_{\aff}(\Vert(D_1'))$,
${\mathbb Z}_{\aff}(\Vert(D_2'))$ are translates 
of lattices ${\mathbb Z}(T_1\setminus T_2)$ and ${\mathbb Z}(T_2)$
by vectors from  ${\mathbb Z}(T_1 \cup T_2)$.  

In paragraph (C) we proved that 
$\rank(T_1\cup T_2) = \frac{1}{2}|T_1\cup T_2|$, hence 
the lattices ${\mathbb Z}(T_1\setminus T_2)$
and ${\mathbb Z}(T_2)$ are complementary
in ${\mathbb  Z}(T_1 \cup T_2)$.

It follows that
the intersection of ${\mathbb Z}_{\aff}(\Vert(D_1'))$
and ${\mathbb Z}_{\aff}(\Vert(D_2'))$ consists of one point, 
which we assume to be $0$ (without limiting
the generality). Note that $0\in K$, and that linear 
spaces $\lin(D_1' - D_1')$ and $\lin(D_2' - D_2')$ 
have a trivial intersection.

Next, for $n=1,2$ let $D_n''$ be the face of $D_n'$ 
of minimal dimension whose affine hull contains $0$. Since
we have assumed that $D_n\cap K=\emptyset$, $0$ cannot 
be a vertex of $D_n$, therefore $\dim(D_n'') > 0$. 
We can choose a vertex $v_n$ of $D_n''$ so that
light direction $v_n$
illuminates the boundary of polytope $D_n''$ at $v_n$.
It follows that light direction 
$v_1 + v_2$ illuminates both vertices $v_1$, $v_2$ 
on $\conv(D_1'' \cup D_2'')$. 
But the polytope $\conv(D_1'' \cup D_2'')$ is skinny
by lemma \ref{subset_ei_lemma}, which is the
desired contradiction.

\bigskip

Next we prove statement 2 of the theorem.
We are given the condition $K_1\cap K_2 = \emptyset$. 
Then $\int(Q_1) \cap \int(Q_2) = \emptyset$.  Since 
the sets $Q_1$, $Q_2$ are convex, there is a hyperplane 
$H$ separating them. We have  $F\subset F_1 \cap F_2$, 
hence $F\subset H$. The affine spaces $K_1$, $K_2$
are contained in the interiors of $Q_1$ and $Q_2$ 
respectively, so they belong to different open 
half-spaces defined by $H$. 

Hyperplane $h(H)$ in $M$ strictly 
separates $h(K_1)$ and $h(K_2)$ (this means they are 
contained in different open half-spaces of $h(H)$), 
because $F\subset H$ and $h$ is the projection
along $F$ onto the complementary linear
space $M$.

Since $F\subset F_n$ and $F_n$ 
is complementary to $K_n$, $\dim(K_n) = \dim(h(K_n))$. 
Let $s = \dim(h(K_1)) + \dim(h(K_2)) - \dim(M)$.
Since $\dim(M) = \combdim(D)$, by the hypopaper 
of the theorem we have $s\ge 0$. Since $h(K_1) \cap h(K_1) 
= \emptyset$, linear spaces $h(L_1) = h(K_1) - h(K_1)$ and
$h(L_2) = h(K_2) - h(K_2)$ have at least $s+1$-dimensional 
intersection, ie. $\dim(h(L_1) \cap h(L_2)) \ge s+1$.
We have proved in paragraph (D) that 
$\dim(h(L_1) \cap h(L_2)) = \dim(L_1 \cap L_2)$, 
and in paragraph (B)
that $\dim(L_1 \cap L_2) = \rank(T_1 \cap T_2) = 
\frac{1}{2}|T_1 \cap T_2|$. Therefore 
$D_1$ and $D_2$ have at least $s+1$ linearly independent 
edge directions. 

\bigskip

To prove statement 3, note that it is equivalent to the following: 
\begin{equation}
\dim(\aff(K_1 \cup K_2))=\dim(\aff(h(K_1) \cup h(K_2))).
\end{equation}
If $K_1 \cap K_2 \ne \emptyset$, then this statement 
was already proved in paragraph (E). If $K_1 \cap K_2 = \emptyset$,
then $\dim(\aff(K_1 \cup K_2)) = 1 + \dim(\lin(L_1 \cup L_2)) = 
1 + \dim(\lin(h(L_1) \cup h(L_2))) = 
\dim(\aff(h(K_1) \cup h(K_2)))$. The middle equality 
is found in paragraph (E).

\bigskip

Statement 4 directly follows from the result of
paragraph (B), that 
$\dim(L_1\cap L_2) = \rank(T_1 \cap T_2)$.
This completes the proof of the theorem.
\end{proof}

%
% IDEA: the above lemma may possibly be used to prove that $\dim(D) \le \combdim(D)$.
% Suppose that it can be generalized to an arbitrary number of cube dual subcells (somehow).
% Then take all edges of $D$ exiting a vertex. They are dual cells. Linear space that
% is spanned by their edge vectors is mapped 1-1 by the natural projection (analogous to statement
% that $\dim(L_1 \cap L_2) = \dim({\tilde L_1} = {\tilde L_2})$). Therefore we have the upper bound
% for the dimension of $D$!! 
%
% We are using condition that $\dim(lin(L_1\cup L_2)) = \dim(L_1) + \dim(L_2) - \dim(L_1 \cap L_2)$.
% For many linear spaces, ``inclusion-exclusion'' formula must be used.
%
% Maybe this is a bad idea. 
%

\begin{figure}
\begin{center}
\resizebox{150pt}{!}{\includegraphics[clip=true,keepaspectratio=true]{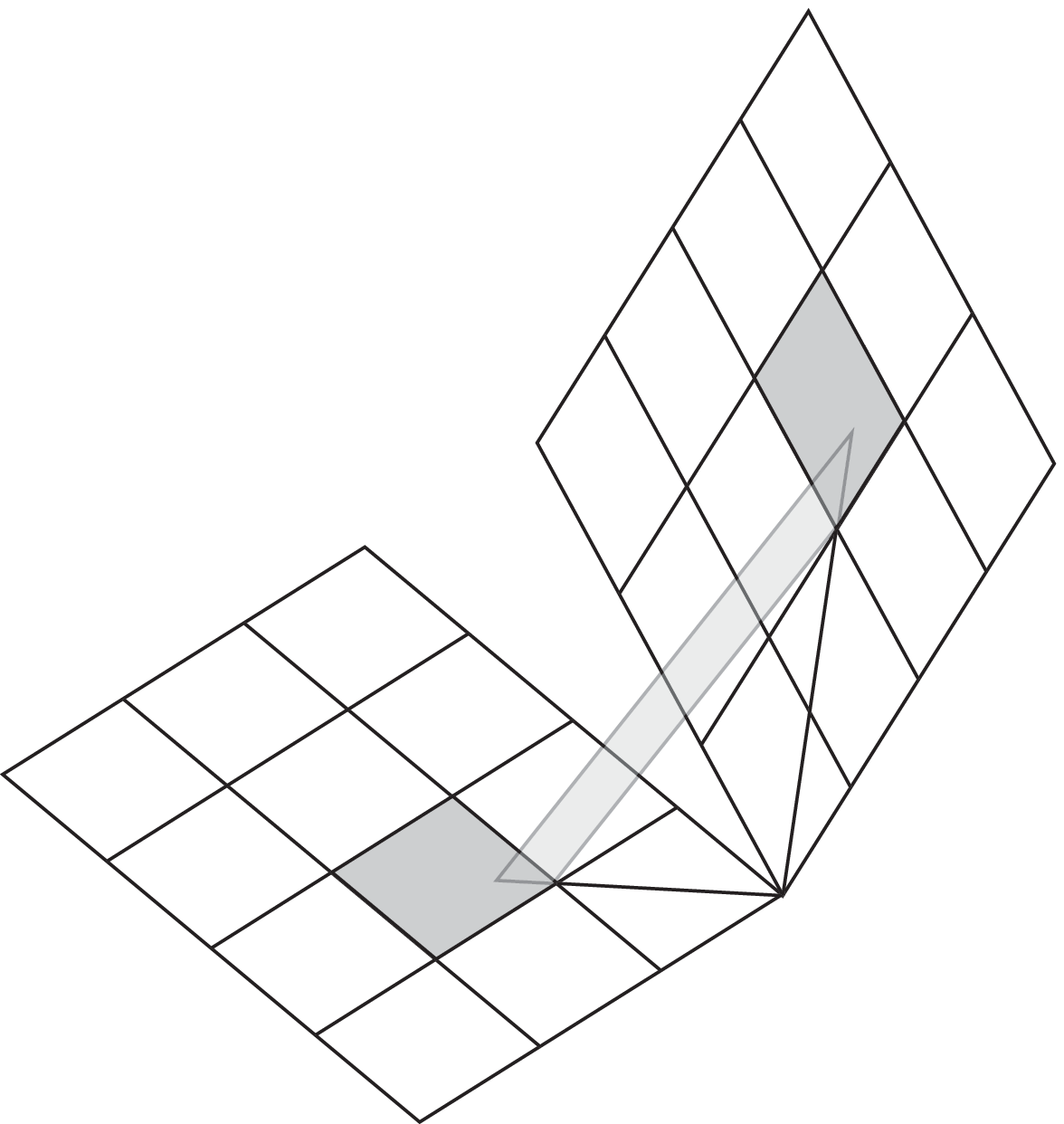}}
\hspace{40pt}
\resizebox{180pt}{!}{\includegraphics[clip=true,keepaspectratio=true]{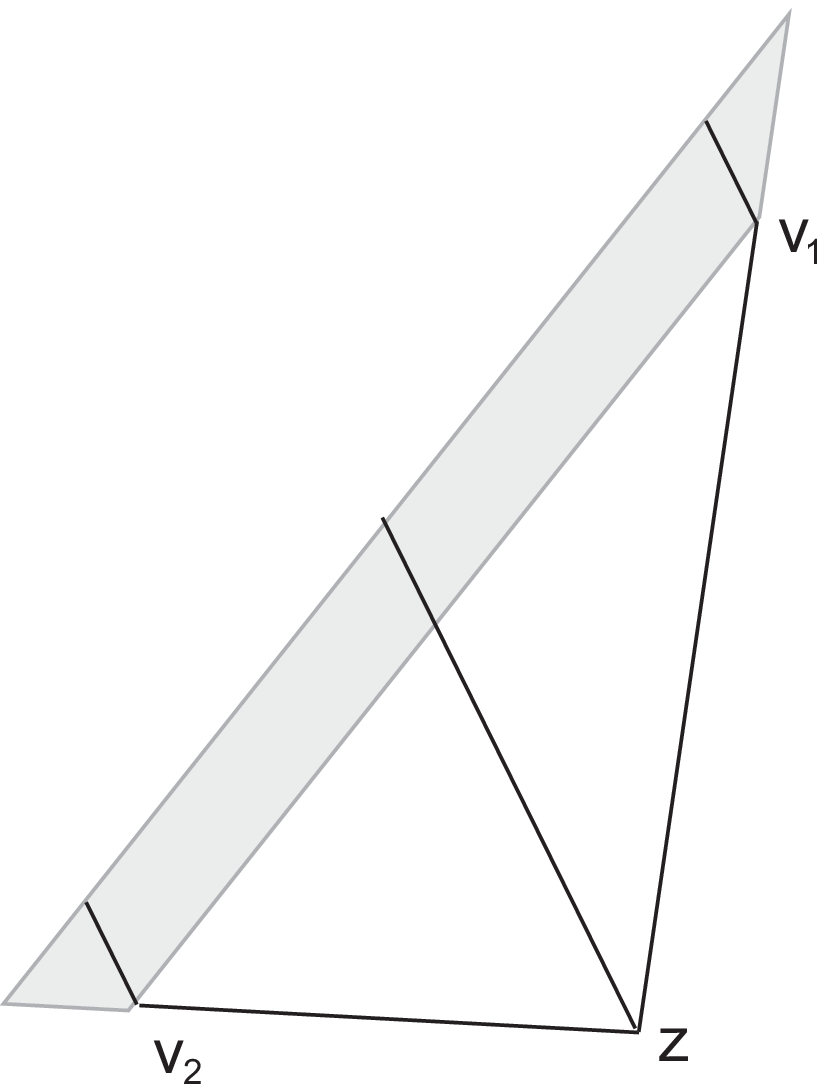}}
\caption{Two parallelograms in complementary 2-spaces}
\label{2-parallelograms-illumination-figure}
\end{center}
\end{figure} 

\begin{corollary} 
\label{two_parallelograms_corollary}
Let $D^4$ be a dual $4$-cell corresponding to a face 
$F$ of the tiling. Then two distinct parallelogram 
subcells $\Pi_1,\Pi_2\subset D^4$ are either
\begin{enumerate}
\item complementary: $\Pi_1 \cap \Pi_2 = \lbrace v \rbrace$ and 
$\dim \aff (\Pi_1\cup \Pi_2) = 4$,
\item adjacent: $\Pi_1 \cap \Pi_2 = I$, where $I$ is an edge 
(a dual combinatorial 1-cell) and $\dim \aff (\Pi_1\cup \Pi_2) = 3$,
\item translate: $\Pi_1 = \Pi_2 + t$ where $t\notin \lin(\Pi_1 - \Pi_1)$ 
and $\dim \aff (\Pi_1\cup \Pi_2) = 3$, or
\item skew: $\Pi_1 \cap \Pi_2 = \emptyset$, there is exactly 
one pair of edges of $\Pi_1$ parallel
to a pair of edges of $\Pi_2$, and $\dim \aff (\Pi_1\cup \Pi_2) = 4$.
\end{enumerate}
Moreover, if $f$ is the projection along $F$ 
onto a complementary $4$-space, then 
$f$ is 1-1 on the affine space $\aff(\Pi_1 \cup \Pi_2)$.
\end{corollary}
\index{complementary parallelograms}
\index{adjacent parallelograms}
\index{translate parallelograms}
\index{skew parallelograms}
\end{subsection}
\end{section}

\begin{section}{Proof of theorem \ref{main_theorem}}
\label{proof-section}

In this section, we complete the proof of our
main result:

\medskip

\noindent{\bf Theorem \ref{main_theorem}} {\em The Voronoi conjecture 
is true for $3$-irreducible parallelotope tilings.}

\medskip

By theorems \ref{equivalence-proposition}
and \ref{coherent-face-proposition}
(pages \pageref{equivalence-proposition}
and \pageref{coherent-face-proposition}), proving
the theorem is equivalent to checking 
that each parallelogram dual cell is coherent
with respect to each dual $4$-cell that 
contains it.

Let $D^4$ be a dual cell in a 3-irreducible tiling.
We will prove that the collection ${\cal R}$ of incoherent 
parallelogram subcells in $D^4$ is empty. If ${\cal R}$
is not empty, what can we say about it? 
\begin{itemize}
\item Each vertex of a parallelogram in ${\cal R}$ 
belongs to at least one other parallelogram in ${\cal R}$
(theorem \ref{exposed-vertex-theorem} on page
\pageref{exposed-vertex-theorem}).
\item Each pair of parallelograms in ${\cal R}$
is either complementary, adjacent, translate, 
or skew (corollary \ref{two_parallelograms_corollary}
on page \pageref{two_parallelograms_corollary}).
\end{itemize}

We suppose that $|{\cal R}| \ge 1$. In each of 
the following cases, which cover all possibilities, 
we will get a contradiction. 
\begin{enumerate}
\item $|{\cal R}| = 1$
\item $D^4$ is asymmetric, $|{\cal R}| \ge 2$ and at
least one pair of parallelograms in ${\cal R}$ is  
adjacent, translate, or skew.
\item $D^4$ is centrally symmetric and $|{\cal R}| \ge 2$.
\item $D^4$ is asymmetric, $|{\cal R}| \ge 2$ and all
pairs of parallelograms in ${\cal R}$ are complementary.  
\end{enumerate}

Case 1 is the easiest; in fact, 
its impossibility follows
directly from theorem \ref{exposed-vertex-theorem}
on page \pageref{exposed-vertex-theorem}.
Cases 2 and 3 are similar in that there is some
extra geometric information available about $D^4$.
Either we know that $D^4$ is centrally symmetric,
or that the system of parallelograms in $D^4$
has a ``singularity'' (a pair of 
adjacent, translate, or skew parallelograms).

In case 4, we do not have this much
geometric information, however 
we have very homogeneous combinatorial conditions
on the system of parallelograms ${\cal R}$. 
Namely, (1) each two parallelograms intersect over
exactly one vertex, and (2) each vertex of 
a parallelogram belongs to at least one
other parallelogram. It allows for some
graph-theoretic methods to be applied.

The next section proves
some technical lemmas. The following 
four sections deal with cases 1-4.

\begin{subsection}{Three technical lemmas}
\begin{lemma} \label{technical-prism-lemma}
Let $C$ be a triangular prism. Let $x_1,x_2,x_3$ be distinct 
vertices of $C$. Then either $\conv\lbrace x_1,x_2,x_3\rbrace$ 
is a triangular facet of $C$, or among $x_1,x_2,x_3$ there are 
two diagonally opposite vertices of a parallelogram face of $C$.
\end{lemma}

{\it Proof }is obtained by inspection. $\Box$

\begin{lemma} 
\label{cube-and-prism-corollary}
If $D^3$, $D$ are dual cells, $D^3\subset D$, $C$ a triangular prism 
with $\Vert(D^3) \subset \Vert(C)\subset\Vert(D)$, 
then $D^3 = C$. 
\end{lemma}

\begin{proof}
Note that a diagonal $I$ of 
a parallelogram face $\Pi$ of $C$
cannot be a dual edge, because 
by lemma \ref{centrally_symmetric_subsets_lemma}
there is a dual subcell $D'\subset D$
with $\Vert(\Pi)\subset\Vert(D')$,
with the same center of symmetry as $\Pi$.
Since $\combdim(D') \ge 2$, the corresponding
face of the tiling is of dimension at most $(d-2)$.
However, it has the same center 
of symmetry as the $(d-1)$-face
corresponding to $I$, which is 
a contradiction.

Now we can prove the lemma.
Since $\Vert(D^3)\subset\Vert(C)$, we have
$|\Vert(D^3)|\le 6$, so $D^3$ can only be a triangular prism, 
a simplex, an octahedron, or a pyramid over a parallelogram. 
The octahedron is excluded because it has the same number
of vertices as $C$, but is not a triangular prism.

If $D^3$ is either a simplex or a pyramid, then by 
the previous lemma combined with the argument
at the beginning of this proof,
each triangular face of $D^3$ is a face 
of $C$, but $C$ has only $2$ triangular faces whereas $D^3$ has $4$.
Therefore $D^3$ is a prism, so $D^3 = C$.
\end{proof}

\begin{lemma}
\label{4-cell-classification-lemma-prism}
Consider a $3$-irreducible tiling.
Let $D^4$ be a dual $4$-cell. Suppose that there is a triangular 
prism $C$ with $\Vert(C)\subset \Vert(D^4)$. 
Suppose also that at least one parallelogram 
face of $C$ is a subcell of $D^4$, and each
parallelogram face of $C$ is either a dual cell, or
is inscribed into an octahedral dual cell.

Then the following statements hold:

(I) Not all parallelogram faces of $C$ are dual cells.

(II) Moreover, exactly $2$ parallelogram 
faces of $C$ are subcells,
$|\Vert(D^4)| = 8$ and $D^4$ is affinely equivalent 
to the following polytope (columns are coordinates
of vertices):
\medskip
\begin{spacing}{1}
\begin{equation}
\label{8-vertex-dual-cell-equation}
\conv\left[ 
\begin{array}{cccccccc} 
0 & 1 & 1 & 1 & 0 & 0 & 0 & 1 \\
0 & 0 & 1 & 1 & 0 & 1 & 1 & 1 \\
0 & 0 & 0 & 0 & 1 & 1 & 1 & 1 \\
0 & 0 & 0 & 1 & 0 & 0 & 1 & 1
\end{array} 
\right]
\end{equation}
\end{spacing}
\end{lemma}

\begin{proof} 
Let $T_1$, $T_2$ be 
triangular faces of $C$, $\Pi_1,\Pi_2,\Pi_3$ parallelogram 
faces. 

Note that all edges of $C$ are dual cells.
Indeed, each parallelogram face $\Pi_k$ of $C$ can
be inscribed into a centrally symmetric
subcell $D$ of $D^4$ (by lemma 
\ref{centrally_symmetric_subsets_lemma}).
By the condition of the lemma,
either $D = \Pi_k$, or $D$ is an octahedron.
In each case, edges of $\Pi_k$ are dual cells.

We prove that 
$T_1$, $T_2$ are dual cells. Indeed, all edges of $C$ 
are dual cells, in particular the edges of triangles 
$T_1$, $T_2$ are dual $1$-cells. Now let $t$ be a vector 
in the lattice of the tiling with $T_1 = T_2 + t$. We have 
$T_1 \subset D^4 \cap (D^4 + t)$. By lemma 
\ref{translate_complex_main_lemma}, $D=D^4 \cap (D^4 + t)$ 
is a dual $2$-or $3$-cell. If $D$ is a dual $2$-cell, 
then it can only be the triangle $T$. If $D$ is a dual 
$3$-cell, then $D$ is either a simplex, an octahedron, 
a triangular prism, or a pyramid over a parallelogram. 
In each case, a cycle of $3$ edges  on $D$ is the boundary 
of a dual $2$-cell, so $T$ is a subcell of $D$ and 
therefore a dual cell.

\bigskip

(A). {\em Definition of the complex $K$.}
Consider the following abstract complex:
\begin{equation} 
\partial D^4 = \lbrace D \subset D^4 : \text{ D is a subcell }, D \ne D^4 \rbrace
\end{equation}
By definition, this complex is dual to 
$\St(F^{d-4})\setminus \lbrace F^{d-4} \rbrace$,
where $F^{d-4}$ is the face of the tiling
corresponding to $D^4$. Therefore it is 
combinatorially isomorphic to a polyhedral subdivision
of the $3$-sphere. 

We will now define complex $K$ as a subdivision
of $\partial D^4$, in the following way.

Each of the vertex sets $\Vert(\Pi_k)$ is a centrally 
symmetric subset of $\Vert(D^4)$, hence by lemma
\ref{centrally_symmetric_subsets_lemma} there is a centrally 
symmetric dual subcell $E_k \subset D^4$ with
$\Vert(\Pi_k)\subset \Vert(E_k)$. By the assumption
in the statement of the lemma, either 
\begin{enumerate}
\item $E_k = \Pi_k$, or 
\item $E_k$ is an octahedron (we say that $\Pi_k$ is inscribed into the 
octahedron $E_k$).
\end{enumerate}
If $E_k$ is an octahedron, then we subdivide it
into two pyramids with base $\Pi_k$, as shown
in figure \ref{splitting-octahedron-figure}). 

Like $\partial D^4$, the new complex 
$K$ is combinatorially
isomorphic to the subdivision
of the $3$-dimensional sphere. 
Complex $K$ is important to us because
faces of $C$ are present among its
cells.
\begin{figure}
\begin{center}
\resizebox{200pt}{!}{\includegraphics[clip=true,keepaspectratio=true]
{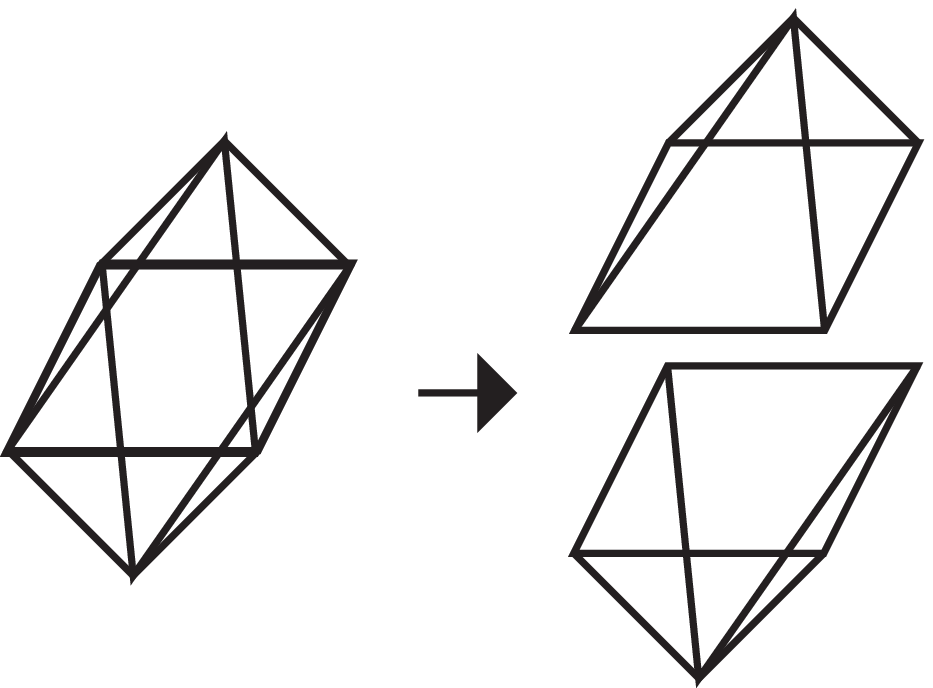}}
\caption{Splitting the octahedral cell}
\label{splitting-octahedron-figure}
\end{center}
\end{figure} 

\bigskip

(B). {\em Application of Alexander's separation theorem.}
We now use the following theorem:
\begin{theorem}(Corollary on page 338 of paper 
\cite{bib-alexander-proof-and-extension})
\label{separation-theorem}
Any closed $2$-chain in $K$ bounds exactly 
two open $3$-chains. Moreover, these two $3$-chains have 
only the elements of the $2$-chain in common.
\end{theorem}
\index{Alexander's separation theorem}
\index{separation theorem}
We define $2$-chain $B$ in the complex $K$ as the sum 
of the cellular $2$-chains defined by faces of $C$,
namely, $T_1$, $T_2$, $\Pi_1$, $\Pi_2$, $\Pi_3$.
Chain $B$ is closed. 
Therefore, by theorem \ref{separation-theorem}
above, $B$ bounds exactly two open $3$-chains in $K$, $K_1$ 
and $K_2$ where $K_1 + K_2 \equiv K \pmod 2$. Moreover, any 
common element of $K_1$ and $K_2$ is contained in $B$.

\bigskip

(C). {\em Counting vertices.}
We define sets $U$, $V$ of vertex pairs. If $D^4$ 
is centrally symmetric, then
\begin{equation}
\begin{split}
U = \left\lbrace (y_1,y_2)\in \Vert(K_1\setminus B) 
\times \Vert(K_2\setminus B): y_1 \ne *(y_2) \right\rbrace, \\
V = \left\lbrace (y_1,y_2)\in \Vert(K_1\setminus B) 
\times \Vert(K_2\setminus B): y_1 = *(y_2) \right\rbrace
\end{split}
\end{equation}
where $*$ is the central symmetry of $D^4$.
If $D^4$ is asymmetric, then we define
\begin{equation}
\begin{split}
U = \Vert(K_1\setminus B) \times \Vert(K_2\setminus B), \\
V = \emptyset
\end{split}
\end{equation}

The following relationship establishes 
a 1-1 correspondence between set $U$ 
and the collection of those parallelograms $\Pi_k$ which 
are inscribed into octahedra $E_k$:
\begin{equation}
\label{matching-equation}
\begin{split}
(y_1,y_2) \to \Pi_k \\
\text{if and only if} \\
E_k = \conv( \Pi_k \cup \lbrace y_1,y_2 \rbrace ).
\end{split}
\end{equation}

First, we prove that for each pair $(y_1, y_2)\in U$,
there is a parallelogram $\Pi_k$ which satisfies
condition \ref{matching-equation}.
Each pair $(y_1, y_2)\in U$ is a diagonal pair of vertices 
of a centrally symmetric dual cell $D \subset D^4$,
$D \ne D^4$.
However, there is no element of complex $K$ which contains 
both $y_1$ and $y_2$, for, if $y_1,y_2 \in D^3$, then, 
since $K \equiv K_1 + K_2 \pmod 2$, either $D^3 \in K_1$ 
or $D^3 \in K_2$. In the first case, point $y_2$ does not 
belong to $(K_2 \setminus B)$, because  $(K_2 \setminus B) 
\cap K_1 = \emptyset$. The second case is treated similarly. 
Therefore $y_1$ and $y_2$ can only be diametrically opposite 
vertices of an octahedron $E_k$. Since $y_1,y_2$ were chosen
not to belong to $B$, we have and $y_1,y_2\notin\Vert(\Pi_k)$,
so $E_k = \conv( \Pi_k \cup \lbrace y_1,y_2 \rbrace )$.

Next, we prove that each parallelogram $\Pi_k$,
inscribed into octahedral dual cell $E_k$, 
corresponds to some pair of vertices $(y_1,y_2) \in U$.
If a parallelogram $\Pi_k$ extends to the 
octahedron $E_k$, then the two points $\lbrace y_1,y_2 \rbrace = 
\Vert(E_k) \setminus \Vert(\Pi_k)$ cannot be vertices of $C$ (this 
is checked directly, using the fact that a diagonal of a parallelogram 
face of $C$ cannot be a dual $1$-cell). 

Pyramids $\conv(\Pi_k \cup\lbrace y_1 \rbrace)$ and
$\conv(\Pi_k \cup\lbrace y_2 \rbrace)$ are elements
of complex $K$. They belong to 
different chains $K_1$, $K_2$, because the cell $\Pi_k$ 
belongs to the common boundary 
$B$ of $K_1$ and $K_2$. Hence, renumbering, if necessary, 
we have $y_1\in K_1 \setminus B$, $y_2 \in K_2 \setminus B$. 
If $*$ is the central symmetry of $D^4$, then
we have $y_1\ne *(y_2)$ because $y_1$, $y_2$ are diametrically opposite 
vertices of a centrally symmetric $3$-subcell $E_k$ of $D^4$.

Therefore $|U|$ is equal to the number of parallelogram 
faces of $C$ which are inscribed into octahedral 
dual cells and $|U|\le 2$.
Let $N_i = |\Vert(K_i \setminus B)|$, $k=1,2$. 
We have 
\begin{equation}
\begin{split}
\label{number-opposite-vertices-equation}
|U| + |V| = N_1 N_2, \\
|V| \le \min(N_1,N_2), \\
N_1 N_1 \ge |U| \ge N_1 N_2 - \min(N_1,N_2), \\
|U| \le 2.
\end{split}
\end{equation}

\bigskip

(I) We now prove claim (I) of the lemma. 
It is sufficient to prove that $|U| > 0$, 
because all faces of $C$ are dual cells 
if and only if $|U| = 0$. 

Firstly, observe that both $N_1$ and $N_2$
are positive. If, say, $N_1 = 0$,
then there are no other vertices in chain $K_1$
except those of chain $B$. However, $K_1$ contains
a cell $D^3$. We have
$\Vert(D^3) \subset \Vert(B) = \Vert(C)$,
so by lemma  \ref{cube-and-prism-corollary} we have
$D^3 = C$, which is a contradiction since $C$
is a prism and the tiling is $3$-irreducible.

Now suppose that $|U| = 0$.
If $D^4$ is asymmetric, then we have 
$|U| = N_1 N_2 = 0$,
so either $N_1 = 0$ or $N_2 = 0$ 
which is impossible by the argument
above. 

If $D^4$ is centrally symmetric, then $|U| = 0$
yields 
\begin{equation}
N_1 N_2 = |V|
\end{equation}
This means that each vertex from $\Vert(K_1 \setminus B)$
is diametrically opposite to each vertex from 
$\Vert(K_1 \setminus B)$. But one vertex cannot
be diametrically opposite to only one vertex.
Therefore: either $N_1 = 0$, or $N_2 = 0$,
or $N_1 = N_2 = 1$. In the latter case,
the two vertices, one in $\Vert(K_1 \setminus B)$
the other in $\Vert(K_2 \setminus B)$,
are diametrically opposite. The central
symmetry of $D^4$ then maps vertices
of $\Vert(B) = \Vert(C)$ onto themselves.
That is a contradiction because $C$ 
is not centrally symmetric. We have proved
that $|U| > 0$, therefore not all parallelogram
faces of $C$ are dual cells.

\bigskip

(II) Now we prove the second claim
of the lemma. By the argument
above, we have $N_1,N_2 > 0$, $|U| > 0$.

From equation \ref{number-opposite-vertices-equation},
it follows that 
the following combinations of $(N_1,N_2)$ are possible: 
$(1,1)$, $(1,2)$, $(2,2)$ (interchanging $N_1$ and $N_2$
if needed).

%
% Change??
%
Suppose that $N_1 = N_2 = 1$. Then $|\Vert(D^4)|=8$. If 
$\dim(D^4)=3$, then  
by lemma \ref{illumination-conjecture-lemma} 
(page \pageref{illumination-conjecture-lemma})
$D^4$ has to be a parallelepiped,
but then its combinatorial dimension is $3$ which 
is a contradiction. Therefore
$\dim(D^4) = 4$. The only possible value of $|U|$
is $|U| = 1$. It follows that $D^4$ is a bipyramid 
over a triangular prism, affinely equivalent to the 
polytope given in equation 
\ref{8-vertex-dual-cell-equation}. 
In the matrix, the middle 6 columns correspond to 
the vertices of the prism $C$: columns 2-4 
and 5-7 give the vertices of triangular faces of $C$.

We have established $D^4$ up 
to affine equivalence. Now we look at the subcell
structure of $D^4$. In complex $K$,
each parallelogram face of $C$ is the base
of two pyramids, with vertices 
$y_1 \in \Vert(K_1 \setminus B)$ 
and $\Vert(K_1 \setminus B)$.
All $3$-cells in $K$ are the 6 pyramids,
therefore $K$ has exactly $3$ parallelogram cells,
faces of $C$. Two of these faces are dual
cells, one is inscribed in an octahedral dual cell.
It follows that $D^4$ contains $2$ parallelogram
dual cells.

Suppose that $N_1 = 1$, $N_2 = 2$.  Then $|\Vert(D^4)|=9$. 
By lemma \ref{illumination-conjecture-lemma}, 
$\dim(D^4) = 4$. Also, $D^4$ cannot 
be centrally symmetric because it has an odd number of vertices,
so there are no diametrically opposite pairs of vertices and $|U|=2$.
This, again, establishes $D^4$ up to affine equivalence: 
\begin{spacing}{1}
\begin{equation}
\label{9-vertex-dual-cell-equation}
D^4 = \conv\left[ 
\begin{array}{ccccccccc} 
0 & 1 & 1 & 1 & 0 & 0 & 0 & 1 & 1 \\
0 & 0 & 1 & 1 & 0 & 1 & 1 & 1 & 1 \\
0 & 0 & 0 & 0 & 1 & 1 & 1 & 1 & 1 \\
0 & 0 & 0 & 1 & 0 & 0 & 1 & 1 & 0
\end{array} 
\right]
\end{equation}
\end{spacing}
\medskip
Columns 2-7 of the matrix give
the vertices of $C$, $\Vert(K_1\setminus B) = \lbrace [0,0,0,0] \rbrace$,
$\Vert(K_2\setminus B) = \lbrace [1,1,1,1],[1,1,1,0] \rbrace$.

Dual cell $D^4$ has a prism face $C'$
defined by hyperplane $x_2 = 1$. By lemma 
\ref{centrally-symmetric-face-lemma}, all parallelogram 
faces of $C'$ are subcells, which is impossible
by claim (I) of the lemma.

Finally, let $N_1 = N_2 = 2$. We have from equation 
\ref{number-opposite-vertices-equation}: $|U| \ge 2$. 
On the other hand, $|U|\le 2$ since there are at most 
two parallelogram faces of $C$ which extend to octahedra. 
Therefore $|U|=|V|=2$. 
Thus $D^4$ is centrally symmetric, and
\begin{equation}
\begin{split}
V = \lbrace (x,*(x)), (y,*(y)) \rbrace, \\ 
U = \lbrace (x,*(y)), (y,*(x)) \rbrace
\end{split}
\end{equation}
Since $*$ is a central symmetry, line segments
$[x,*(y)]$ and $[y,*(x)]$ are translates
of each other, however they are diagonals
of two octahedra which are not translates
of each other. This contradiction
completes the proof of the lemma.
\end{proof}
\end{subsection}

\begin{subsection}{The four cases}
We have an arbitrary dual $4$-cell $D^4$ in 
a $3$-irreducible parallelotope tiling.
To complete the proof of the theorem,
we need to show that the collection ${\cal R}$
of incoherent parallelogram subcells in $D^4$
is empty. 

We suppose that $|{\cal R}| \ge 1$. In each of 
the following cases, which cover all possibilities, 
we will get a contradiction. 
\begin{enumerate}
\item $|{\cal R}| = 1$
\item $D^4$ is asymmetric, $|{\cal R}| \ge 2$ and at
least one pair of parallelograms in ${\cal R}$ is  
adjacent, translate, or skew.
\item $D^4$ is centrally symmetric and $|{\cal R}| \ge 2$.
\item $D^4$ is asymmetric, $|{\cal R}| \ge 2$ and all
pairs of parallelograms in ${\cal R}$ are complementary.  
\end{enumerate}

\end{subsection}

\begin{subsection}{Case 1}
The first case is the easiest to get rid of. We have 
${\cal R} = \lbrace \Pi \rbrace$. By theorem 
\ref{exposed-vertex-theorem} on page
\pageref{exposed-vertex-theorem}, $\Pi$ is coherent:
take any vertex $v\in\Pi$ and observe that all parallelogram
subcells of $D^4$ different from $\Pi$ (if any) 
are coherent, because they do not belong
to ${\cal R}$. We have a contradiction.
\end{subsection}

\begin{subsection}{Case 2}
We now consider the case when the dual cell
$D^4$ is asymmetric and the collection
${\cal R}$ of incoherent parallelogram cells 
contains an adjacent, translate, or skew pair
of parallelograms. There is a prism
$C$ with $\Vert(C)\subset\Vert(D^4)$, and one 
of the parallelogram faces of $C$ is a dual cell.
The following figure 
shows an adjacent, translate, and a skew pair
of parallelogram cells together with the prism.
\bigskip
\begin{center}
\resizebox{250pt}{!}{\includegraphics[clip=true,keepaspectratio=true]
{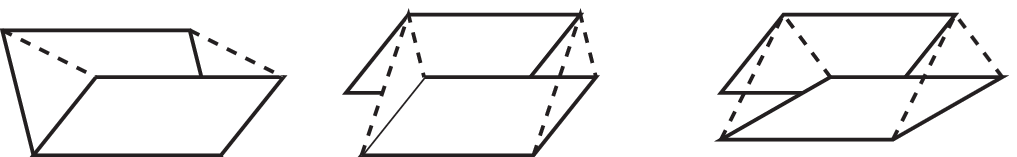}}
\end{center}

Therefore, by lemma \ref{4-cell-classification-lemma-prism}, 
$D^4$ is affinely equivalent to the polytope in 
equation \ref{8-vertex-dual-cell-equation}
on page \pageref{8-vertex-dual-cell-equation}. 
It has exactly $2$ parallelogram subcells, 
therefore each of these subcells is coherent 
in $D^4$, by the principle that incoherent 
parallelograms ``come in droves'',  specified 
in theorem \ref{exposed-vertex-theorem}
(page \pageref{exposed-vertex-theorem}).
Therefore ${\cal R} = \emptyset$ which is a 
contradiction.
\end{subsection}

\begin{subsection}{Case 3}
In this case, we have a centrally symmetric dual cell $D^4$
in a $3$-irreducible tiling. Let $\Pi$ be 
an incoherent parallelogram subcell of $D^4$.

Let $*$ be the central symmetry mapping of $D^4$.
Note that the centers of symmetry of $D^4$ and $\Pi$
are different, because these centers are also 
the centers of symmetry of the corresponding $(d-4)$-
and $(d-2)$-faces of the tiling.
Let $C = \conv( \Pi \cup *(\Pi) )$. We prove 
the following result:

{\em
\label{4-cell-classification-lemma-cube}
$D^4$ is a centrally symmetric bipyramid over $C$
and the faces of $D^4$ coincide with its subcells.}

Once we have established this, we will apply theorem
\ref{bipyramid-coherence-theorem} on page
\pageref{bipyramid-coherence-theorem} to show that
all parallelogram subcells of $D^4$ are coherent.
This will be a contradiction with the assumption that 
$\Pi$ is incoherent.

\bigskip

\begin{proof} 
The argument is very similar to lemma
\ref{4-cell-classification-lemma-prism} 
on page \pageref{4-cell-classification-lemma-prism},
with prism substituted by the parallelepiped $C$.

\bigskip

(A). {\em Definition of complex $K$.}
Let $\Pi_k$, $k=1,\dots,6$ be the facets of $C$.
We define abstract complex $K$ as a subdivision
of $\partial D^4 = \lbrace D: D \text{ is a subcell of } D^4, D \ne D^4 \rbrace$,
as follows.

Each of the vertex sets $\Vert(\Pi_k)$ is a centrally 
symmetric subset of $\Vert(D^4)$, hence by lemma
\ref{centrally_symmetric_subsets_lemma} there is a centrally 
symmetric dual subcell $E_k \subset D^4$ with
$\Vert(\Pi_k)\subset \Vert(E_k)$. By lemma
\ref{translate_complex_main_lemma}, either 
\begin{enumerate}
\item $E_k = \Pi_k$, or 
\item $E_k$ is an octahedron (we say that $\Pi_k$ is inscribed into the 
octahedron $E_k$).
\end{enumerate}
If $E_k$ is an octahedron, then we subdivide it
into two pyramids with base $\Pi_k$, as shown
in figure \ref{splitting-octahedron-figure-2}. 
This finishes the definition of complex $K$.
\begin{figure}
\begin{center}
\resizebox{200pt}{!}{\includegraphics[clip=true,keepaspectratio=true]
{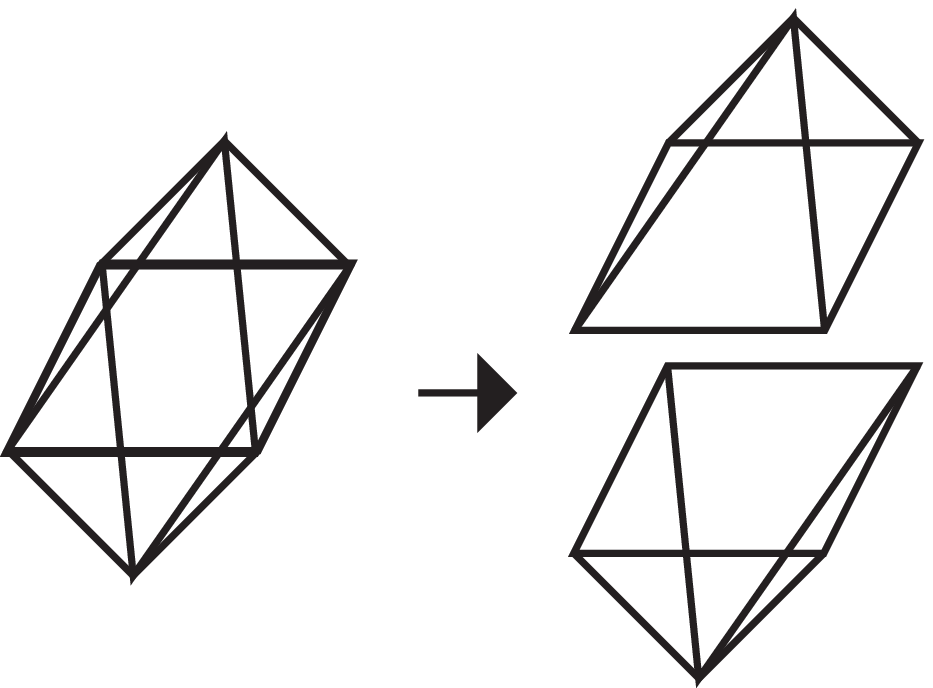}}
\caption{Splitting the octahedral cell}
\label{splitting-octahedron-figure-2}
\end{center}
\end{figure}

This is exactly the construction used in lemma 
\ref{4-cell-classification-lemma-prism}. 

\bigskip

(B). {\em Application of Alexander's separation 
theorem.}
We define $2$-chain $B$ in the complex $K$ 
as the sum of the cellular $2$-chains 
corresponding to the facets of $C$. 
$B$ is closed. Therefore, by 
the separation theorem 
(theorem \ref{separation-theorem} 
on page \pageref{separation-theorem}), 
$B$ bounds exactly two open $3$-chains in $K$, $K_1$ 
and $K_2$ where $K_1 + K_2 \equiv K \pmod 2$. 
Moreover, any common element of $K_1$ and $K_2$ 
is contained in $B$.

Note that the central symmetry $*$ of the dual cell 
$D^4$ maps chain $K_1$ onto $K_2$ and vice versa.
It follows from the fact that mapping $*$ 
changes the orientation of cells in $K_1$, 
but preserves the orientation of $C$.

\bigskip

(C). {\em Counting vertices.}
We define set of vertex pairs $U$ as follows.
\begin{equation}
\begin{split}
U = \left\lbrace (y_1,y_2)\in \Vert(K_1\setminus B) 
\times \Vert(K_2\setminus B): y_1 \ne *(y_2) \right\rbrace
\end{split}
\end{equation}
where $*$ is the central symmetry of $D^4$.

The following relationship establishes 
a 1-1 correspondence between set $U$ 
and the collection of those parallelograms $\Pi_k$ which 
are inscribed into octahedra $E_k$:
\begin{equation}
\label{matching-equation-2}
\begin{split}
(y_1,y_2) \to \Pi_k \\
\text{if and only if} \\
E_k = \conv( \Pi_k \cup \lbrace y_1,y_2 \rbrace ).
\end{split}
\end{equation}
The proof is identical to the one 
in lemma \ref{4-cell-classification-lemma-prism}.
By construction, at least two parallelograms
$\Pi_k$ are dual cells, therefore $|U|\le 4$.

As we noted in paragraph (B), $*$ maps 
the chains $K_1$ and $K_2$ onto each other, therefore each 
vertex in $\Vert(K_1\setminus B)$ has its centrally symmetric
counterpart in $\Vert(K_2\setminus B)$, and vice versa. 
Therefore $|\Vert(K_1 \setminus B)| = |\Vert(K_2 \setminus B)| = N$,
and we have
\begin{equation}
|U| = N^2 - N.
\end{equation}
Since $|U| \le 4$, we can only have 
$N = 0,1,2$. We now consider each of these outcomes separately. 

If $N = 0$, then $D^4 = C$, so by lemma 
\ref{description-of-cube-subcells-lemma}
(page \pageref{description-of-cube-subcells-lemma}), 
the combinatorial 
dimension of $D^4$ is $3$ which is a contradiction.

If $N = 1$ or $2$, then
$\dim(D^4) = 4$ 
(if $\dim(D^4) = 3$, then $D^4$ 
is not skinny
by lemma \ref{cube-maximality-lemma} 
(page \pageref{cube-maximality-lemma})
applied to $C$).

If $N = 1$, then $D^4$ is a bipyramid 
over the parallelepiped $C$. Each 
parallelogram face of $C$ belongs to two pyramids.
There are 12 pyramids, 2 for each of the $6$
parallelograms. All pyramids are facets of $D^4$.
But there are a total of 12 facets.
It follows
that subcells of $D^4$ coincide with its faces.

If $N = 2$, then $|U| = 2$ and there 
are two parallelogram faces of $C$ which
extend to octahedra. 
Since $D^4$ is centrally symmetric, 
the two octahedra are mapped onto
each other by the central symmetry.
Therefore $D^4$ is a direct Minkowski sum 
of an octahedron and a line segment.
It has a prism face $C'$; by lemma  
\ref{centrally-symmetric-face-lemma}
on page \pageref{centrally-symmetric-face-lemma}, all facets 
of $C'$ are dual cells, which is impossible
by lemma \ref{4-cell-classification-lemma-prism}
on page \pageref{4-cell-classification-lemma-prism}.
\end{proof}

We have thus established that $D^4$ is
a centrally symmetric bipyramid 
over a parallelepiped, therefore
all parallelogram subcells of $D^4$
are coherent by theorem
\ref{bipyramid-coherence-theorem} on page
\pageref{bipyramid-coherence-theorem}
and ${\cal R} = \emptyset$. This completes
the analysis of case 3.
\end{subsection}

\begin{subsection}{Case 4}
This is the final and the most complex part
of the proof. We have an asymmetric dual cell 
$D^4$ and the following two conditions 
hold for the collection
${\cal R}$ of incoherent parallelogram 
subcells of $D^4$:
\begin{enumerate}
\item each two parallelograms in ${\cal R}$ intersect 
over a vertex.
\item each vertex of a parallelogram in ${\cal R}$ 
belongs to at least one other parallelogram in ${\cal R}$.
\end{enumerate}
\label{closed-conditions-page}
We prove that ${\cal R} = \emptyset$ by assuming
it is not true, performing a combinatorial analysis
of the collection of parallelograms ${\cal R}$
and obtaining a contradiction.

The argument consists of three parts.
Firstly, we forget about the fact that elements
of ${\cal R}$ are parallelograms. 
We look at ${\cal R}$ as a {\em hypergraph},
\index{hypergraph}
or a system of $4$-point subsets $\Vert(\Pi)$
of the vertex set
${\cal V} = \cup_{\Pi\in{\cal R}} \Vert(\Pi)$.
These $4$-point subsets are called
{\em hyperedges}. 

Hypergraphs can be drawn like usual graphs,
with lines representing hyperedges.

We prove that ${\cal R}$ contains
either a 5-10, or a 6-11 subsystem
of parallelograms (see definition 
in the statement of lemma 
\ref{5-10-6-11-lemma} below).

In the second and the third 
parts, we study
the 5-10 and 6-11 subsystems.
We remember that elements
of ${\cal R}$ are parallelogram dual cells
and use geometric properties
of dual cells, established
in section \ref{affine-results-section}, 
to complete the analysis. 

\bigskip

\noindent{\bf Combinatorial analysis}
A hypergraph with $4$-element hyperedges
will be called {\em closed} if it 
satisfies the following conditions:
\begin{enumerate}
\item each two hyperedges intersect
in exactly one vertex,
\item each vertex belongs to at least
two hyperedges.
\end{enumerate}
Viewed as a hypergraph, the system
of parallelograms ${\cal R}$ is closed.
\footnote{David Gregory 
\cite{bib_david_gregory_private_comm_1} 
suggested that the hypergraph terminology
can be used in this work.

Participants of the Discrete Mathematics Seminar at 
Queen's University and Royal Military College 
let us know that the dual hypergraph 
to a ``closed'' hypergraph is an example of a 
`Pairwise Balanced Design` with parameter $\lambda=1$. 
Special thanks to Reza Naserasr of Royal 
Military College, Kingston, Ontario,
Canada for interest in the work.
\index{partial block design}
}

Let ${\cal R}$ be an arbitrary
closed hypergraph.
Let ${\cal V}$ be the set of vertices
of ${\cal R}$, $V = |{\cal V}|$, 
$R = |{\cal R}|$. For a vertex $v\in {\cal V}$, its 
{\em degree} $m_v$ is the number 
of hyperedges containing it.

\begin{proposition}
For all $v\in {\cal V}$ we have $2 \le m_v \le 4$.
\end{proposition}

\begin{proof} The first inequality directly follows from 
the definition of a closed system. To prove the second,
suppose that there are $5$ or more hyperedges containing 
$v$. There is a hyperedge $\Pi\in{\cal R}$ such that 
$v\notin \Pi$, and it intersects each of the hyperedges 
which contain $v$, yet it has only $4$ vertices which 
is a contradiction.
\end{proof}

\begin{calculation} 
The following formulas hold:
\begin{equation}
\label{1_moment_equation}
\sum_{v\in V} m_v = 4R,
\end{equation}
\begin{equation}
\label{2_moment_equation}
\sum_{v\in V} m_v^2 = R(R + 3),
\end{equation}
\begin{equation}
\label{3_moment_equation}
\sum_{v\in V} (m_v - 2) = 4R - 2V,
\end{equation}
\begin{equation}
\label{4_moment_equation}
\sum_{v\in V} {(m_v - 2)}^2 = R(R - 13) + 4V,
\end{equation}
\begin{equation}
\label{num_vertices_degree_4_equation}
|\lbrace v \in {\cal V} : m_v = 4 \rbrace| 
= \frac{R(R - 17)}{2} + 3V.
\end{equation}
\end{calculation}

\begin{proof}
The first equality is obvious. The second one 
follows from it and from the fact that each two hyperedges 
have exactly one common vertex so
\begin{equation}
\frac{R(R - 1)}{2} = \sum_{v\in V} \frac{m_v(m_v - 1)}{2}.
\end{equation}
The third and fourth statements are implied from the first two. 
The fifth follows from the equation
\begin{equation}
|\lbrace v \in {\cal V} : m_v = 4 \rbrace| = 
\frac{\sum_{v\in V} {(m_v - 2)}^2 - \sum_{v\in V} (m_v - 2)}{2}.
\end{equation}
which is true because $2\le m_v \le 4$ and the right hand
side is the sum of the indicator function
$\frac{1}{2}({(m_v - 2)}^2 - (m_v - 2))$ of the subset 
of vertices $v$ with $m_v = 4$.
\end{proof}

For example, it follows from the calculations that 
if for all $v\in V$ $m_v = 2$, then $R = 5$ and $V = 10$. 

\begin{corollary} 
\label{two-inequalities-corollary}
$R \le V \le 2R$.
\end{corollary}
\begin{proof}
The result follows from formula \ref{3_moment_equation}
and inequality $2 \le m_v \le 4$.
\end{proof}

\begin{lemma} 
\label{5-10-6-11-lemma}
A nonempty closed hypergraph
contains a subgraph
isomorphic to the 5-10 hypergraph
or the 6-11 hypergraph.
The two hypergraphs are shown 
in figure \ref{hypergraph-diagrams-figure}.
\end{lemma}

The hypergraphs are called so because
the first has 5 hyperedges and 10 vertices,
the second 6 hyperedges and 11 vertices.

\begin{figure}
\begin{center}
\resizebox{150pt}{!}{\includegraphics[clip=true,keepaspectratio=true]
{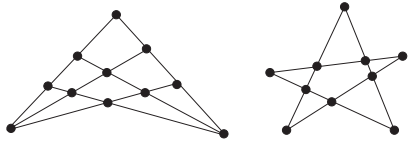}}
\caption{The 6-11 and 5-10 hypergraphs}
\label{hypergraph-diagrams-figure}
\end{center}
\end{figure} 
\index{5-10 hypergraph}
\index{6-11 hypergraph}

\begin{proof}
We assume that closed hypergraph ${\cal R}$ 
is minimal (that is, it does not contain a proper
closed subgraph).

We continue to use the notations 
$R = |{\cal R}|$, $V=|{\cal V}|$. Note that $R\ge 5$: 
since there is at least one hyperedge $Q$,
there should be $4$ more, each containing one vertex of $Q$. 
The following observations will be used in the argument below.

(A) If there are at least two vertices $v_1$, $v_2$ 
of degree $4$ in ${\cal R}$, then ${\cal R}$ 
contains a 6-11 subgraph. Indeed, we can pick up $3$ 
hyperedges containing $v_1$ and $3$ hyperedges 
containing $v_2$, which can be done since there can be 
at most one hyperedge containing both vertices.

(B) If there are two vertices $v_1$, $v_2$ of degrees 
at least $3$ which do not belong to the same hyperedge, 
then the system ${\cal R}$ contains a 6-11 subgraph. 
Again one picks up $3$ hyperedges containing 
$v_1$ and $3$ hyperedges containing $v_2$.

(C) If all vertices have degree 2, then ${\cal R}$ 
is a 5-10 hypergraph. This follows from the formulas 
on page \pageref{1_moment_equation}.

(D) If none of (A),(B),(C) hold and $R \ge 9$, then 
the ${\cal R}$ contains a 5-10 subgraph. Let $v$ 
be the vertex of degree $4$, if such exists, or a vertex 
of degree $3$ otherwise. Let $\Pi_1,\dots,\Pi_4$
(or $\Pi_1,\dots,\Pi_3$) be all hyperedges which contain 
$v$. Then, since (B) does not hold, all vertices of degree 
greater than $2$ are contained in the union of hyperedges 
$\Pi_i$. Since (A) does not hold, all vertices in the union 
of hyperedges $\Pi_i$ are of degree at most $3$
(except, maybe, $v$). Let $Q_1,\dots,Q_5 \in {\cal  R}$ be 
some hyperedges different from $\Pi_i$. They form 
a 5-10 hypergraph (for no three of them can intersect in one point).

We are left with cases where $5 \le R \le 8$. We will use 
inequalities $R \le V \le 2R$, proved in corollary
\ref{two-inequalities-corollary}. The table 
below shows the number of vertices of degree $4$ for each 
pair of values $R$ and $V$, based on formula
(\ref{num_vertices_degree_4_equation}).
Dash ``-'' means that formula 
\ref{num_vertices_degree_4_equation} returns a negative 
number which is a contradiction, so corresponding pair 
of numbers $R$, $V$ is not possible. Star ``*'' means that 
condition $R \le V \le 2R$ is violated.

\begin{table}[htbp]
$
\begin{array}{|r|cccccccc|}
\hline
R \setminus V & 5..9 & 10 & 11 & 12 & 13 & 14 & 15 & 16 \\
\hline
5 & - & 0 & * & * & * & * & * & *  \\ 
6 & - & - & 0 & 3 & * & * & * & *  \\
7 & - & - & - & 1 & 4 & 7 & * & *  \\
8 & - & - & - & 0 & 3 & 6 & 9 & 12 \\
\hline 
\end{array}
$ 
\caption{Number of degree 4 vertices by values of $V$ and $R$}
\label{degree-4-number-table} 
\end{table}

Pairs of $R$ and $V$ where there are more than one degree $4$ 
vertex are covered by case (A). We have to consider
the rest.

If $R = 5$ and $V = 10$, ${\cal R}$ is a 5-10 hypergraph.

If $R = 6$ and $V = 11$, there are no degree $4$ vertices 
and $2$ degree $3$ vertices $v_1$, $v_2$ because 
$\sum_{v\in V} (m_v - 2) = 2$ and $\sum_{v\in V} {(m_v - 2)}^2 = 2$ 
(we use the calculation above). If $v_1$, $v_2$ do not belong 
to the same hyperedge, then we use case (B). If they belong 
to some hyperedge $Q$, then it is sufficient to pick $Q$ 
and hyperedges $Q_1,\dots,Q_4$ at each of the vertices 
of $Q$ to get a 5-10 hypergraph.

If $R = 7$ and $V = 12$, then there is $1$ vertex 
of degree $4$, but this implies that there are at least $13$
vertices - a contradiction.

If $R = 8$ and $V = 12$, then $\sum_{v\in V} (m_v - 2) = 8$ 
and $\sum_{v\in V} {(m_v - 2)}^2 = 8$, so there are $8$
vertices of degree $3$ and $4$ vertices of degree $2$. 
Take a vertex $v$ of degree $3$ and let
$Q_1$, $Q_2$ and $Q_3$ be all hyperedges containing it. 
If there is another vertex $u$ so that $m_u = 3$ and
$u\notin\cup_{i=1}^{3} Q_i$, then we use case (B). Otherwise 
the remaining $7$ vertices of degree $3$ are
found among vertices of $Q_i$, $i=1\dots 3$, which means 
that on one of $Q_i$ all vertices are of degree $3$.
Therefore ${\cal R} \setminus\lbrace Q_i \rbrace$ 
is closed so ${\cal R}$ is not minimal contrary to the assumption. 
The contradiction finalizes the proof.
\end{proof}

We consider the cases when ${\cal R}$ contains a 5-10 subgraph 
and when ${\cal R}$ contains a 6-11 subgraph separately.

\medskip

\noindent {\bf Hypergraph ${\cal R}$ contains a 5-10 subgraph ${\cal R}'$}

We calculate all possible systems ${\cal R}'$
and then use geometric 
properties of dual cells to obtain the desired
contradiction.

So far our analysis only considered parallelograms 
as hyperedges, unordered $4$-element subsets of ${\cal V}$.
Now we take into account the fact that 
parallelograms are centrally symmetric
$4$-vertex polygons. We match
diagonally opposite vertices of
each parallelogram.
As a result, the vertices of 
each hyperedge are split into
two pairs. We will call this 
construction ``vertex matching''.

We will now classify all possible vertex matchings 
of the 5-10 hypergraph, up to isomorphism.

It is convenient to use the dual graph
to the 5-10 hypergraph, the complete graph
on $5$ elements $K_5$. Matching of vertices
in each hyperedge of ${\cal R}'$ corresponds
to matching of edges of $K_5$ incident to 
the the same vertex. 

Matching of edges of $K_5$ at each vertex
is equivalent to building a snow ploughing scheme
on ``roads''-edges of $K_5$, where multiple 
ploughing machines travel on circuits so 
that each edge is cleaned, and no edge 
is cleaned twice
\footnote{A circuit that passes each 
edge exactly once 
is called {\em Eulerian cycle}.}.
Since the paper was written in Canada, this
analogy makes a lot of sense to us.
\index{Eulerian cycle}
\index{snow ploughing scheme}

Correspondence between 
snow ploughing schemes and edge matchings is 
straightforward: as each machine travels 
on the graph, we simply match the entering 
and exiting roads (edges)
at each crossing (vertex). Conversely,
matching of edges at each vertex tells 
a machine which way to go from a crossing (vertex)
after entering from a given road (edge).

We will write a snow ploughing scheme 
as a system of vertex cycles, indicating
vertices that each ploughing machine
visits.
\begin{lemma} 
\label{eulerian-roadsystems-lemma}
Vertices of $K_5$ can be labeled by $1,\dots,5$ 
so that a snow ploughing scheme of $K_5$ 
is one of the following:
\begin{spacing}{1}
\begin{enumerate}
\item $(1,2,3,1,4,2,5,4,3,5)$
\item $(1,3,5,1,2,5,4,2,3,4)$
\item $(1,2,3,1,4,2,5,3,4,5)$
\item $(1,2,3,1,4,5,2,4,3,5)$
\item $(1,4,2,5,3,4,5)$, $(1,2,3)$
\item $(1,2,3,4)$, $(1,5,2,4,5,3)$
\item $(1,2,3,4,5)$, $(1,3,5,2,4)$
\item $(1,5,2,4)$, $(1,2,3)$, $(3,4,5)$
\end{enumerate}
\end{spacing}
\end{lemma}
\begin{proof} We used a computer program to find 
all Eulerian cycles on $K_5$ (lines 1-4). 
For a scheme of two or more
cycles (lines 5-8), an easy manual 
classification proves the result.
\end{proof}

From this lemma we know all vertex matchings on 
edges of a 5-10 hypergraph, up to
isomorphism.
We denote the hyperedges/parallelograms
in ${\cal R}'$ by $\Pi_1,\dots,\Pi_5$,
and label vertices of the hypergraph
${\cal R}'$ by $v_{ij}$, $i\ne j$, 
where $v_{ij}$ stands for 
the intersection of hyperedges $\Pi_i$ and $\Pi_j$.
Now, each vertex is a point in space,
and matching pairs of vertices of 
the same hyperedge represent diagonals
of a parallelogram. This allows
us to write a system of linear
equations.

For example, snow ploughing scheme 1
translates into the following system:
\begin{equation}
\begin{split}
v_{12}+v_{15}&=v_{13}+v_{14}, \\
v_{12}+v_{23}&=v_{24}+v_{25}, \\
v_{23}+v_{13}&=v_{34}+v_{35}, \\
v_{14}+v_{24}&=v_{34}+v_{45}, \\
v_{25}+v_{45}&=v_{15}+v_{35}
\end{split}
\end{equation} 
We choose a coordinate system $e_1,\dots,e_d$ 
in ${\mathbb R}^d$ so that parallelograms
$\Pi_1$ and $\Pi_2$
are $\conv\lbrace 0,e_1,e_2,e_1+e_2\rbrace$, 
$\conv\lbrace 0,e_3,e_4,e_3+e_4\rbrace$. 
This is possible by the assumption
that each two parallelograms are complementary,
that is, they share a vertex and span a 
$4$-space.

We then consider each of
the cases in lemma \ref{eulerian-roadsystems-lemma}, 
write down the corresponding linear equations
and solve them. The results, easily verifiable, 
are shown in table \ref{5-10-systems-table}. 
Each solution is shown in the format
$[v_{12},v_{13},v_{14},v_{15},v_{23},v_{24},v_{25},v_{34},v_{35},v_{45}]$.
For each solution, we indicate if it 
is immediately obvious that the parallelogram
system cannot occur in a dual $4$-cell,
in column ``reason for contradiction''.

\begin{table}
\begin{spacing}{1}
{\tiny
\begin{equation}
\label{5-10-systems-table}
\begin{array}{|l|l|l|}
\hline
\text{Case} & \text{Solution} & \text{Reason for contradiction} \\
\hline
1  & \text{None} & \text{n/a} \\

\smallskip \\
\hline

2 &
{
\left[ \begin {array}{cccccccccc} 
0&1&1&0&0&0&0& a^1-1& a^1   & a^1   \\
\\
0&0&1&1&0&0&0& a^2-1& a^2-1 & a^2   \\
\\
0&0&0&0&1&1&0& a^3+1& a^3   & a^3   \\
\\
0&0&0&0&0&1&1& a^4+1& a^4+1 & a^4   \\
\\
0&\dots&&&&&&  a^5  & a^5   & a^5   \\
\dots 
\end {array} \right] 
}
& \text{See below} \\

\smallskip \\
\hline

3 & \text{None} & \text{n/a} \\

\smallskip \\
\hline

4 &
{
\left[ \begin {array}{cccccccccc} 
0&1&0&1&0&0&0&1&0&1\\
0&0&1&1&0&0&0&1&-1&0\\
0&0&0&0&1&1&0&0&1&1\\
0&0&0&0&1&0&1&0&1&0\\
0\\
\dots
\end {array} \right] 
}
& v_{34} = v_{15} \\

\smallskip \\
\hline

5 & \text{None} & \text{n/a} \\

\smallskip \\
\hline

6 & 
{
\left[ \begin {array}{cccccccccc} 
0&1&1&0&0&0&0&0&-1&1 \\
0&0&1&1&0&0&0&0&0&1  \\
0&0&0&0&1&1&0&0&1&-1 \\
0&0&0&0&1&0&1&0&1&0  \\
0 \\
\dots
\end{array} \right] 
}
& v_{34} = v_{12} \\

\smallskip \\
\hline

7  & 
{
\left[ \begin {array}{cccccccccc} 
0&1&0&1&0&0&0&1&0&-1 \\
0&0&1&1&0&0&0&1&1&0  \\
0&0&0&0&1&1&0&0&1&1  \\
0&0&0&0&1&0&1&-1&0&1 \\
0\\
\dots
\end {array} \right]
}
& \text{ See below } \\

\smallskip \\
\hline

8 &
{
 \left[ \begin {array}{cccccccccc} 
0&1&1&0&0&0&0&1&0&0 \\ 
0&1&0&1&0&0&0&0&1&0 \\
0&0&0&0&1&1&0&1&0&0 \\
0&0&0&0&1&0&1&0&1&0 \\
0 \\
\dots
\end {array} \right] 
}
& v_{45} = v_{12} \\

\smallskip \\
\hline
\end{array}
\end{equation}
}
\caption{All parallelogram systems with 5-10 hypergraph, up to affine equivalence}
\end{spacing}
\end{table}

\begin{figure}
\begin{center}
\resizebox{160pt}{!}{
\includegraphics[clip=true,keepaspectratio=true]
{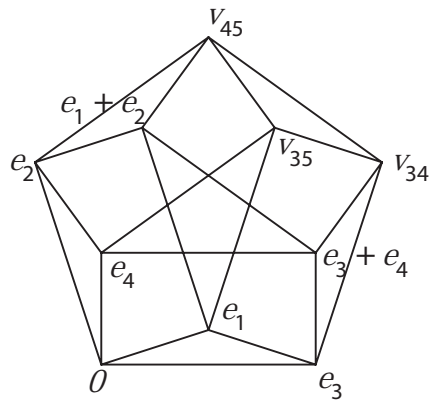}
}
\caption{System of parallelograms from line 2 of table 
\ref{5-10-systems-table}}
\label{5-10-system-figure}
\end{center}
\end{figure} 

We are left to prove that systems of parallelograms 
with vertices shown in lines 2 and 7 of table \ref{5-10-systems-table}
cannot occur in the dual $4$-cell.

First consider row 7. We have
\begin{equation}
\label{octahedron_equation}
[v_{12},v_{35},v_{15},v_{23},v_{34},v_{45}]=
{\tiny
\left[ \begin {array}{cccccc} 
0  & 0  & 0  & 0  &  1  & -1 \\
0  & 1  & 1  & 0  &  1  &  0 \\
0  & 1  & 0  & 1  &  0  &  1 \\
0  & 0  & 0  & 0  & -1  &  1 \\
0 \dots
\end {array} \right].
}
\end{equation}
The 6 points $v_{12},v_{35},v_{15},v_{23},v_{34},v_{45}$ form 
a centrally symmetric subset of $\Vert(D^4)$
(the center of symmetry is ${[0,1/2,1/2,0]}$). 
By lemma \ref{centrally_symmetric_subsets_lemma}, 
there is a centrally symmetric dual cell $D\subset D^4$
with 
$\lbrace v_{12},v_{35},v_{15},v_{23},v_{34},v_{45}\rbrace \subset \Vert(D)$. 
Since $D^4$ is asymmetric, $D$ is a proper subcell 
of $D^4$ so $\combdim(D)\le 3$. But $\combdim(D) \ge 3$
since the 6 points in equation \ref{octahedron_equation} span a $3$-dimensional 
space. Therefore $\combdim(D)=3$. By lemma \ref{dual-2-3-cells-lemma},
$D$ is a simplex, an octahedron 
or a pyramid. But $D$ is centrally symmetric so $D$ 
is an octahedron. Line segment $[v_{15},v_{45}]$ is an edge
of the octahedron, so it is a dual $1$-cell by lemma 
\ref{dual_3_cells_theorem}. However, $[v_{15},v_{45}]$ 
is a diagonal of parallelogram cell $\Pi_5$, which is
impossible.

Next consider the system of parallelograms in row 2 
of table \ref{5-10-systems-table} (shown schematically 
in figure \ref{5-10-system-figure}). This is the 
hardest case of all, because there is a vector 
parameter $a \in {\mathbb R}^d$ in the solution. 

\medskip

\noindent (A) {\em Setup.}
We introduce the following 
notations. Firstly, the new origin in the coordinate 
system will be $E = e_1 + e_4 = v_{25} + v_{13}$.
\begin{spacing}{1}
\begin{equation*}
\begin{aligned} 
y_1 &= v_{52} - E, \\
y_2 &= v_{13} - E, \\
y_3 &= v_{24} - E, \\
y_4 &= v_{35} - E, \\
y_5 &= v_{41} - E.
\end{aligned}
\end{equation*}
We then have:
\begin{equation*}
\begin{aligned}
E+y_1+y_2 &= v_{12}, \\
E+y_2+y_3 &= v_{23}, \\
E+y_3+y_4 &= v_{34}, \\
E+y_4+y_5 &= v_{45}, \\
E+u_5+u_1 &= v_{51}.
\end{aligned}
\end{equation*}
\end{spacing}
\bigskip
From this moment on, we shift the coordinate system, 
so that $E = 0$. We then have
\begin{spacing}{1}
\begin{equation*}
\begin{aligned}[c]
\Pi_1&=\conv\lbrace y_5,y_2,y_5+y_1,y_1+y_2 \rbrace,\\
\Pi_2&=\conv\lbrace y_1,y_3,y_1+y_2,y_2+y_3 \rbrace,\\
\Pi_3&=\conv\lbrace y_2,y_4,y_2+y_3,y_3+y_4 \rbrace,\\
\Pi_4&=\conv\lbrace y_3,y_5,y_3+y_4,y_4+y_5 \rbrace,\\
\Pi_5&=\conv\lbrace y_4,y_1,y_4+y_5,y_5+y_1 \rbrace\\
\end{aligned}
\end{equation*}
\end{spacing}
\bigskip
By theorem \ref{two_cubes_theorem}, statement 3 
(page \pageref{two_cubes_theorem}) the 
linear space 
$L^4 = <y_1, y_2, y_3, y_4> = \aff(\Pi_2\cup\Pi_3)$ 
is complementary to $\lin(F^{d-4} - F^{d-4})$, 
where $F^{d-4}$ is the face of the tiling 
corresponding to the dual cell 
$D^4$. Let $h$ be the projection along $F^{d-4}$
onto $L^4$. By the same theorem, each two parallelograms 
$h(\Pi_i)$ and $h(\Pi_j)$ have a common vertex 
and span a $4$-space.

Now we will lift the system of parallelograms 
$h(\Pi_1),\dots,h(\Pi_5)$ into a $5$-dimensional space. Take 
a vector $x\notin L^4$, and let $u_i=h(y_i)=y_i$, 
$i=1,\dots,4$, $u_5 = h(y_5) + x$. 

We have obtained a pleasantly symmetric representation 
of the lifted system of parallelograms, with vertex set
$\lbrace u_1,\dots, u_5, u_1 + u_2, \dots, u_5 + u_1 \rbrace$. 
This system of parallelograms is invariant under 
transformations of space induced by cyclical substitution 
of indices. The lifted parallelograms are:
\begin{spacing}{1}
\begin{equation*}
\begin{aligned}[c]
P_1&=\conv\lbrace u_5,u_2,u_5+u_1,u_1+y_2 \rbrace,\\
P_2&=\conv\lbrace u_1,u_3,u_1+u_2,u_2+y_3 \rbrace,\\
P_3&=\conv\lbrace u_2,u_4,u_2+u_3,u_3+y_4 \rbrace,\\
P_4&=\conv\lbrace u_3,u_5,u_3+u_4,u_4+y_5 \rbrace,\\
P_5&=\conv\lbrace u_4,u_1,u_4+u_5,u_5+y_1 \rbrace\\
\end{aligned}
\end{equation*}
\end{spacing}
\bigskip
Let $Q = \conv(\lbrace u_1,\dots, u_5, u_1 + u_2, 
\dots, u_5 + u_1 \rbrace)$. Note that parallelograms
$P_1$,\dots,$P_5$ are faces of $Q$ (a trivial 
computation proves it).
Let $p$ be the projection from 
$<u_1,\dots,u_5>$ to $<u_1,\dots,u_4> = <y_1,\dots,y_4>$ 
along vector $x$. 

We have $p(Q)\subset h(D^4)$. We will 
analyze vectors $x=(x^1,\dots,x^5)$ 
where $x^i$ are its coordinates in the basis
$u_1,\dots,u_5$. Our objective is to test
all vectors $x\ne 0$, to prove
that inclusion $p(Q)\subset h(D^4)$
is impossible.

\bigskip

\noindent (B) {\em Tests for vector $x$.}
We now present two conditions that vector
$x$ must satisfy. 

The first condition arises 
from theorem \ref{two_cubes_theorem},
statement 3 (page \pageref{two_cubes_theorem}).
It says that each two parallelograms
$h(\Pi_i) = p(P_i)$, $h(\Pi_j) = p(P_j)$
span a $4$-space. The null space
$<x>$ of projection $p$
must be therefore complementary 
to the space spanned by each pair
of parallelograms $P_i$, $P_j$.
It follows that
\begin{equation}
\begin{split}
x^i \ne 0, \\
i=1,\dots,5,
\end{split}
\end{equation}

The second condition follows from 
statements 2 and 3 of lemma
\ref{translate_complex_main_lemma}.
Let $v$ be an arbitrary vertex of $Q$,
and suppose that $v\notin P_i$.
Let $C_v$ be the
cone bounded by the hyperplanes which
define facets of $Q - v$ containing $0$.
In other words, $C_v = \cone(Q - v)$.
Polytope $Q - v$ coincides with 
$C_v$ in some neighborhood $U_\epsilon(0)$
of $0$. Then:
\begin{equation}
\label{vertex-against-parallelogram-equation}
x \notin \relint(C_v - \lin(P_i - P_i)) \cup -\relint(C_v - \lin(P_i - P_i))
\end{equation}
We call this condition {\em testing
parallelogram $P_i$ against vertex
$v$.}

To show that the condition holds,
suppose eg. $x\in \relint(C_v - \lin(P_i - P_i))$.
Then $x = z_1 - z_2$, where $z_1 \in \relint(C_v)$,
$z_2 \in \lin(P_i - P_i)$. 
We have $z_1 \ne 0$, since $C_v$
is a cone with vertex $0$ 
and $0\notin\relint(C_v)$.

Consider the case when $z_2 = 0$. Then
$x = z_1 \in \relint(C_v)$. But $<x>$ is
the null space of projection $p$,
therefore $p(v)$ is not a vertex of $p(Q)$,
which is a contradiction: both polytopes
$p(Q)$ and $Q$ have $10$ vertices.

Now suppose that $z_2 \ne 0$. We then have
\begin{equation}
z_2 = \alpha_1 t_1 + \alpha t_2
\end{equation}
where $t_1,t_2$ are edge vectors of 
parallelogram $P_i$ with appropriately
chosen directions, 
$\alpha_1,\alpha_2 \ge 0$.
Since $x$ is interesting to us only as the
direction vector for linear projection $p$,
we can multiply vector $x$ and vectors $z_1$ and $z_2$
by the same small positive number so that
$z_1 \in U_\epsilon(0)$, $\alpha_1,\alpha_2 < 1$.
We then have
\begin{equation}
\begin{split}
z_1 \in \relint(Q - v), \\
z_2 \in P_i - v'
\end{split}
\end{equation}
for some choice of vertex $v'$ of $P_i$.
From $x = z_1 - z_2$, we have $p(z_1) = p(z_2)$.
It follows that:
\begin{equation}
\begin{split}
\relint(p(Q - v)) \cap p(P_i - v') \ne \emptyset, \\
\relint(p(Q - v)) \cap p(Q - v') \ne \emptyset, \\
\relint(p(Q) - (p(v) - p(v'))) \cap \relint(p(Q)) \ne\emptyset, \\
\relint( h(D^4) - (p(v)-p(v')) ) \cap \relint(h(D^4)) \ne \emptyset.
\end{split}
\end{equation}
By the choice of $v$ and $P_i$, points $p(v)$
and $p(v')$ are different vertices of $h(D^4)$,
therefore polytope $h(D^4)$ is not skinny. This is 
a contradiction with lemma \ref{dual-cells-skinny-lemma}
on page \pageref{dual-cells-skinny-lemma}.
We have proved that condition 
\pageref{vertex-against-parallelogram-equation}
is a valid test for vector $x$.

Another simple yet useful idea is that the polytope $Q$ 
is invariant under a group consisting of 5 
cyclical coordinate substitutions. This means that 
whatever condition on the projection vector $x$ 
we have derived, up to 5 more conditions can be 
produced by cyclically substituting indices.

\medskip

\noindent (C) {\em Application of the tests.}
Now we test vector $x$ against the conditions
specified above. We used computer software PORTA 
by T. Christof (University of Heidelberg)
to perform polytope computations.

First, we test parallelogram $P_2$ against each 
of the vertices $u_2$, $u_4$, $u_5$. We conclude 
that the projection vector $x$ does not belong 
to any of the following open cones and their centrally 
symmetric images:
\begin{equation}
\begin{split}
\lbrace x: x^4 < 0, x^5 > 0, x^1 + x^3 > 0 \rbrace, \\
\lbrace x: x^4 > 0, x^5 < 0, x^1 + x^3 > 0 \rbrace, \\
\lbrace x: x^4 > 0, x^5 > 0, x^1 + x^3 > 0 \rbrace
\end{split}
\end{equation}
We also know that $x^1\ne 0$, \dots, $x^5 \ne 0$ 
because the projection vector $x$ must be 
complementary to the affine $4$-space spanned by any 
pair of parallelograms $P_i,P_j$. Therefore $x$ must 
satisfy one of the following conditions:
\begin{equation}
\begin{split}
x &\in K^1_{+} = \lbrace x: x^4 > 0, x^5 > 0, x^1 + x^3 < 0 \rbrace, \\
x &\in K^1_{-} = \lbrace x: x^4 < 0, x^5 < 0, x^1 + x^3 > 0 \rbrace, \\
x &\in K^1_{0} = \lbrace x: x^1 + x^3 = 0 \rbrace
\end{split}
\end{equation}
We then rotate the indices in the inequalities to obtain
sets $K^i_+$, $K^i_-$, $K^i_0$ for $i=1,\dots,5$. For example, 
$K^2_0 = \lbrace x: x^2 + x^4 = 0 \rbrace$. We get
\begin{equation}
\begin{split}
x \in \cap_{i=1}^5 (K^i_{+} \cup K^i_{-} \cup K^i_{0} ), \\ 
i=1,\dots,5. 
\end{split}
\end{equation}
Opening brackets in the expression, we have
\begin{equation}
x \in \bigcup_{[\sigma_1,\dots,\sigma_5],\sigma_i\in\lbrace +,-,0\rbrace}   
(K^1_{\sigma_1} \cap K^2_{\sigma_2} \cap K^3_{\sigma_3} 
\cap K^4_{\sigma_4} \cap K^5_{\sigma_5}  ).
\end{equation}
Now we test parallelogram $P_2$ against the vertex $u_4+u_5$. 
The result is:
\begin{equation}
\begin{split}
x \notin J_+ = \lbrace x^5 < 0, x^4 < 0, x^1 + x^3 + x^4 + x^5 < 0, 
x^1 + x^3 > 0 \rbrace, \\
\text{and} \\
x \notin J_- = \lbrace x^5 > 0, x^4 > 0, x^1 + x^3 + x^4 + x^5 > 0, 
x^1 + x^3 < 0 \rbrace. \\
\end{split}
\end{equation}
Using cyclical substitutions, we obtain $10$ conditions
from this test.

After a trivial calculation performed with the PORTA software,
the only vector which passes the tests is $x = [-1,-1,-1,1,1]$ (and its multiples). 
It is treated in a different way.

\bigskip

\noindent (D) {\em Case of $x=[-1,-1,-1,1,1]$.}
The polytope $p(Q)$ is the projection of $Q$
along $x$ onto the space spanned by $u_1,\dots,u_4$. 
Polytope $p(Q)$ has the following vertices 
(coordinates are shown in the basis $u_1,\dots,u_4$): 
\begin{spacing}{1}
\begin{equation}
\begin{split}
p(u_1) = h(y_1) &= [1, 0, 0, 0], \\
p(u_2) = h(y_2) &= [0, 1, 0, 0], \\
p(u_3) = h(y_3) &= [0, 0, 1, 0], \\
p(u_4) = h(y_4) &= [0, 0, 0, 1], \\
p(u_5) = h(y_5) &= [1, 1, 1, -1],\\
p(u_1 + u_2) = h(y_1 + y_2) &= [1, 1, 0, 0], \\
p(u_2 + u_3) = h(y_2 + y_3) &= [0, 1, 1, 0], \\
p(u_3 + u_4) = h(y_3 + y_4) &= [0, 0, 1, 1], \\
p(u_4 + u_5) = h(y_4 + y_5) &= [1, 1, 1, 0], \\
p(u_5 + u_1) = h(y_5 + y_1) &= [2, 1, 1, -1] \\
\end{split}
\end{equation}
\end{spacing}
\medskip
Consider the line segment $[[1,1,1,0],[1,\frac{1}{2},1,0]]$.
The following inclusions hold:
\begin{equation}
\begin{split}
[[1,1,1,0],[1,\frac{1}{2},1,0]] \subset p(Q), \\
[[1,1,1,0],[1,\frac{1}{2},1,0]] \subset p(Q) + p(u_4 + u_5 - u_1 - u_2)
\end{split}
\end{equation}
Let $t = y_4 + y_5 - y_1 - y_2$. Since $p(Q)\subset h(D^4)$,
\begin{equation}
p(Q) + p(u_4 + u_5 - u_1 - u_2) = 
p(Q) + h(t) \subset h(D^4) + h(t) = h(D^4 + t),
\end{equation}
we have
\begin{equation}
[[1,1,1,0],[1,\frac{1}{2},1,0]] \subset h(D^4) \cap h(D^4 + t).
\end{equation}
By lemma \ref{translate_complex_main_lemma}
on page \pageref{translate_complex_main_lemma}, polytopes
$D^4$ and $D^4 + t$ can be separated
by a hyperplane $N$ so that the following
conditions hold:
\begin{equation}
\label{application-separation-equation}
\begin{split}
D = D^4 \cap (D^4 + t)
= N \cap D^4 = N \cap (D^4 + t), \\
\lin(F^{d-4} - F^{d-4}) \subset N - N, \\
D\text{ is a dual cell}
\end{split}
\end{equation}
We therefore have
\begin{equation}
[[1,1,1,0],[1,\frac{1}{2},1,0]] \subset N
\end{equation} 
It follows that 
\begin{equation}
[1,0,1,0] = h(y_1 + (y_4 + y_5 - y_1 - y_2)) \in N,
\end{equation}
hence $y_1 + (y_4 + y_5 - y_1 - y_2) \in N$.
Point $y_1 + (y_4 + y_5 - y_1 - y_2) = y_4 + y_5 - y_2$ 
is a vertex 
of $D^4 + (y_4 + y_5 - y_1 - y_2)$. By 
equation (\ref{application-separation-equation}),
it is also a vertex of $D^4$. 

Note that $\Vert(\Pi_2) \cup \lbrace y_4 + y_5, y_4 + y_5 - y_2 \rbrace$
is a vertex set of a triangular prism. 

This, by lemma \ref{4-cell-classification-lemma-prism}, implies
that $|\Vert(D^4)| = 8$, however we know that $D^4$ contains 
$10$ vertices of the $5$ parallelograms $\Pi_1$, \dots, 
$\Pi_5$. The contradiction finishes the 5-10 parallelogram 
system analysis.
\bigskip

\noindent{\bf 6-11 system}
The 6-11 hypergraph is shown in figure 
\ref{6-11-enumeration-choice-figure}.

Since each hyperedge is the quartet
of vertices of a parallelogram in ${\cal R}'$,
we can match vertices of each hyperedge
into pairs corresponding to
diagonals of the parallelogram.
Let $S$, $S'$ be the two collections 
of parallelograms, each sharing a common vertex $s$, $s'$. 
The vertex $s$ on each of parallelograms $\Pi \in S$ is matched 
with a vertex $\Pi \cap \Pi'$ for some $\Pi'\in S'$. This 
establishes a mapping $\sigma: S \to S'$ by $\sigma(\Pi) = \Pi'$. 
Similarly we define a mapping $\sigma' :S' \to S$. 
The vertex matching is completely 
defined by the two mappings $\sigma$ 
and $\sigma'$. 
\begin{figure}
\begin{center}
\resizebox{250pt}{!}{\includegraphics[clip=true,keepaspectratio=true]
{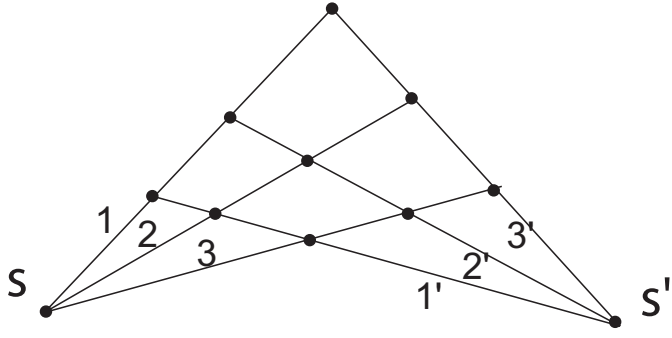}}
\caption{The 6-11 hypergraph}
\label{6-11-enumeration-choice-figure}
\end{center}
\end{figure} 

\begin{lemma}
\label{6-11-matching-lemma}
Parallelograms in collections $S$, $S'$ can be labeled
by symbols $1$-$3$ and $1'$-$3'$ so that the mappings
$\sigma_1$, $\sigma_2$ are: 

{\tiny
\begin{spacing}{1}
\begin{enumerate}
\item 
$
 \left( 
\begin{array}{ccc} 
1  &2  &3  \\ 
1' &1' &1' \\
\end{array} 
\right)
$,
$
 \left( 
\begin{array}{ccc} 
1'  &2'  &3' \\ 
1   &1   &1  \\
\end{array} 
\right)
$,
\item 
$
 \left( 
\begin{array}{ccc} 
1  &2  &3  \\ 
1' &1' &1' \\
\end{array} 
\right)
$,
$
 \left( 
\begin{array}{ccc} 
1'  &2'  &3' \\ 
1   &1   &2  \\
\end{array} 
\right)
$,
\item
$
 \left( 
\begin{array}{ccc} 
1  &2  &3  \\ 
1' &1' &1' \\
\end{array} 
\right)
$,
$
 \left( 
\begin{array}{ccc} 
1'  &2'  &3' \\ 
1   &2   &2  \\
\end{array} 
\right)
$,
\item
$
 \left( 
\begin{array}{ccc} 
1  &2  &3  \\ 
1' &1' &1' \\
\end{array} 
\right)
$,
$
 \left( 
\begin{array}{ccc} 
1'  &2'  &3' \\ 
1   &2   &3  \\
\end{array} 
\right)
$, 
\item
$
 \left( 
\begin{array}{ccc} 
1  &2  &3  \\ 
1' &1' &2' \\
\end{array} 
\right)
$,
$
 \left( 
\begin{array}{ccc} 
1'  &2'  &3' \\ 
1   &1   &2  \\
\end{array} 
\right)
$,
\item
$
 \left( 
\begin{array}{ccc} 
1  &2  &3  \\ 
1' &2' &1' \\
\end{array} 
\right)
$,
$
 \left( 
\begin{array}{ccc} 
1'  &2'  &3' \\ 
1   &1   &2  \\
\end{array} 
\right)
$,
\item
$
 \left( 
\begin{array}{ccc} 
1  &2  &3  \\ 
2' &1' &1' \\
\end{array} 
\right)
$,
$
 \left( 
\begin{array}{ccc} 
1'  &2'  &3' \\ 
1   &1   &2  \\
\end{array} 
\right)
$,
\item
$
 \left( 
\begin{array}{ccc} 
1  &2  &3  \\ 
1' &1' &2' \\
\end{array} 
\right)
$,
$
 \left( 
\begin{array}{ccc} 
1'  &2'  &3' \\ 
1   &2   &1  \\
\end{array} 
\right)
$, reduces to 6  ($S \leftrightarrow S'$),
\item
$
 \left( 
\begin{array}{ccc} 
1  &2  &3  \\ 
1' &2' &1' \\
\end{array} 
\right)
$,
$
 \left( 
\begin{array}{ccc} 
1'  &2'  &3' \\ 
1   &2   &1  \\
\end{array} 
\right)
$,
\item
$
 \left( 
\begin{array}{ccc} 
1  &2  &3  \\ 
2' &1' &1' \\
\end{array} 
\right)
$,
$
 \left( 
\begin{array}{ccc} 
1'  &2'  &3' \\ 
1   &2   &1  \\
\end{array} 
\right)
$,
\item
$
 \left( 
\begin{array}{ccc} 
1  &2  &3  \\ 
1' &1' &2' \\
\end{array} 
\right)
$,
$
 \left( 
\begin{array}{ccc} 
1'  &2'  &3' \\ 
2   &1   &1  \\
\end{array} 
\right)
$, reduces to 7  ($S \leftrightarrow S'$),
\item
$
 \left( 
\begin{array}{ccc} 
1  &2  &3  \\ 
1' &2' &1' \\
\end{array} 
\right)
$,
$
 \left( 
\begin{array}{ccc} 
1'  &2'  &3' \\ 
2   &1   &1  \\
\end{array} 
\right)
$, reduces to 10  ($S \leftrightarrow S'$),
\item
$
 \left( 
\begin{array}{ccc} 
1  &2  &3  \\ 
2' &1' &1' \\
\end{array} 
\right)
$,
$
 \left( 
\begin{array}{ccc} 
1'  &2'  &3' \\ 
2   &1   &1  \\
\end{array} 
\right)
$,
\item 
$
 \left( 
\begin{array}{ccc} 
1  &2  &3  \\ 
1' &1' &2' \\
\end{array} 
\right)
$,
$
 \left( 
\begin{array}{ccc} 
1'  &2'  &3' \\ 
1   &2   &3  \\
\end{array} 
\right)
$, 
\item 
$
 \left( 
\begin{array}{ccc} 
1  &2  &3  \\ 
1' &2' &1' \\
\end{array} 
\right)
$,
$
 \left( 
\begin{array}{ccc} 
1'  &2'  &3' \\ 
1   &2   &3  \\
\end{array} 
\right)
$,
\item 
$
 \left( 
\begin{array}{ccc} 
1  &2  &3  \\ 
2' &1' &1' \\
\end{array} 
\right)
$,
$
 \left( 
\begin{array}{ccc} 
1'  &2'  &3' \\ 
1   &2   &3  \\
\end{array} 
\right)
$,
\item 
$
 \left( 
\begin{array}{ccc} 
1  &2  &3  \\ 
1' &2' &3' \\
\end{array} 
\right)
$,
$
 \left( 
\begin{array}{ccc} 
1'  &2'  &3' \\ 
1   &2   &3  \\
\end{array} 
\right)
$, 
\item 
$
 \left( 
\begin{array}{ccc} 
1  &2  &3  \\ 
1' &2' &3' \\
\end{array} 
\right)
$,
$
 \left( 
\begin{array}{ccc} 
1'  &2'  &3' \\ 
2   &1   &3  \\
\end{array} 
\right)
$.
\end{enumerate}
\end{spacing}}

\end{lemma}

\begin{proof}
The sizes of images of the mappings are independent 
of the labeling of parallelograms. We use them 
to classify the cases we need to consider:
\begin{equation}
\begin{array}{|c|c|c|}
\hline
|\Im(\sigma)| & |\Im(\sigma')| \\
\hline
1 & 1 \\
1 & 2 \\
1 & 3 \\
2 & 2 \\
2 & 3 \\
3 & 3 \\
\hline
\end{array}
\end{equation}
We then use direct inspection.
\end{proof}

The common vertex of parallelograms 
$k$ and $l'$ will be denoted $v_{kl'}$.
Since parallelograms 1 and 2 share exactly one
common vertex, they are complementary and span
a $4$-space.
We choose a basis $e_1,\dots,e_d$ in ${\mathbb R}^d$ so that 
parallelograms 1 and 2 are 
$\conv\lbrace 0,e_1,e_2,e_1+e_2\rbrace$, 
$\conv\lbrace 0,e_3,e_4,e_3+e_4\rbrace$.

We then solve the systems of linear equations for points
$v_{kl'}$ which arise from the combinatorial
information that tells us which vertices
of each hyperedge form a diagonal.
For example, the vertex matching represented
by line 1 in lemma \ref{6-11-matching-lemma}
results in the following system
of linear equations:
\begin{equation}
\begin{split}
s +v_{11'}&=v_{12'}+v_{13'}, \\
s +v_{21'}&=v_{22'}+v_{23'}, \\
s +v_{31'}&=v_{32'}+v_{33'}, \\
s'+v_{11'}&=v_{21'}+v_{31'}, \\
s'+v_{12'}&=v_{22'}+v_{32'}, \\
s'+v_{13'}&=v_{23'}+v_{33'}
\end{split}
\end{equation}

The solutions are presented in the tables \ref{6-11-systems-table},
\ref{6-11-systems-table-cont}, and \ref{6-11-systems-table-cont-cont} 
below, in the format 
$[s,v_{11'},v_{12'},v_{13'},v_{21'},v_{22'},v_{23'},v_{31'},v_{32'},v_{33'}, s']$.
The coordinates are in the basis $e_1,\dots,e_d$.
Column ``reason for contradiction'' shows why
each outcome is impossible.

\begin{spacing}{1}
\begin{table}
\label{6-11-systems-table}
{\tiny
\begin{equation}
\begin{array}{|l|l|l|}
\hline
\text{Case} & \text{Solution} & \text{Reason for contradiction} \\
\hline
1  & 
 \left[ \begin {array}{ccccccccccc} 0&1&1&0&0&0&0&1&1&0&0
\\\noalign{\medskip}0&1&0&1&0&0&0&1&0&1&0\\\noalign{\medskip}0&0&0&0&1
&1&0&-1&-1&0&0\\\noalign{\medskip}0&0&0&0&1&0&1&-1&0&-1&0\\\noalign{\medskip}0\dots\end {array}
 \right] &
s = s' 
\smallskip \\
\hline \\

2  &
\left[ \begin {array}{ccccccccccc} 0&1&1&0&0&0&0&1&1&0&0
\\\noalign{\medskip}0&1&0&1&0&0&0&3&2&1&2\\\noalign{\medskip}0&0&0&0&1
&1&0&-1&-1&0&0\\\noalign{\medskip}0&0&0&0&1&0&1&-3&-2&-1&-2\\\noalign{\medskip}0\dots
\end {array} \right] &
s' \equiv s \pmod{2\Lambda} 
\smallskip \\
\hline \\

3 &
 \left[ \begin {array}{ccccccccccc} 0&1&1&0&0&0&0&3&1&2&2
\\\noalign{\medskip}0&1&0&1&0&0&0&3&2&1&2\\\noalign{\medskip}0&0&0&0&1
&1&0&-3&-1&-2&-2\\\noalign{\medskip}0&0&0&0&1&0&1&-3&-2&-1&-2\\\noalign{\medskip}0\dots
\end {array} \right] &
s' \equiv s \pmod{2\Lambda} \\
\smallskip \\
\hline \\

4 &
 \left[ \begin {array}{ccccccccccc} 0&1&1&0&0&0&0&-1&-3&2&-2
\\\noalign{\medskip}0&1&0&1&0&0&0&1&0&1&0\\\noalign{\medskip}0&0&0&0&1
&1&0&1&3&-2&2\\\noalign{\medskip}0&0&0&0&1&0&1&1&2&-1&2\\\noalign{\medskip}0\dots
\end {array}
 \right] &
s' \equiv s \pmod{2\Lambda} \\
\smallskip \\
\hline \\

5 &
\left[ \begin {array}{ccccccccccc} 0&1&1&0&0&0&0&1&1&0&0
\\\noalign{\medskip}0&1&0&1&0&0&0&1&0&-1&0\\\noalign{\medskip}0&0&0&0&
1&1&0&-1&-1&0&0\\\noalign{\medskip}0&0&0&0&1&0&1&-1&0&1&0\\\noalign{\medskip}0\dots
\end {array}
 \right] &
s' \equiv s \pmod{2\Lambda} \\
\smallskip \\
\hline \\

6 &
 \left[ \begin {array}{ccccccccccc} 0&1&1&0&0&0&0&1&1&0&0
\\\noalign{\medskip}0&1&0&1&0&0&0&3&2&1&2\\\noalign{\medskip}0&0&0&0&1
&1&0&-1&-1&0&0\\\noalign{\medskip}0&0&0&0&0&1&1&0&-1&1&0\\\noalign{\medskip}0\dots
\end {array}
 \right] &
s' \equiv s \pmod{2\Lambda} \\
\smallskip \\
\hline \\

7 &
 \left[ \begin {array}{ccccccccccc} 0&1&1&0&0&0&0&1&1&0&0
\\\noalign{\medskip}0&0&1&1&0&0&0&0&1&-1&0\\\noalign{\medskip}0&0&0&0&
1&1&0&-1&-1&0&0\\\noalign{\medskip}0&0&0&0&1&0&1&-3&-2&-1&-2\\\noalign{\medskip}0\dots
\end {array} \right] &
s' \equiv s \pmod{2\Lambda} \\
\smallskip \\
\hline \\ 
8 & \text{Reduces to 6} & \\
\hline
\end{array}
\end{equation}
}
\caption{All parallelogram systems with 6-11 hypergraph, up to affine equivalence}
\end{table}

\begin{table}
\label{6-11-systems-table-cont}
{\tiny
\begin{equation}
\begin{array}{|l|l|l|}
\hline
\text{Case} & \text{Solution} & \text{Reason for contradiction} \\
\hline

9 &
 \left[ \begin {array}{ccccccccccc} 0&1&1&0&0&0&0&3&1&2&2
\\\noalign{\medskip}0&1&0&1&0&0&0&1&0&1&0\\\noalign{\medskip}0&0&0&0&1
&1&0&-3&-1&-2&-2\\\noalign{\medskip}0&0&0&0&0&1&1&0&1&-1&0\\\noalign{\medskip}0\dots
\end {array}
 \right] &
s' \equiv s \pmod{2\Lambda} \\
\smallskip \\
\hline \\

10 &
 \left[ \begin {array}{ccccccccccc} 0&1&1&0&0&0&0&3&1&2&2
\\\noalign{\medskip}0&0&1&1&0&0&0&0&-1&1&0\\\noalign{\medskip}0&0&0&0&
1&1&0&-3&-1&-2&-2\\\noalign{\medskip}0&0&0&0&1&0&1&-1&0&-1&0\\\noalign{\medskip}0\dots
\end {array} \right] &
s' \equiv s \pmod{2\Lambda} \\
\smallskip \\
\hline \\

11 & \text{Reduces to 7} & \\

\hline \\

12 & \text{Reduces to 10} & \\

\hline \\

13 &
 \left[ \begin {array}{ccccccccccc} 0&1&1&0&0&0&0&-3&-1&-2&-2
\\\noalign{\medskip}0&0&1&1&0&0&0&-2&-1&-1&-2\\\noalign{\medskip}0&0&0
&0&1&1&0&3&1&2&2\\\noalign{\medskip}0&0&0&0&1&0&1&3&2&1&2\\\noalign{\medskip}0\dots
\end{array}
 \right] &
s' \equiv s \pmod{2\Lambda} \\
\smallskip \\
\hline \\

14 &
 \left[ \begin {array}{ccccccccccc} 0&1&1&0&0&0&0&3&1&-2&2
\\\noalign{\medskip}0&1&0&1&0&0&0&3&2&-1&2\\\noalign{\medskip}0&0&0&0&
1&1&0&-3&-1&2&-2\\\noalign{\medskip}0&0&0&0&1&0&1&-1&0&1&0\\\noalign{\medskip}0\dots
\end{array}
 \right] &
 s' \equiv s \pmod{2\Lambda} \\
\smallskip \\
\hline \\

15 &
 \left[ \begin {array}{ccccccccccc} 0&1&1&0&0&0&0&-1&-3&2&-2
\\\noalign{\medskip}0&1&0&1&0&0&0&1&0&1&0\\\noalign{\medskip}0&0&0&0&1
&1&0&1&3&-2&2\\\noalign{\medskip}0&0&0&0&0&1&1&2&3&-1&2\\\noalign{\medskip}0\dots
\end{array}
 \right] &
s' \equiv s \pmod{2\Lambda} \\
\smallskip \\
\hline \\

16 & 
 \left[ \begin {array}{ccccccccccc} 0&1&1&0&0&0&0&-1&-3&2&-2
\\\noalign{\medskip}0&0&1&1&0&0&0&0&-1&1&0\\\noalign{\medskip}0&0&0&0&
1&1&0&1&3&-2&2\\\noalign{\medskip}0&0&0&0&1&0&1&1&2&-1&2\\\noalign{\medskip}0\dots
\end{array}
 \right] &
s' \equiv s \pmod{2\Lambda} \\
\smallskip \\
\hline \\
\end{array}
\end{equation}
}
\caption{All parallelogram systems with 6-11 hypergraph, up to affine equivalence (continued)}
\end{table}

\begin{table}
\label{6-11-systems-table-cont-cont}
{\tiny
\begin{equation}
\begin{array}{|l|l|l|}
\hline
\text{Case} & \text{Solution} & \text{Reason for contradiction} \\
\hline
17 &
 \left[ \begin {array}{ccccccccccc} 0&1&1&0&0&0&0&1&-1&0&0
\\\noalign{\medskip}0&1&0&1&0&0&0&1&0&1&0\\\noalign{\medskip}0&0&0&0&1
&1&0&-1&1&0&0\\\noalign{\medskip}0&0&0&0&0&1&1&0&1&1&0\\\noalign{\medskip}0\dots
\end{array}
 \right] &
s' \equiv s \pmod{2\Lambda} \\
\smallskip \\
\hline \\

18 &
 \left[ \begin {array}{ccccccccccc} 0&1&1&0&0&0&0&-1&1&0&0
\\\noalign{\medskip}0&1&0&1&0&0&0&-1/3&2/3&1/3&2/3\\\noalign{\medskip}0
&0&0&0&1&1&0&1&-1&0&0\\\noalign{\medskip}0&0&0&0&0&1&1&2/3&-1/3&1/3&2/\\\noalign{\medskip}0\dots
\end {array} \right] &
\text{not in a convex position} \\

\hline
\end{array}
\end{equation}
}
\caption{All parallelogram systems with 6-11 hypergraph, up to affine equivalence (continued)}
\end{table}
\end{spacing}

\bigskip
In each case we see that either some of the 11 vertices of the
parallelogram system coincide, or are equivalent to each other
modulo $2 \Lambda$ where $\Lambda$ is the lattice of the tiling.
This completes the analysis of the 6-11 system of parallelograms,
and finishes the proof of the theorem.
\end{subsection}
\end{section}

\begin{section}{Conclusion}

The value of many hard problems is often not in the problem
itself, but in the theories and side results that come 
out of it. From Voronoi's conjecture on parallelotopes, 
originated the idea of quality translation, a number
of results on zonotopes, and generally a better understanding of 
the geometry of polytopes. 

We now look back at the results we have obtained,
discuss ways to resolve the Voronoi
conjecture, and pose some questions which we couldn't answer.

The idea of ``coherent'' parallelogram dual cells was
very effective in proving the Voronoi conjecture for
$3$-irreducible tilings. The existence of
incoherent parallelograms would disprove the Voronoi
conjecture on parallelotopes. The most important
result was that incoherent parallelograms in a dual
$4$-cell can only come in groups: a vertex of an 
incoherent parallelogram must belong to another
incoherent parallelogram.

Our definition of an incoherent parallelogram was relative
to a dual $4$-cell. Robert Erdahl suggested another definition.
Consider a parallelogram dual cell $\Pi$. Given
any two facets $F_1^{d-1}$, $F_2^{d-1}$ which correspond 
to the edges of the parallelogram, it follows from 
our theorem \ref{irreducibility_criterion_theorem} that
there are combinatorial paths on the $(d-1)$-skeleton
of the tiling which connect $F_1^{d-1}$, $F_2^{d-1}$ and
where the $(d-2)$-joints are all hexagonal. This path 
can be used to transfer the scale factor 
from $F_1^{d-1}$ to $F_2^{d-1}$. In the new definition,
the parallelogram $\Pi$ is called {\em coherent} if 
all such combinatorial paths lead to the same scale factor
on $F_2^{d-1}$, and if the two facets are
translates of each other, then the scale factor
on $F_2^{d-1}$ is the same as on $F_1^{d-1}$.

Existence of incoherent parallelograms in the new sense
would contradict the Voronoi conjecture. Their nonexistence,
however, will prove the Voronoi conjecture. 

\medskip

\noindent C0. {\em Can we transfer our results 
on incoherent parallelograms for the case 
of the new definition?}

\medskip

Below is a list of other statements that we are very interested in,
but couldn't prove.

\medskip

\noindent C1. {\em The dimension of a dual cell and the dimension 
of the corresponding face of the tiling sum to $d$.} We have
neither lower nor upper bound. 

\medskip

\noindent C2. {\em Dual cells form a polyhedral complex.}
So far we could only prove that dual cells cover the space.
The proof can be found in the paper proposal document,
which can be downloaded from the author's web site
{\em http://www.mast.queensu.ca/\ $\tilde\ $ordine}.

\medskip

\noindent C3. {\em A 'local' version of the Venkov graph criterion.}
Let $F^{d-k}$ be a face of the tiling. Facets in the star
of $F^{d-k}$ define a subgraph $V_{F^{d-k}}(P)$ of the Venkov graph
to the parallelotope. The following statements are
equivalent:
\begin{enumerate}
\item $V_{F^{d-k}}(P)$ is connected by red edges
\item The dual cell corresponding to $F^{d-k}$ is irreducible
(ie. it cannot be represented as a direct Minkowski sum
of polytopes of smaller positive dimensions)
\item The tiling is locally reducible at $F^{d-k}$.
\end{enumerate}

\medskip

\noindent C4. {\em Equivalence of two versions of tiling 
irreducibility.} The parallelotope of a $d$-reducible tiling 
can be represented as a direct Minkowski sum of parallelotopes
of smaller positive dimensions.

\medskip

\noindent C5. {\em A regularity condition for dual cells.}
Consider a dual $4$-cell. We do not know if the dimension
of a dual $4$-cell (as a polytope) is actually $4$. However, 
we know the full classification of dual $3$-cells. Therefore
the boundary of a dual $4$-cell can be considered as an
immersion of a $3$-dimensional polyhedral complex in the
$d$-dimensional space. What are properties of the immersion?
Is it an embedding?

\medskip

We think that a proof (or a counterexample) 
to the Voronoi conjecture can be obtained by considering
increasingly large classes of tilings: $2$-irreducible, $3$-irreducible,
$4$-irreducible and so on, and trying to devise an induction
scheme. We conjecture that a $d$-irreducible tiling
of the $d$-dimensional space has a reducible parallelotope
$P = P_1 \oplus P_2$. Then it would be sufficient to
solve the conjecture for parallelotopes $P_1$, $P_2$
of smaller dimensions. 
\end{section}

\bibliography{affeq}

\printindex

%\end{doublespace}
\end{document}